\documentclass[oneside,11pt]{article}

\usepackage[pdftex,final,bookmarksnumbered,breaklinks,colorlinks,unicode]{hyperref}
\hypersetup{
	final,
	bookmarksnumbered=true,
	bookmarksopen=true,
	bookmarksopenlevel=0,
	unicode=false,
	pdftoolbar=true,
	pdfmenubar=true,
	pdffitwindow=false,
	pdftitle={Robust Ergodic Control of Jump-Diffusion Systems \break under Drift and Intensity Uncertainty},
	pdfdisplaydoctitle=true,
	pdftoolbar=true,
	pdfmenubar=true,
	pdflang={English},
	pdfauthor={Abel Azze, Bernardo D'Auria, Giorgio Ferrari},
	pdfsubject={arXiv paper},
	pdfcreator={Abel Azze},
	pdfproducer={Abel Azze},
	pdfkeywords={stochastic control}{Lévy process},
	pdfnewwindow=true,
	breaklinks=true,
	linkcolor=black,
	citecolor=black,
	filecolor=magenta,
	urlcolor=cyan
}

\usepackage{amsmath,amssymb,amsthm}
\usepackage{bbm}                  % for \mathbbm{1}, used through \Ind
\usepackage{empheq}               % used by empheq environments

\usepackage{graphicx}
\usepackage{xcolor}               % used by \details and draft color macros
\usepackage[labelformat=simple]{subcaption}

\usepackage{ifthen}
\newboolean{details}
\newcommand{\details}[2]{%
  \ifthenelse{\boolean{details}}%
    {\textcolor{purple}{#1}}%
    {#2}%
}

\usepackage{algorithm}
\usepackage{algpseudocode}

\usepackage{booktabs}
\usepackage{enumitem}

\usepackage[top=1.5cm,bottom=2cm,right=1cm,left=1cm]{geometry}
\usepackage{titling}

\usepackage{natbib}

\usepackage{xparse}

\allowdisplaybreaks

\newtheorem{theorem}{Theorem}
\newtheorem{proposition}{Proposition}
\newtheorem{lemma}{Lemma}
\newtheorem{corollary}{Corollary}
\newtheorem{remark}{Remark}
\newtheorem{assumption}{Assumption}
\renewenvironment{proof}{\textit{Proof.}}{\qed\\}

\newcommand{\R}{\mathbb{R}}
\newcommand{\cA}{\mathcal{A}}

\newcommand{\lp}{\left(}   \newcommand{\rp}{\right)}
\newcommand{\lb}{\left[}   \newcommand{\rb}{\right]}
\newcommand{\lrp}[1]{\lp#1\rp}

\newcommand*{\defeq}{\mathrel{\vcenter{
                     \baselineskip0.5ex\lineskiplimit0pt
                     \hbox{\scriptsize.}\hbox{\scriptsize.}
                     }}%
                     =}

\RenewDocumentCommand{\Pr}{oe{_^}}{%
  \operatorname{\mathbb{P}}%
  \IfValueT{#2}{_{#2}}%
  ^{\IfValueTF{#3}{#3}{}}%
  \IfValueT{#1}{\lp #1 \rp}%
}

\NewDocumentCommand{\Esp}{oe{_^}}{%
  \operatorname{\mathbb{E}}%
  \IfValueT{#2}{_{#2}}%
  ^{\IfValueTF{#3}{#3}{}}%
  \IfValueT{#1}{\lb #1 \rb}%
}

\newcommand{\intOp}{\mathbb{I}}
\newcommand{\InfGen}{\mathbb{L}}
\newcommand{\Ind}{\mathbbm{1}}
\newcommand{\rmd}{\mathrm{d}}
\newcommand{\uline}[1]{\underline{#1}}
\newcommand{\oline}[1]{\overline{#1}}
\newcommand{\altBER}[1]{\begingroup\color{black}#1\endgroup}
\newcommand{\BER}[1]{\begingroup\color{violet}#1\endgroup}
\newenvironment{ber}{\color{violet}}{} 
\newcommand{\bb}{\begin{ber}
		\let\BER\altBER}
	\newcommand{\eb}{\end{ber}}

\setboolean{details}{false}

\numberwithin{equation}{section}

\begin{document}
	
	%-----------------------------------------------%
	\title{Robust Ergodic Control of Jump-Diffusion Systems \break under Drift and Intensity Uncertainty}
	\setlength{\droptitle}{-1cm}
	\predate{}%
	\postdate{}%
	\date{}
	%-----------------------------------------------%
	
	%-----------------------------------------------%
	\author{Abel Azze$^{1}$, Bernardo D'Auria$^{2}$, Giorgio Ferrari$^{3}$}
	\footnotetext[1]{Department of Quantitative Methods, CUNEF Universidad (Spain). Corresponding author. Email: abel.guada@cunef.edu}
	\footnotetext[2]{Department of Mathematics ``Tullio Levi Civita'', University of Padova (Italy). Email: bernardo.dauria@unipd.it}
	\footnotetext[3]{Center for Mathematical Economics, Bielefeld University (Germany). Email: giorgio.ferrari@uni-bielefeld.de}
	\maketitle
	%-----------------------------------------------%
	
	\begin{abstract}
		We study a regulation problem for stochastic systems subject to both continuous fluctuations and rare but significant shocks, modeled as a jump-diffusion with uncertainty in both the drift and the jump intensity. Such settings arise in applications including inventory control, cash management, and capacity planning.
		
		We formulate the problem as a robust ergodic singular control problem in which a decision maker applies upward and downward interventions while accounting for model ambiguity through entropy-penalized distortions. The resulting max-min problem involves a long-run average performance criterion. 
		
		We show that the associated Hamilton--Jacobi--Bellman equation reduces to a nonlinear integro-differential free-boundary problem with a tractable structure. The worst-case model exhibits a bang-bang form, and the optimal policy is characterized by reflecting barriers. Under exponentially distributed jumps, the problem further reduces to a system of ordinary differential equations, enabling efficient numerical computation.
		
		% Our numerical analysis quantifies the impact of model ambiguity on optimal decisions. Accounting for uncertainty in both drift and jump intensity significantly alters the inaction region, leading to wider and asymmetric intervention bands. Ignoring ambiguity can result in significant performance losses, and the impact differs across the two components: drift ambiguity and jump-intensity ambiguity interact, with the dominant source of misspecification depending on the underlying drift of the uncontrolled system. These findings highlight the importance of robust regulation in environments with poorly estimated dynamics and rare events.
		
		Our numerical analysis shows that ambiguity in drift and jump intensity can significantly alter the inaction region, leading to wider and asymmetric intervention bands. Ignoring ambiguity can generate substantial performance losses, with the dominant source of misspecification depending on the underlying drift. These findings highlight the importance of robust regulation in environments with poorly estimated dynamics and rare events.
		
	\end{abstract}
	
	\begin{flushleft}
		\small \textbf{Keywords:} robust control, singular stochastic control, jump-diffusion, ergodic control, model uncertainty, free-boundary problems
	\end{flushleft}
	
	\begin{flushleft}
		\small\textbf{OR/MS Classification:} Dynamic programming/optimal control, Stochastic models, Decision analysis: uncertainty, Robust optimization, Continuous-time models
	\end{flushleft}
	
	\begin{flushleft}
		\small\textbf{AMS 2020 Subject Classification:} 93E20, 49L20, 60J75, 90C40
	\end{flushleft}
	
	%-------------------------------------------------%
	\section{Introduction}
	%-------------------------------------------------%
	
	Many operational and economic systems require continuous regulation over long time horizons in the presence of uncertainty and occasional large shocks. 
	Typical examples include inventory and cash management, capacity adjustment, dividend distribution, and harvesting of renewable resources. 
	In these settings, the state of the system evolves stochastically over time, while the decision maker can intervene through upward or downward 
	adjustments that incur proportional costs. The objective is to keep the system within desirable operating levels while minimizing long-run 
	average costs. Problems of this type arise naturally in operations management, finance, and economics, where decisions must be taken repeatedly 
	and the system is expected to operate indefinitely.
	
	Such regulation problems are naturally formulated as singular stochastic control problems, in which the control is a bounded-variation process 
	allowing for instantaneous adjustments of arbitrary size. Since the early work of \citet{Bather_1967_sequential}, singular control has been 
	widely used to model regulation problems in operations management, finance, and economics, including inventory control, cash and dividend 
	management, capacity expansion, equilibrium models, optimal harvesting of renewable resources, and queuing systems; see, among many others, 
	\cite{aid2025stationary, alvarez_optimal_1998, bandini_optimal_2022,cohen2019asymptotic,HarrisonTaksar1983, yamazaki_inventory_2017, merhi2007model, jeanblanc1995optimization, rodosthenous2023two}, as well as \cite{kumar2004numerical} for the numerical analysis of such a class of problems.
	In many applications the relevant performance criterion is the long-run average payoff, which leads to ergodic singular control problems; see, e.g., 
	\cite{menaldi1992singular,taksar1985average,liang2025singular,calvia2025optimal} and references therein. These problems characterize optimal stationary 
	regulation policies and are particularly suitable for systems operating over long horizons, where transient effects become negligible and performance 
	is measured through steady-state behavior.
	
	In many realistic applications, the dynamics of the controlled system are not fully known. Besides continuous fluctuations, the state may be subject 
	to sudden shocks, such as abrupt demand changes, unexpected cash inflows or outflows, catastrophic events in population dynamics, or discontinuous 
	movements in asset prices. Jump-diffusion models provide a natural framework to capture these features, as they combine diffusive uncertainty with 
	rare but potentially large jumps. In practice, however, both the drift and the jump intensity are typically estimated from limited data and may 
	therefore be misspecified. In particular, the frequency of rare events is often difficult to assess, for instance in inventory or production systems 
	with irregular demand surges, in cash-management problems with unpredictable large transactions, or in energy and resource systems subject to failures 
	or environmental shocks. Ignoring such uncertainty may lead to regulation policies that perform poorly over long time horizons and can result in substantial performance losses when rare events are misestimated. This motivates the 
	study of regulation problems in which the controller accounts for ambiguity not only in the drift but also in the intensity of the jump component.
	
	These considerations naturally motivate the study of robust regulation policies under model misspecification. 
	In this paper we consider a long-run singular control problem for a jump-diffusion system in which the decision maker faces ambiguity about both the drift of the diffusion component and the intensity of the jump process. 
	Model ambiguity is incorporated through entropy-type penalties, which allow for adversarial distortions of the reference model while penalizing deviations from it, leading to a robust ergodic singular control problem in which the controller optimizes the long-run performance against the worst-case model within a prescribed ambiguity set.
	
	The worst-case long-run average payoff criterion leads to a nonlinear control problem whose Hamilton--Jacobi--Bellman equation takes the form of an integro-differential free-boundary problem. 
	The combination of singular controls and model uncertainty makes the analysis significantly more involved than in the classical diffusion setting.
	
	Robust and ambiguity-averse stochastic control has received increasing attention in operations research, economics, and finance. 
	Entropy-penalized formulations and worst-case control problems have been studied, among others, by 
	\cite{hansen2001robust,iyengar2005robust,nilim2005robust}. These models lead to max-min control problems and nonlinear Hamilton--Jacobi--Bellman 
	equations, but are usually developed either for diffusion models or for regular control, and only rarely in the presence of singular controls and 
	long-run average criteria. The few recent papers dealing with singular control under drift ambiguity, in a diffusive infinite-horizon discounted 
	setting, include \cite{ferrari2022knightian,ferrari2022optimal,archankul2025singular}.
	
	Despite these developments, the combination of ergodic criteria, singular controls, jump-diffusion dynamics, and ambiguity in both drift and jump intensity appears to have received little attention in the literature. In particular, robust ergodic singular control problems for jump-diffusion systems lead to integro-differential free-boundary problems whose structure differs substantially from the classical diffusion case and requires new analytical arguments.
	
	The main contribution of this paper is to provide a tractable and operationally relevant analysis of a class of robust ergodic singular control problems for jump-diffusion systems with ambiguity in both the drift and the jump intensity. We show that the associated max-min Hamilton--Jacobi--Bellman equation can be reduced to a nonlinear integro-differential free-boundary problem in which the worst-case parameters exhibit a bang-bang structure. This characterization leads to a simple and interpretable optimal policy in the form of a reflecting-barrier strategy and allows us to establish a verification theorem for the ergodic value.
	
	From a computational and operational perspective, we further show that under negative exponentially distributed jumps, constant drift, and quadratic holding costs, the problem reduces to a system of linear ordinary differential equations. This yields a numerically stable and efficient procedure for computing optimal policies and worst-case models, making the framework suitable for practical implementation in operational settings.
	
	Our results demonstrate that robustness considerations typically lead to more cautious regulation rules, but they do not merely expand the inaction region. Instead, they introduce significant asymmetry: while drift ambiguity primarily shifts the location of the reflecting band, intensity uncertainty (reflecting uncertainty about rare, high-impact shocks) tends to widen the cautious zone significantly. Furthermore, variations in intervention costs tilt the band toward the side where adjustments are cheaper. These findings provide clear operational implications for managers in inventory and cash-flow systems, indicating that policies should be adjusted not only in magnitude but also asymmetrically when the frequency of extreme events is uncertain. Moreover, our results highlight that neglecting model ambiguity can lead to systematically suboptimal policies. We quantify the resulting performance loss in long-run average cost and show that the value of robustness depends on the interaction between the baseline drift and the source of ambiguity: in the numerical examples, ignoring ambiguity leads to substantially larger performance losses when the baseline drift is positive, and these losses are especially sensitive to uncertainty in the jump intensity; by contrast, when the baseline drift is negative, the losses are smaller and are more strongly driven by drift ambiguity.
	
	The remainder of the paper is organized as follows. Section~\ref{sec:the_model} introduces the controlled jump-diffusion model under ambiguity. 
	Section~\ref{sec:control_problem} formulates the robust ergodic singular control problem. 
	Section~\ref{sec:hjb_verification} derives the Hamilton--Jacobi--Bellman equation and provides a verification theorem. 
	Section~\ref{sec:optimalpolicy} studies the structure of the optimal policy in the case of negative jumps. 
	Section~\ref{sec:negative_exp_jumps} develops the explicit characterization for exponential jumps. 
	Section~\ref{sec:numerics} presents numerical results and economic interpretations.
	
	%-------------------------------------------------%
	\section{Controlled dynamics under ambiguity}\label{sec:the_model}
	%-------------------------------------------------%
	
	We consider a one-dimensional stochastic system whose state evolves over time under the combined effect of
	continuous fluctuations and random jumps. Such dynamics arise naturally in applications including inventory levels,
	cash balances, production systems, and regulated economic processes subject to sudden shocks.
	
	We model the system as a compound-Poisson jump-diffusion $X=(X_t)_{t\ge0}$ defined on a complete filtered probability space
	$(\Omega,\mathcal F,(\mathcal F_t)_{t\ge0},\Pr^0)$, and given as the unique solution (see, e.g., Theorem 1.19 in \cite{OksendalSulem2019})
	of the stochastic differential equation
	\begin{align}\label{eq:benchmark_model}
		\rmd X_t
		= b(X_{t-})\, \rmd t
		+ \sigma\, \rmd B_t
		+ \rmd \widetilde Z_t,
	\end{align}
	where $b:\R\to\R$ is a Lipschitz drift function and $\sigma>0$ is constant.
	Under $\Pr^0$, $B=(B_t)_{t\ge0}$ is a standard Brownian motion and
	$Z=(\sum_{i=1}^{N_t}Y_i)_{t\ge0}$ is a compound-Poisson process,
	where $N=(N_t)_{t\ge0}$ is a Poisson counting process with intensity $r>0$ and the jump sizes
	$(Y_i)_{i\ge1}$ are i.i.d.\ with distribution $F$ satisfying $\Esp[|Y|]<\infty$.
	We denote by $\widetilde Z=(\widetilde Z_t)_{t\ge0}$ the compensated jump process, where
	\[
	\widetilde Z_t = Z_t - r\Esp[Y]t .
	\]
	The processes $B$, $N$, and the sequence $(Y_i)_{i\ge1}$ are assumed independent, and
	$(\mathcal F_t)_{t\ge0}$ is the completed filtration generated by $B$ and $\widetilde Z$.
	
	\medskip
	
	The system can be regulated through upward and downward adjustments.
	Given two nondecreasing adapted processes $U=(U_t)_{t\ge0}$ and $D=(D_t)_{t\ge0}$,
	representing the cumulative upward and downward interventions, the controlled state
	$X^{U,D}=(X_t^{U,D})_{t\ge0}$ evolves as
	\begin{align}\label{eq:ambiguous_controlled_model}
		\rmd X_t^{U,D}
		= b(X_{t-}^{U,D})\, \rmd t
		+ \sigma\, \rmd B_t
		+ \rmd \widetilde Z_t
		+ \rmd U_t
		- \rmd D_t.
	\end{align}
	
	Admissible controls are pairs $(U,D)$ belonging to
	\[
	\underline{\mathcal A}
	=
	\left\{
	(U,D)\ \middle|\
	\begin{aligned}
		&U,D \text{ are nondecreasing, càdlàg, adapted processes}\\
		&\text{of bounded variation with } U_{0-}=D_{0-}=0
	\end{aligned}
	\right\}.
	\]
	
	\medskip
	
	To capture misspecification in the drift and in the jump intensity,
	we consider a family of alternative probability measures obtained through
	Girsanov transformations of both the Brownian and Poisson components.
	For predictable processes
	\[
	\kappa_t=\kappa(X_{t-}^{U,D}),
	\qquad
	\lambda_t=\lambda(X_{t-}^{U,D}),
	\]
	we define probability measures $\Pr^{\kappa,\lambda}$ locally absolutely continuous
	with respect to $\Pr^0$ through the Doléans-Dade exponential
	\begin{align}\label{eq:RN-density}
		\frac{\rmd\Pr^{\kappa,\lambda}}{\rmd\Pr^0}\Big|_{\mathcal F_t}
		=
		\mathcal E_t
		=
		\exp\!\left(
		\int_0^t \kappa_s\,\rmd B_s
		-\frac12\int_0^t \kappa_s^2\,\rmd s
		\right)
		\exp\!\left(
		\int_0^t \ln\!\frac{\lambda_s}{r}\,\rmd N_s
		-\int_0^t (\lambda_s-r)\,\rmd s
		\right).
	\end{align}
	
	The processes $\kappa$ and $\lambda$ represent distortions of the benchmark drift
	and Poisson intensity and are assumed to satisfy the bounds
	\begin{align}\label{eq:ambiguity_bounds}
		|\kappa(x)| \le \delta,
		\qquad
		(1-\varepsilon)r \le \lambda(x) \le r(1+\varepsilon),
	\end{align}
	for all $x\in\R$, where $\delta\ge0$ and $0\le\varepsilon<1$.
	
	Under these conditions, Novikov's criterion for Lévy processes
	(see Theorem 1.36 in \cite{OksendalSulem2019})
	ensures that $\mathcal E_t$ is a martingale with expectation one,
	so that $\Pr^{\kappa,\lambda}$ is well defined.
	We denote by $\Lambda$ the set of admissible distortion pairs $(\kappa,\lambda)$.
	To keep notation simple, we do not explicitly indicate the dependence of
	$\Pr^{\kappa,\lambda}$ on the control $(U,D)$.
	
	\medskip
	
	By Girsanov's theorem for jump-diffusions (see Theorem 1.33 in \cite{OksendalSulem2019}),
	under $\Pr^{\kappa,\lambda}$ the controlled state satisfies
	\begin{align}\label{eq:ambiguous-controlled_model}
		\rmd X_t^{U,D}
		&=
		\bigl(
		b(X_{t-}^{U,D})
		+\sigma\kappa(X_{t-}^{U,D})
		+(\lambda(X_{t-}^{U,D})-r)\Esp[Y]
		\bigr)\rmd t
		+\sigma\,\rmd B_t^\kappa
		+\rmd \widetilde Z_t^\lambda
		+\rmd U_t
		-\rmd D_t,
		%\nonumber
	\end{align}
	where $B^\kappa$ is a standard Brownian motion under $\Pr^{\kappa,\lambda}$ and
	\[
	\widetilde Z_t^\lambda
	=
	\sum_{i=1}^{N^\lambda_t}Y_i
	-
	\int_0^t \lambda(X_{s-}^{U,D})\Esp[Y]\,\rmd s,
	\]
	with $N^\lambda$ having intensity $\lambda(X_{t-}^{U,D})$.
	
	We write $\Pr_x^{\kappa,\lambda}$ for the conditional measure with
	$\Pr_x^{\kappa,\lambda}(X_0^{U,D}=x)=1$ and denote the corresponding
	expectation by $\Esp_x^{\kappa,\lambda}$.
	
	\medskip
	
	For a twice differentiable function $f$ with suitable growth conditions,
	the infinitesimal generator of the uncontrolled process under
	$\Pr^{\kappa,\lambda}$ is
	\begin{align}
		(\InfGen^{\kappa,\lambda}f)(x)
		&=
		(b(x)+\sigma\kappa(x)+(\lambda(x)-r)\Esp[Y])f'(x)
		+\frac12\sigma^2 f''(x)
		\nonumber\\
		&\quad
		+\lambda(x)
		\int_\R
		\bigl(
		f(x+y)-f(x)-y f'(x)
		\bigr)\rmd F(y)
		\nonumber\\
		&=
		(b(x)+\sigma\kappa(x)-r\Esp[Y])f'(x)
		+\frac12\sigma^2 f''(x)
		+\lambda(x)(\Psi f)(x),
		\label{eq:inf_gen}
	\end{align}
	where
	\begin{align}\label{eq:int_op}
		(\Psi f)(x)
		=
		\int_\R f(x+y)\,\rmd F(y)
		-
		f(x).
	\end{align}
	
	\begin{remark}
		The change of measure \eqref{eq:RN-density} corresponds to a relative-entropy
		penalization of the drift and intensity distortions.
		Indeed, the Kullback-Leibler divergence of $\Pr^{\kappa,\lambda}$
		with respect to $\Pr^0$ on $\mathcal F_t$ is
		\[
		D_{\mathrm{KL}}^t(\Pr^{\kappa,\lambda}\mid\Pr^0)
		=
		\frac12 \Esp[\int_0^t \kappa_s^2\,ds]^{\kappa,\lambda}
		+
		\Esp[
		\int_0^t
		\left(
		\lambda_s\ln\frac{\lambda_s}{r}
		-\lambda_s+r
		\right)ds
		]^{\kappa,\lambda}.
		\]
		Under the bounds \eqref{eq:ambiguity_bounds}, this divergence grows at most
		linearly in time, which justifies the interpretation of $\kappa$ and $\lambda$
		as entropy-penalized model distortions.
	\end{remark}
	
	%-------------------------------------------------%
	\section{Formulation of the control problem, HJB equation, and verification theorem}
	%-------------------------------------------------%
	
	%-------------------------------------------------%
	\subsection{Formulation of the control problem}
	\label{sec:control_problem}
	%-------------------------------------------------%
	
	For the measurable cost function $c:\R\rightarrow\R$ such that 
	\begin{align*}
		\Esp[\int_0^t |c(X_s^{U,D})|\,\rmd s]_{x}^{\kappa,\lambda} < \infty,
	\end{align*} 
	for all $t \geq 0$, $x\in\R$, and $(\kappa,\lambda)\in\Lambda$, and for per-unit control costs $c_U, c_D > 0$, define the long-run average cost
	\begin{align}\label{eq:long-run_average}
		J_x(U,D,\kappa,\lambda)
		\defeq
		\limsup_{t\rightarrow\infty}
		\frac{1}{t}
		\Esp[
		\int_0^t c(X_s^{U,D})\,\rmd s
		+ c_U U_t
		+ c_D D_t
		]_{x}^{\kappa,\lambda},
	\end{align}
	where $(\kappa,\lambda) \in \Lambda$ and $(U,D) \in \cA$, with
	\begin{align}\label{eq:admissible_controls}
		\cA
		\defeq
		\left\{
		(U,D) \in \uline{\cA}
		:
		\limsup_{t\to\infty}
		\frac{1}{t}
		\Esp[|X_t^{U,D}|]_x^{\kappa,\lambda}
		= 0,
		\ \text{for all } (\kappa,\lambda)\in\Lambda
		\right\}.
	\end{align}
	
	We assume that the decision maker is ambiguity-averse and evaluates each control under the worst-case model,
	leading to the robust ergodic singular control problem
	\begin{align}\label{eq:ergodic_singular_control_problem}
		\gamma
		\defeq
		\inf_{(U,D)\in\cA}
		\sup_{(\kappa,\lambda)\in\Lambda}
		J_x(U,D,\kappa,\lambda).
	\end{align}
	
	Solving \eqref{eq:ergodic_singular_control_problem} means obtaining a tractable characterization of $\gamma$ and,
	if they exist, identifying optimal admissible controls $U^*$ and $D^*$ together with worst-case ambiguity
	functions $\kappa^*$ and $\lambda^*$.
	
	\begin{remark}
		We later prove (see Theorem \ref{thm:verification_theorem}) that optimality in
		\eqref{eq:ergodic_singular_control_problem} is achieved by a constant-barrier
		Skorokhod-reflection policy. Under such a policy, the reflected process is positive recurrent
		and admits an invariant probability measure $\pi$. Hence, for any bounded running cost $c$
		(see, e.g., Theorem~17.1.7 in \cite{MeynTweedie2009}),
		\[
		\frac{1}{t}\int_0^t c(X_s^{U,D})\,ds
		\longrightarrow
		\int c(z)\,\rmd \pi(z)
		\qquad
		\text{a.s. under $\Pr_x$, for all $x\in[\underline x,\overline x]$.}
		\]
		In particular, the long-run average cost is independent of the initial state $x$,
		so the ergodic value $\gamma$ does not depend on $x$.
		See also \cite{MenaldiRobin1997} for related ergodic-control results for reflected jump diffusions.
	\end{remark}
	
	\begin{remark}
		The $(\kappa,\lambda)$-uniform condition
		\[
		\limsup_{t\to\infty}
		\frac{1}{t}
		\Esp[|X_t^{U,D}|]_x^{\kappa,\lambda}
		= 0
		\]
		in the definition of $\cA$ is a standard technical assumption used in the verification argument.
		It is satisfied by the reflecting-barrier controls that will later be shown to be optimal.
	\end{remark}
	
	For later reference, we derive a bound for the ergodic value $\gamma$.
	We first introduce an auxiliary result controlling the long-run intervention rate for
	Skorokhod-reflection policies.
	
	A pair $(U,D)$ is said to be of reflecting type if there exist barriers $x_1<x_2$ such that
	$(U,D)=(U^{x_1},D^{x_2})$ and the corresponding reflected process
	$X_t^{x_1,x_2} \defeq X_t^{U^{x_1},D^{x_2}}$ remains in $[x_1,x_2]$
	$\Pr^{\kappa,\lambda}$-a.s. for all $(\kappa,\lambda)\in\Lambda$,
	while the Skorokhod complementarity condition
	\begin{align}\label{eq:no_simultaneous_activation}
		\int_0^\infty  (X_t^{U^{x_1}, D^{x_2}} - x_1)\rmd U_t^{x_1}
		=
		\int_0^\infty  (x_2 - X_t^{U^{x_1}, D^{x_2}})\rmd D_t^{x_2}
		=
		0
	\end{align}
	holds.
	See \cite{slominski_eulers_2001,slominski_stochastic_2010} for existence and uniqueness
	results for Skorokhod reflection in jump-diffusion settings.
	
	\begin{lemma}\label{lm:control_ergodic_bounds}
		Fix $x_1<x_2$ and assume $\Esp[Y^2]<\infty$.
		For all $x\in\R$ and all $(\kappa,\lambda)\in\Lambda$,
		\begin{align}\label{eq:ergodic_control_bound}
			\limsup_{t\to\infty}
			\frac{1}{t}
			\Esp[
			U_t^{x_1}
			+
			D_t^{x_2}
			]_x^{\kappa,\lambda}
			\le
			K(x_1,x_2),
		\end{align}
		where
		\[
		K(x_1,x_2)
		=
		\left(
		\max_{x\in[x_1,x_2]}|b(x)|
		+
		\sigma\delta
		+
		\varepsilon r\Esp[|Y|]
		\right)
		+
		\frac{\sigma^2+r(1+\varepsilon)\Esp[Y^2]}{x_2-x_1}.
		\]
	\end{lemma}
	
	The proof of Lemma \ref{lm:control_ergodic_bounds} is given in Appendix~\ref{app:someproofs}.
	
	The previous bound implies that the ergodic value is finite.
	
	\begin{proposition}\label{pr:ergodic-value_upper_bound}
		For any $x_1<x_2$, 
		$$\gamma \le \Gamma(x_1,x_2),$$
		where the upperbound is given by 
		\begin{align}\label{eq:ergodic-value_upper_bound}
			\Gamma(x_1,x_2) 
			\defeq \max_{x\in[x_1,x_2]} c(x) + (c_U+c_D)K(x_1,x_2),
		\end{align}
	\end{proposition}
	
	The proof of Proposition \ref{pr:ergodic-value_upper_bound} is given in
	Appendix~\ref{app:someproofs}.
	
	%%%%%%%%%%%
	
	%-------------------------------------------------
	\subsection{HJB equation and verification}\label{sec:hjb_verification}
	%-------------------------------------------------
	
	An application of the dynamic programming principle suggests the existence of a function
	$V:\R\to\R$ and a constant~$\gamma$ satisfying the Hamilton--Jacobi--Bellman (HJB) equation
	\begin{align}\label{eq:HJB}
		\min
		\left\{
		c_D - V'(x),
		\;
		c_U + V'(x),
		\;
		\sup_{(\kappa,\lambda)\in\Lambda}
		\left\{
		c(x) + (\InfGen^{\kappa,\lambda}V)(x)
		\right\}
		- \gamma
		\right\}
		=0.
	\end{align}
	
	Using the expression of the generator \eqref{eq:inf_gen} and writing
	$H(x)\defeq V'(x)$, the inner maximization in \eqref{eq:HJB} yields the
	bang-bang distortion policies
	\begin{align}\label{eq:bang-bang_opt}
		\kappa^*(x)
		=
		\begin{cases}
			+\delta, & H(x)\ge 0,\\
			-\delta, & H(x)<0,
		\end{cases}
		\qquad
		\lambda^*(x)
		=
		\begin{cases}
			r(1+\varepsilon), & (\Psi V)(x)\ge 0,\\
			r(1-\varepsilon), & (\Psi V)(x)<0.
		\end{cases}
	\end{align}
	
	Denoting by $\InfGen^* \defeq \InfGen^{\kappa^*,\lambda^*}$, we obtain
	\begin{align}\label{eq:optimal_ambiguity_inner_max}
		\sup_{(\kappa,\lambda)\in\Lambda}
		(\InfGen^{\kappa,\lambda}V)(x)
		=
		(\InfGen^*V)(x).
	\end{align}
	
	On the control side, it is natural to expect that optimal policies are of
	reflecting-barrier type in the sense of Skorokhod reflection.
	Accordingly, we look for candidate optimal controls of the form
	\begin{align}\label{eq:optimal_controls}
		(U^*,D^*)
		=
		(U^{\underline x},D^{\overline x}),
	\end{align}
	for unknown reflecting barriers $\underline x$ and $\overline x$, with $\underline x < \overline x$.
	
	Combining the HJB equation with the barrier-structure ansatz leads to the
	following free-boundary problem for~$(V,\gamma,\underline x,\overline x)$:
	\begin{subequations}
		\begin{empheq}[left=\empheqlbrace]{alignat=3}
			\InfGen^*V(x)+c(x) &= \gamma,
			& x&\in[\underline x,\overline x],
			\label{eq:FBP_continuation}
			\\
			\InfGen^*V(x)+c(x) &\ge \gamma,
			& x&\notin[\underline x,\overline x],
			\label{eq:FBP_stopping}
			\\
			V'(x) &\in [-c_U ,c_D],
			& x&\in(\underline x,\overline x),
			\label{eq:FBP_V'_constraints}
			\\
			V'(x) &= -c_U,
			& x&\le \underline x,
			\label{eq:FBP_lower_stopping}
			\\
			V'(x) &= c_D,
			& x&\ge \overline x,
			\label{eq:FBP_upper_stopping}
			\\
			V &\in C^2(\R). \hspace{2cm}
			\label{eq:FBP_C2}
		\end{empheq}
	\end{subequations}
	
	Since $V$ is linear outside the inaction region,
	\eqref{eq:FBP_lower_stopping}-\eqref{eq:FBP_upper_stopping} imply that
	$V''\equiv 0$ outside $[\underline x,\overline x]$.
	
	The next result shows that any sufficiently regular solution of the
	free-boundary problem characterizes the optimal ergodic value
	and the optimal controls.
	
	\begin{theorem}[Verification theorem]\label{thm:verification_theorem}
		Suppose that $V:\R\to\R$ and $\gamma\in\R$
		solve the free-boundary problem
		\eqref{eq:FBP_continuation}-\eqref{eq:FBP_C2}.
		Then
		\begin{align}\label{eq:gamma_equals_control_value}
			\gamma
			=
			\inf_{(U,D)\in\cA}
			\sup_{(\kappa,\lambda)\in\Lambda}
			J_x(U,D,\kappa,\lambda)
			=
			\sup_{(\kappa,\lambda)\in\Lambda}
			\inf_{(U,D)\in\cA}
			J_x(U,D,\kappa,\lambda)
			=
			J_x(U^*,D^*,\kappa^*,\lambda^*),
		\end{align}
		where $(U^*,D^*)$ are the reflecting-barrier controls
		\eqref{eq:optimal_controls} and
		$(\kappa^*,\lambda^*)$ are given by
		\eqref{eq:bang-bang_opt}.
	\end{theorem}
	
	The proof follows standard verification arguments based on Dynkin's formula
	and the gradient constraints, and is reported in
	Appendix~\ref{app:someproofs}.
	
	%-------------------------------------------------%
	\section{Characterization of the optimal policy: The case of negative jumps}
	\label{sec:optimalpolicy}
	%-------------------------------------------------%
	
	\subsection{Structure of the ambiguity regions}
	\label{sec:bang-bang_negative-jumps}
	
	This section characterizes the structure of the optimal ambiguity distortions
	introduced in \eqref{eq:bang-bang_opt}. 
	Under natural convexity conditions on the value function, the worst-case drift
	and intensity distortions exhibit a threshold (bang-bang) structure, which plays a key
	role in the construction of the optimal policy and in the numerical solution of the model.
	
	Throughout the section we assume that the candidate value function $V$
	is strictly convex in the inaction region~$(\uline{x},\oline{x})$,
	so that $V'$ is increasing.
	This assumption is not restrictive for the class of problems considered here.
	In the guess-and-verify procedure used later (see Theorem~\ref{thm:verifying_candidate}),
	the candidate solution inherits strict convexity from the running cost $c$,
	and the resulting structure can then be verified through
	Theorem~\ref{thm:verification_theorem}.
	
	Under this regularity, the ambiguity regions are separated by a finite number
	of thresholds that determine the worst-case drift and intensity distortions.
	
	\medskip
	
	A first consequence of strict convexity is that the ambiguous drift switches sign
	at a unique point inside the inaction region.
	
	\begin{proposition}[Structure of the drift ambiguity region]
		\label{pr:ambg-drift_bang-bang}
		If $V$ and $\gamma$ satisfy the free-boundary problem
		\eqref{eq:FBP_stopping}-\eqref{eq:FBP_C2} and $V$ is strictly convex in
		$(\uline{x},\oline{x})$, then there exists a unique point
		$x^\kappa \in (\uline{x},\oline{x})$ such that
		\[
		V'(x^\kappa)=0, \qquad
		V'(x)<0 \text{ for } x<x^\kappa, \qquad
		V'(x)>0 \text{ for } x>x^\kappa.
		\]
	\end{proposition}
	
	The proof is given in Appendix~\ref{app:someproofs}.
	
	\medskip
	
	We next study the structure of the ambiguity region associated with the jump
	intensity. To obtain a monotone structure, we restrict attention to the case of
	negative jumps.
	
	\begin{assumption}[Negative jumps]
		\label{asm:neg_jumps}
		$\mathrm{supp}(F)\subset(-\infty,0]$ and $\Pr[Y<0]>0$.
	\end{assumption}
	
	Under this assumption, the sign of $(\Psi V)(x)$ changes at most once,
	leading to a single threshold for the intensity distortion.
	
	\begin{proposition}[Structure of the intensity ambiguity region]
		\label{pr:ambg-intensity_bang-bang}
		Suppose that $V$ satisfies the free-boundary problem
		\eqref{eq:FBP_continuation}-\eqref{eq:FBP_C2} and is strictly convex in
		$(\uline{x},\oline{x})$, and let Assumption~\ref{asm:neg_jumps} hold.
		Then there exists a unique $x^\lambda$, such that~$x^\lambda>x^\kappa>\uline{x}$, and 
		\[
		(\Psi V)(x^\lambda)=0, \qquad
		(\Psi V)(x)>0 \text{ for } x<x^\lambda, \qquad
		(\Psi V)(x)<0 \text{ for } x>x^\lambda.
		\]
		
		In addition, %$x^\lambda<\oline{x}$ if and only if
		\begin{align}
			x^\lambda<\oline{x} \iff
			(\Psi V)(\oline{x})
			=
			\Esp[V(\oline{x}+Y)-V(\oline{x})]
			<0.
			\label{eq:xla_in_inaction_suf_cond}
		\end{align}
	\end{proposition}
	
	The proof is given in Appendix~\ref{app:someproofs}.
	
	\medskip
	
	Propositions
	\ref{pr:ambg-drift_bang-bang}-\ref{pr:ambg-intensity_bang-bang}
	show that the optimal distortions have the threshold form
	\begin{align}
		\kappa^*(x)
		&=
		\begin{cases}
			+\delta, & x \ge x^\kappa,\\
			-\delta, & x < x^\kappa,
		\end{cases}
		\qquad
		\lambda^*(x)
		=
		\begin{cases}
			r(1+\varepsilon), & x \le x^\lambda,\\
			r(1-\varepsilon), & x > x^\lambda,
		\end{cases}
		\label{eq:bang-bang_neg-jumps}
	\end{align}
	for some $\uline{x}<x^\kappa<x^\lambda$.
	
	Depending on the position of $x^\lambda$ relative to the upper barrier,
	two mutually exclusive regimes may arise:
	\begin{align}
		\text{Regime 1: }
		\uline{x} < x^\kappa < x^\lambda < \oline{x},
		\qquad
		\text{Regime 2: }
		\uline{x} < x^\kappa < \oline{x} \le x^\lambda.
		\label{eq:regimes}
	\end{align}
	
	Although the regime cannot always be determined a priori,
	this does not create difficulties either theoretically or numerically.
	From a theoretical point of view, both regimes can be treated within the same
	verification argument (see Theorem~\ref{thm:verifying_candidate}), and 
	from a computational perspective, both cases can be handled within the same
	numerical procedure, as described in Section~\ref{sec:numerics}.
	
	Nevertheless, identifying the regime in advance may simplify the analysis
	and reduce the computational cost.
	The next result provides a sufficient condition for Regime~1 to hold,
	expressed in terms of the parameters of the model.
	
	\begin{proposition}
		\label{pr:suf_condition_regime1}
		Condition \eqref{eq:xla_in_inaction_suf_cond} holds if there exist
		$x_1<x_2$ such that
		\[
		\Gamma(x_1,x_2)
		<
		\uline{c}
		+
		c_D(\uline{b}+\sigma\delta-r\Esp[Y]),
		\]
		for constants $\uline{c}$ and $\uline{b}$ satisfying
		$c(\oline{x})\ge \uline{c}$ and $b(\oline{x})\ge \uline{b}$.
	\end{proposition}
	
	The proof is given in Appendix~\ref{app:someproofs}.
	
	\medskip
	
	For later use, we record the following identity for the operator $\Psi$
	under negative jumps. Its proof follows immediately by an integration by parts.
	
	\begin{lemma}
		\label{lm:int_op}
		Under Assumption~\ref{asm:neg_jumps}, if $f\in C^1(\R)$ has bounded derivative
		and $\lim_{y\to-\infty}f(x+y)F(y)=0$, then
		\begin{equation}
			\Psi f = -\mathbb{J} f',
			\label{eq:int_op_neg-jumps}
		\end{equation}
		where
		\[
		(\mathbb{J}f')(x)
		=
		\int_{-\infty}^0 f'(x+y)F(y)\,dy.
		\]
	\end{lemma}

	%%%%%%%%%%%%

	%-------------------------------------------------
	\subsection{Candidate solution and verification}
	\label{sec:guess_verify}
	%-------------------------------------------------
	
	In this subsection we construct a candidate solution to the robust ergodic singular control problem
	\eqref{eq:ergodic_singular_control_problem} and verify its optimality.
	The construction relies on the structural results obtained in
	Section~\ref{sec:bang-bang_negative-jumps}, which show that under negative jumps
	the optimal distortions have a threshold (bang-bang) form.
	Motivated by these results, we look for a solution characterized by four unknown
	thresholds corresponding to the reflecting barriers and to the switching points
	of the drift and intensity distortions.
	
	For some $\uline{x}_*, \oline{x}_*, x_*^\kappa, x_*^\lambda \in \R$ such that
	$\uline{x}_* < x_*^\kappa < x_*^\lambda < \oline{x}_*$, define the intervals
	\begin{align*}
		I_1^* &\defeq (\uline{x}_*, x^\kappa_*), \quad 
		I_2^* \defeq (x^\kappa_*, x^\lambda_*), \quad 
		I_3^* \defeq (x^\lambda_*, \oline{x}_*),
	\end{align*}
	and consider the bang-bang distortions
	\begin{align} \label{eq:bang-bang_candidate}
		\kappa^*(x) &= 
		\begin{cases}
			+\delta, & x \geq x_*^\kappa, \\
			-\delta, & x < x_*^\kappa,
		\end{cases}
		\qquad
		\lambda^*(x) =  
		\begin{cases}
			r(1+\varepsilon), & x \leq x_*^\lambda, \\
			r(1-\varepsilon), & x > x_*^\lambda.
		\end{cases}
	\end{align}
	
	To obtain a tractable characterization, we restrict attention to the case of
	constant drift.
	
	\begin{assumption}[Constant drift]\label{asm:constant_drift}
		The drift $b(x)$ in \eqref{eq:benchmark_model} is constant, which we denote by
		$b(x)\equiv b$.
	\end{assumption}
	
	Let $\InfGen^{*} = \InfGen^{\kappa^*,\lambda^*}$.
	We look for a function $H_*$ satisfying
	\begin{subequations}
		\begin{empheq}[left=\empheqlbrace]{alignat=2}
			(\InfGen^{*} H_*)(x) + c'(x) &= 0 & x &\in \cup_{i=1}^3 I_i^*, \label{eq:FBP_candidate_continuation} \\
			H_*(x) &= -c_U, & x &\leq \uline{x}_*, \label{eq:FBP_candidate_lower_stopping} \\
			H_*(x) &= c_D, & x & \geq \oline{x}_*, \label{eq:FBP_candidate_upper_stopping} \\ 
			H_*(x_*^\kappa)               &= 0,           \label{eq:FBP_candidate_xk} \\
			(\mathbb{J} H_*)(x_*^\lambda) &= 0,           \label{eq:FBP_candidate_xla} \\
			H_*&\in C^1(\R), \label{eq:FBP_candidate_C2} \\
			H_*&\in C^3\lrp{\cup_{i=1}^3 I_i^*}, \label{eq:FBP_candidate_C4}
		\end{empheq}
	\end{subequations}
	
	We impose the following regularity and convexity assumptions.
	
	\begin{assumption}[Convexity and regularity of the cost function]\label{asm:C2_convex_cost}
		The cost function $c:\R\to\R$ is $C^2(\R)$, convex everywhere and strictly convex in
		$[\uline{x}_*, \oline{x}_*]$, and attains its unique minimum at $x = 0$.
	\end{assumption}
	
	\begin{assumption}[Convexity at ambiguity thresholds]\label{asm:V_convex_ambiguity-thresholds}
		$H_*'(x_*^\kappa) > 0$ and $H_*'(x_*^\lambda) > 0$.
	\end{assumption}
	
	Define
	\begin{align}\label{eq:gamma_opt}
		\gamma_*  \defeq 
		%-c_U(b - \delta\sigma - r\Esp[Y]) - c_U(1+\varepsilon)r\Esp[Y] + c(\uline{x}_*) =
		c_U(\delta\sigma - b - \varepsilon r\Esp[Y]) + c(\uline{x}_*).
	\end{align}
	
	We now prove that the candidate function $V_*$, defined by $V_*' = H_*$,
	together with $\gamma_*$, solves the free-boundary problem
	\eqref{eq:FBP_continuation}--\eqref{eq:FBP_C2}.
	By Theorem~\ref{thm:verification_theorem}, this implies optimality for the ergodic
	control problem.
	
	\begin{theorem}[Verification of the candidate solution]\label{thm:verifying_candidate}
		Under Assumptions
		\ref{asm:neg_jumps},
		\ref{asm:constant_drift},
		\ref{asm:C2_convex_cost}
		and
		\ref{asm:V_convex_ambiguity-thresholds},
		the pair $(V_*,\gamma_*)$ solves the free-boundary problem
		\eqref{eq:FBP_continuation}--\eqref{eq:FBP_C2}.
	\end{theorem}
	
	\begin{proof}
		The proof is organized in six steps.
		\smallskip
		
		\noindent \textbf{Step 1. Validity of the HJB on $\pmb{(\uline{x}_*, \oline{x}_*)}$.} 
		
		With $a^*(x) \defeq b + \sigma\kappa^*(x) - r\Esp[Y]$, let
		\begin{align*}
			\mathcal{R}(x) \defeq c(x) + (\InfGen^* V_*)(x) - \gamma_* 
			= 
			c(x) + \frac{1}{2}\sigma^2 H'_*(x) + a^*(x)H_*(x) - \lambda^*(x)(\mathbb{J}H_*)(x) - \gamma_*.
		\end{align*}
		We now prove that the HJB equation is satisfied on the inaction region $I^* = (\uline{x}_*, \oline{x}_*)$. 
		
		\begin{itemize}[label=$\triangleright$]
			\item \emph{$\mathcal{R}(x) = 0$  for $x \in I^*$}:
			Notice that, due to \eqref{eq:FBP_candidate_C2}, and using that $\mathbb{J}H_*$ is continuous as $H_*$ is bounded and $\Esp[Y]$ is finite, we have that
			\begin{align*}
				\mathcal{R}(x+) - \mathcal{R}(x-) = (a^*(x+) - a^*(x-))H_*(x) - (\lambda^*(x+) - \lambda^*(x-))(\mathbb{J}H_*)(x).
			\end{align*}
			Hence, $\mathcal{R} \in C^0(\R)$ after \eqref{eq:FBP_candidate_xk} and \eqref{eq:FBP_candidate_xla}. 
			
			The IDE \eqref{eq:FBP_candidate_continuation} yields that $\mathcal{R}' \equiv 0$ on $\cup_{i=1}^3I_i^*$. That is, $\mathcal{R}$ is constant on each interval $I_i^*$. 
			
			By recalling $\gamma_*$ in \eqref{eq:gamma_opt} one obtains that $\mathcal{R}(\uline{x}_*) = 0$ and, then, due to the continuity of $\mathcal{R}$ \details{(which results from \eqref{eq:FBP_candidate_C2})}{} and its constancy on each $I_i^*$, we obtain that $\mathcal{R} \equiv \mathcal{R}(\uline{x}_*) = 0$ on $[\uline{x}_*, \oline{x}_*]$.
			
			\item \emph{$-c_U \leq V_*'(x) \leq c_D$ for $x \in (\uline{x}_*, \oline{x}_*)$}:
			This claim follows after recalling that $V_*' = H_*$ is non decreasing, along with $H_*(\uline{x}_*) = -c_U$ and $H_*(\oline{x}_*) = c_D$.
		\end{itemize}
		
		\noindent \textbf{Step 2. $\pmb{H_*}$ is increasing in $\pmb{[\uline{x}_*, \oline{x}_*]}$.} 
		
		Inside $\cup_{i=1}^3 I_i^*$ we have higher smoothness of $\mathcal{R}$ due to \eqref{eq:FBP_candidate_C4}: $\mathcal{R} \in C^2(\cup_{i=1}^3 I_i^*)$. Moreover, differentiating $\mathcal{R}$ twice on each interval $I_i$, and using \eqref{eq:FBP_candidate_continuation}, one obtains that
		\begin{align}\label{eq:IDE_W}
			0 = \mathcal{R}''(x) = \frac{1}{2}\sigma^2W''_*(x) + a^*(x)W_*'(x) + \lambda^*(x)(\Psi W_*)(x) + c''(x).
		\end{align}
		for $W_* = H_*'$, for all $x\in\cup_{i=1}^3 I_i^*$.
		
		Since $W_*$ is continuous everywhere (due to \eqref{eq:FBP_candidate_C2}), it attains its minimum $m$ over $[\uline{x}_*, \oline{x}_*]$ at some point $x_0$ in the same compact interval. A maximum principle argument (see, e.g., Theorem 2.8 and Corollary 2.11 from \cite{huang_maximum_2018}) for integro-differential Waldenfels-type operators implies that $x_0$ cannot be an interior point of the inaction region. Alternatively, a direct verification can be straightforwardly done.
		
		Indeed, assume that $x_0 \in (\uline{x}_*, \oline{x}_*)$. Since $W_*(\uline{x}) = W_*(\oline{x}) = 0$, then $m \leq 0$, which implies that $x_0 \notin \{x_*^\kappa, x_*^\lambda\}$ due to Assumption \ref{asm:V_convex_ambiguity-thresholds}. Hence, $x_0 \in I_i^*$ for some $i = 1,2,3$. Since $W_*$ is sufficiently smooth in each $I_i^*$, it follows that $W_*'(x_0)=0$, $W_*''(x_0)\ge 0$. Additionally, that $x_0$ is a minimizer on the inaction region and $W_*(x)=0$ for all $x\leq\uline{x}$ means that $W_*(x_0+y)-W_*(x_0)\ge 0$ for all $y\leq 0$, implying that the integral term $(\intOp W_* - W_*)(x_0)$ is non-negative. Plugging these into \eqref{eq:IDE_W} gives that
		\begin{align*}
			0=\frac{1}{2}\sigma^2W_*''(x_0)+\lambda(x_0)(\Psi W_*)(x_0) + c''(x_0).    
		\end{align*}
		The right-hand side is a sum of nonnegative terms, which implies that all the addends must vanish for the sum to do the same. In particular, $(\Psi W_*)(x_0) = 0$, which alongside $F$ having support on $(-\infty,0)$, $W_*(x_0+y)-W_*(x_0) \geq 0$ for all $y \leq 0$, and $W_*$ being continuous, forces $W_*$ to be constant (equal to $W_*(x_0) = 0$) over $[\uline{x}_*,x_0]$. Plugging back into~\eqref{eq:IDE_W} the constancy of $W_*$ over $[\uline{x}_*,x_0]$ yields that $c'' \equiv 0$ over the same interval, which contradicts that $c$ is strictly convex there.
		
		Therefore $W_* > 0$ on $(\uline{x}_*, \oline{x}_*)$, so $H_*$ is strictly increasing and $V_*$ is strictly convex on the inaction region.
		
		\medskip
		
		\noindent \pmb{Step 3. $\uline{x}_* \leq 0.$}
		
		Since $H_*(\uline{x}_*) = -c_U$, $H_*'(\uline{x}_*) = 0$ and $(\Psi H_*)(\uline{x}_*) = 0$\details{(since $H_* \equiv -c_U$ on $(-\infty, \uline{x}_*]$)}{}, and due to $H_*$ satisfying the IDE \eqref{eq:FBP_candidate_continuation} on $I_1^*$ along with $H_* \in C^1(\R)$, one has that
		\begin{align*}
			0 
			&= 
			\lim_{x\downarrow\uline{x}_*} 
			\lrp{\frac{1}{2}\sigma^2H_*''(x) + a^*(x)H_*'(x) + \lambda^*(x)(\Psi H_*)(x) + c'(x)}
			= \frac{1}{2}\sigma^2 H_*''(\uline{x}_*+) + c'(\uline{x}_*),
		\end{align*}
		meaning that the right second derivative of $H_*$ at $x = \uline{x}_*$ takes the form $H_*''(\uline{x}_*+) = -2c'(\uline{x}_*)/\sigma^2$. Since $H_*\in C^2(I_1^*)$, a right Taylor expansion at $x = \uline{x}_*$ gives, for $h > 0$,
		\begin{align*}
			H_*(\uline{x}_* + h) 
			&= 
			H_*(\uline{x}_*) + H_*'(\uline{x}_*)h + \frac{1}{2}H_*''(\uline{x}_*+)h^2 + o(h^2)
			= -c_U - \sigma^{-2}c'(\uline{x}_*)h^2 + o(h^2).
		\end{align*}
		Hence, $c'(\uline{x}_*) \leq 0$ must hold to avoid contradicting the fact that $H_*$ is increasing. Hence, as $c$ is convex and it attains its global minimum at $x = 0$, $\uline{x}_* \leq 0$ must hold true. 
		\medskip
		
		\noindent \pmb{Step 4. $\oline{x}_* > 0$}
		
		Due to \eqref{eq:FBP_candidate_continuation}-\eqref{eq:FBP_candidate_upper_stopping}, using $H_* \in C^1(\R)$, and applying Lemma \ref{lm:int_op} to $H_*$, one has that
		\begin{align*}
			0 
			&= 
			\lim_{x\uparrow\oline{x}_*} 
			\lrp{\frac{1}{2}\sigma^2H_*''(x) + a^*(x)H_*'(x) - \lambda^*(x)(\mathbb{J}H_*')(x) + c'(x)}
			= \frac{1}{2}\sigma^2 H_*''(\oline{x}_*-) - \lambda^*(\oline{x}_*)(\mathbb{J}H_*')(\oline{x}_*) + c'(\oline{x}_*),
		\end{align*}
		that is, the left second derivative of $H$ at $x = \oline{x}_*$ takes the form $H_*''(\oline{x}_*-) = 2\sigma^{-2}\lrp{\lambda^*(\oline{x}_*)(\mathbb{J}H_*')(\oline{x}_*) - c'(\oline{x}_*)}$. Hence, a left Taylor expansion at $x = \oline{x}_*$ yields, for $h > 0$,
		\begin{align}\label{eq:H_left-Taylor_at_xU}
			H_*(\oline{x}_* - h) 
			&= 
			H_*(\oline{x}_*) - H_*'(\oline{x}_*)h + \frac{1}{2}H_*''(\oline{x}_*-)h^2 + o(h^2)
			= c_D + \sigma^{-2}\lrp{\lambda^*(\oline{x}_*)(\mathbb{J}H_*')(\oline{x}_*) - c'(\oline{x}_*)}h^2 + o(h^2).
		\end{align}
		Then, $c'(\oline{x}_*) \geq \lambda^*(\oline{x}_*)(\mathbb{J}H_*')(\oline{x}_*) > 0$ must hold to avoid contradicting the fact that $H_*$ is increasing. Since $c$ is convex with global minimum at $x = 0$, this implies that $\oline{x}_* > 0$.
		\medskip
		
		\noindent \textbf{Step 5. Validity of the HJB on $\pmb{(-\infty, \uline{x}_*]}$.}
		
		\begin{itemize}[label=$\triangleright$]
			\item \emph{$\mathcal{R}(x) \geq 0$ for $x \in (-\infty, \uline{x}_*]$}:
			On $(-\infty, \uline{x}_*]$, we have that $V_*' \equiv -c_U$, $V_*'' \equiv 0$, $a^* \equiv b - \delta\sigma - r\Esp[Y]$, $\lambda^* \equiv r(1+\varepsilon)$ and, by Lemma \ref{lm:int_op}, $\Psi V_* \equiv -c_U\Esp[Y]$. Hence, for $x\leq \uline{x}_*$,
			\begin{align*}
				\mathcal{R}(x) 
				&= \frac{1}{2}\sigma^2V_*''(x) + a^*(x)V_*'(x) + \lambda^*(x)(\Psi V_*)(x)  + c(x) - \gamma_* \\
				&= -(b - \delta\sigma - r\Esp[Y])c_U - r(1+\varepsilon)c_U\Esp[Y] + c(x) - \gamma_* \\
				&= c(x) - c(\uline{x}_*).
			\end{align*}
			
			Since $\uline{x}_* < 0$ and $c$ is a convex function attaining its minimum at $x = 0$, then $c(x) - c(\uline{x}_*) \geq 0$ for $x \leq \uline{x}_*$ and, consequently, $\mathcal{R}(x) \geq 0$. 
			
			\item \emph{$-c_U = V_*'(x) \leq c_D$ for $x \in (-\infty, \uline{x}_*]$}:
			This claim follows trivially from the construction $V_*' = H_*$ and \eqref{eq:FBP_candidate_lower_stopping}.
		\end{itemize}
		\medskip
		
		\noindent \textbf{Step 6. Validity of the HJB on $\pmb{[\oline{x}_*, \infty)}$.}
		
		\begin{itemize}[label=$\triangleright$]
			\item \emph{$\mathcal{R}(x) \geq 0$ for $x \in [\oline{x}_*, \infty)$}:
			On the interval $(\oline{x}_*, \infty)$ we have that: $H_*''$ and $H_*'$ vanish; $\Psi H_*' > 0$ as $H_*$ is strictly increasing on $I^*$ zero elsewhere; and $c'' \geq 0$ by convexity. Hence, $\mathcal{R}'' > 0$ on the same interval.  
			
			On the other hand, 
			\begin{align*}
				\mathcal{R}'(\oline{x}_*+) =  c'(\oline{x}_*) - \lambda^*(\oline{x}_*)(\mathbb{J}H_*')(\oline{x}_*) \geq 0
			\end{align*}
			where the inequality is justified as, otherwise, it would conflict with the increasing monotonicity of $H_*$, as shown in the left-Taylor expansion \eqref{eq:H_left-Taylor_at_xU}.
			
			We have previously proved that $\mathcal{R} \equiv 0$ on $I^*$. In particular, since $\mathcal{R}$ is continuous, then $\mathcal{R}(\oline{x}_*) = 0$, which, combined with $\mathcal{R}'(\oline{x}_*+) \geq 0$ and $\mathcal{R}''(x) > 0$ for all $x\in(\oline{x}_*,\infty)$, proves that $\mathcal{R} \geq 0$ on $[\oline{x}_*,\infty)$.
			
			\item \emph{$-c_U \leq V_*'(x) = c_D$ for $x \in [\oline{x}_*, \infty)$}:
			This claim is trivial after $V_*' = H_*$ and \eqref{eq:FBP_candidate_upper_stopping}.    
		\end{itemize}
	\end{proof}
	\begin{remark}
		Theorem \ref{thm:verifying_candidate} is stated under Regime 1.
		However, analogous arguments apply in Regime 2.
		In that case the intensity distortion remains constant on the inaction region,
		which simplifies the free-boundary system and reduces the number of matching
		conditions. In this regime, the threshold $x^\lambda$ is determined by the sign
		change of $\Psi V$ (see Section~\ref{sec:numerics}, Algorithm \ref{alg:two-stages_solver}).
	\end{remark}
	
	\begin{remark}\label{rmk:convexity_V}
		Proposing a candidate $V_*$ with the simplified bang-bang regions
		\eqref{eq:bang-bang_candidate} implicitly assumes convexity of $V_*$,
		as discussed in Section~\ref{sec:bang-bang_negative-jumps}.
		The proof of Theorem~\ref{thm:verifying_candidate} shows that this convexity
		is not restrictive but is in fact a property that must hold for any solution
		of the free-boundary problem
		\eqref{eq:FBP_continuation}--\eqref{eq:FBP_C2}
		under Assumption~\ref{asm:neg_jumps}.
	\end{remark}
	
	%-------------------------------------------------%
	\subsection{Explicit solution under exponentially distributed negative jumps}
	\label{sec:negative_exp_jumps}
	%-------------------------------------------------%
	
	In this subsection we specialize the model to the case of exponentially distributed
	negative jumps. This specification yields a tractable representation of the
	integro-differential free-boundary problem and allows us to reduce the characterization
	of the optimal policy to a system of ordinary differential equations.
	The exponential case is particularly useful for two reasons. First, it provides
	additional analytical insight into the structure of the optimal policy. Second,
	it leads to a formulation that is well suited for numerical implementation.
	
	\begin{assumption}[Negative exponential jumps]\label{asm:neg-exp_jumps}
		$F(y) = e^{\mu (y\wedge 0)}$, for all $y\in\R$, and for some $\mu > 0$.
	\end{assumption}
	
	A direct consequence of Assumption \ref{asm:neg-exp_jumps} is that the integral operator
	$\intOp$ satisfies a simple relation, which follows as a particular case of
	Lemma \ref{lm:int_op}.
	
	\begin{corollary}\label{cor:int_op_neg-exp-jumps}
		Under Assumption \ref{asm:neg-exp_jumps}, and with $f\in C^1(\R)$ such that
		$\lim_{y\to-\infty}f(x + y)e^{\mu y} = 0$, the operator $\intOp$ satisfies
		\begin{equation}\label{eq:der_int_op}
			\intOp f - f = -\frac{1}{\mu}\intOp f'.
		\end{equation}
	\end{corollary}
	
	Assume now that Regime 1 in \eqref{eq:regimes} holds
	(we provide in Lemma \ref{lm:Regime-1_parabolic_cost} a sufficient condition for
	Regime 1 under $c(x)=x^2$), and let Assumptions
	\ref{asm:C2_convex_cost}, \ref{asm:V_convex_ambiguity-thresholds},
	and \ref{asm:neg-exp_jumps} hold.
	Then the candidate solution described in Section~\ref{sec:guess_verify}
	is strictly convex (see Remark \ref{rmk:convexity_V}), and the optimal ambiguity
	functions therefore take the threshold form \eqref{eq:bang-bang_neg-jumps}.
	
	Recall that our candidate $V_*$ should satisfy $H_* = V_*'$ and solve the
	free-boundary problem \eqref{eq:FBP_candidate_continuation}--\eqref{eq:FBP_candidate_C4},
	for free boundaries
	$\uline{x}_* < x_*^\kappa < x_*^\lambda < \oline{x}_*$.
	In particular,
	\begin{align}
		0 &= \frac{1}{2}\sigma^2H_*''(x) + a^*(x)H_*'(x) -\lambda^*(x)H_*(x) + \lambda^*(x) (\intOp H_*)(x) + c'(x), \nonumber \\
		&= \frac{1}{2}\sigma^2H_*''(x) + a^*(x)H_*'(x) - \lambda^*(x)H_*(x) + \mu\lambda^*(x)(V_*(x) - \intOp V_*(x)) + c'(x), \label{eq:IDE_H}
	\end{align}
	for all $x \in \cup_{i=1}^3 I_i^*$,
	where the second equality follows from Corollary \ref{cor:int_op_neg-exp-jumps}.
	Under Assumption \ref{asm:neg-exp_jumps}, $\Esp[Y] = -(1/\mu)$ and~$a^*(x) = b + \sigma\kappa^*(x) + r/\mu$.
	As shown in the proof of Theorem \ref{thm:verifying_candidate},
	$V_*$ satisfies $\InfGen^*V_* + c \equiv \gamma_*$ in the inaction region~$(\uline{x}_*,\oline{x}_*)$, which, after rearranging terms, gives
	\begin{align}\label{eq:int_op_V_substitution}
		\lambda^*(x)(\intOp V_*(x) - V_*(x)) 
		&= \gamma_* - \frac{1}{2}\sigma^2H_*'(x) - a^*(x)H_*(x) - c(x).
	\end{align}
	Substituting \eqref{eq:int_op_V_substitution} into \eqref{eq:IDE_H} yields the
	second-order ordinary differential equation
	\begin{align}\label{eq:ODE_H}
		\frac{1}{2}\sigma^2H_*''(x) + \left(a^*(x) + \frac{1}{2}\mu\sigma^2\right) H_*'(x) + (a^*(x)\mu - \lambda^*(x))H_*(x) = - \mu c(x) - c'(x) + \mu\gamma_*,
	\end{align}
	for $x \in \cup_{i=1}^3 I_i^*$.
	Restricting \eqref{eq:ODE_H} to a single interval $I_i^*$ yields the constant-coefficient ODE
	\begin{align}\label{eq:ODE_H_const_coeff}
		a_{i,3}H_*''(x) + a_{i,2} H_*'(x) + a_{i,1}H_*(x) + b_2c(x) + b_1c'(x) + b_0 = 0,
	\end{align}
	where $b_2 = \mu$, $b_1 = 1$, $b_0 = -\mu\gamma_*$, and    
	\begin{align}\label{eq:ODE_H_coeff}
		a_{i,3} &= \sigma^2/2, \quad
		a_{i,2} = a_i^* + \mu\sigma^2/2, \quad
		a_{i,1} = \mu\left(b + \sigma \kappa^*_i + r/\mu\right) - \lambda^*_i, 
	\end{align}
	for $a_i^* = a^*(x)$ and $\lambda^*_i=\lambda^*(x)$ for $x \in I_i^*$.
	
	%-------------------------------------------------%
	\subsubsection{The case of a quadratic running cost}
	\label{sec:parabolic_cost}
	%-------------------------------------------------%
	
	We now specialize further to a quadratic running cost.
	
	\begin{assumption}[Quadratic cost]\label{asm:parabolic_cost}
		$c(x) = x^2$.
	\end{assumption}
	
	Under parabolic cost and constant drift, the upper bound
	\eqref{eq:ergodic-value_upper_bound} can be made fully explicit. The proof of the next result is postponed to Appendix \ref{app:someproofs}.
	
	\begin{proposition}\label{pr:ergodic-value_upper_bound_parabolic-cost}
		Under Assumptions \ref{asm:constant_drift} and \ref{asm:parabolic_cost},
		and for a jump distribution $F$ such that $\Esp[Y^2] < \infty$, one has
		\begin{align*}
			\gamma \leq \min_{x_1<x_2} \Gamma(x_1,x_2)
			=
			(c_U+c_D)\left(|b|+\sigma\delta+\varepsilon r\Esp[|Y|]\right)
			+ 3\left((c_U+c_D)\left(\sigma^2+r(1+\varepsilon)\Esp[Y^2]\right)/4\right)^{2/3},
		\end{align*}
		where $\Gamma(x_1,x_2)$ is as in \eqref{eq:ergodic-value_upper_bound}.
	\end{proposition}
	
	When parabolic cost is combined with negative exponential jumps and constant drift,
	Proposition \ref{pr:ergodic-value_upper_bound_parabolic-cost} yields a sufficient
	condition for Regime 1 in \eqref{eq:regimes} that depends only on the model parameters
	$(b,\delta,r,\varepsilon,\sigma,c_U,c_D,\mu)$.
	
	\begin{lemma}\label{lm:Regime-1_parabolic_cost}
		Under Assumptions \ref{asm:constant_drift}, \ref{asm:neg-exp_jumps}, and \ref{asm:parabolic_cost},
		the following condition is feasible and implies that
		$\uline{x}_* < x_*^\kappa < x_*^\lambda < \oline{x}_*$:
		\begin{align}\label{eq:suf_cond_regime1_parabolic-cost}
			(c_U+c_D)\left(|b|+\sigma\delta+\varepsilon r/\mu\right)
			+ 3\left((c_U+c_D)\left(\sigma^2+2r(1+\varepsilon)/\mu^2\right)/4\right)^{2/3}
			< c_D(\delta\sigma + b + r/\mu).
		\end{align}
		Moreover, condition \eqref{eq:suf_cond_regime1_parabolic-cost} is satisfied if
		\begin{align}\label{eq:suf_to_suf_regime1_condition}
			b \geq 0,\quad 
			\varepsilon < c_D/(c_D+c_U),\quad 
			(r/\mu)^{1/3} > \max\{2\hat{K}_3/\hat{K}_1, (2\hat{K}_2/\hat{K}_1)^{1/3}\},
		\end{align}
		where
		\begin{align*}
			\hat{K}_1 \defeq c_D - (c_D+c_U)\varepsilon, \quad
			\hat{K}_2 \defeq c_U(b + \delta\sigma) + 3((c_D+c_U)\sigma^2/4)^{2/3},\quad
			\hat{K}_3 \defeq 3((c_D+c_U)/4)^{2/3}(2(1+\varepsilon)/\mu)^{2/3}.
		\end{align*}
	\end{lemma}
	The proof of Lemma \ref{lm:Regime-1_parabolic_cost} can be found in Appendix \ref{app:someproofs}.
	
	\begin{remark}
		Condition \eqref{eq:suf_to_suf_regime1_condition}, which is sufficient for
		\eqref{eq:suf_cond_regime1_parabolic-cost}, can be met quite easily.
		For example, after fixing all parameters except $r$, once the first two inequalities
		in \eqref{eq:suf_to_suf_regime1_condition} hold, one can choose $r$ sufficiently large
		so that the third inequality is also satisfied.
	\end{remark}
	
	Under Assumption \ref{asm:parabolic_cost}, equation \eqref{eq:ODE_H_const_coeff}
	admits an explicit solution.
	Indeed, since
	\[
	\Delta_i = a_{i,2}^2 - 4a_{i,3}a_{i,1}
	= \left(a_i^* - \frac{1}{2}\mu\sigma^2\right)^2 + 2\sigma^2\lambda_i^* > 0,
	\]
	the ODE \eqref{eq:ODE_H_const_coeff} has the solution
	(see, e.g., Sections 0.4.1-1 and 0.4.1-2 in \cite{PolyaninZaitsev2003HandbookODE})
	\begin{align}\label{eq:ODE_H_sol}
		H_i(x) &=  \mathbf{c}_{i}^{-} e^{-\rho_i^- x} + \mathbf{c}_{i}^{+} e^{-\rho_i^+ x} + p_i(x),
	\end{align} 
	where
	\begin{align}\label{eq:ODE_H_exponents}
		\rho_i^+ = \frac{a_{i,2}+\sqrt{\Delta_i}}{2 a_{i,3}}, \qquad
		\rho_i^- = \frac{a_{i,2}-\sqrt{\Delta_i}}{2 a_{i,3}},
	\end{align}
	and $p_i(x)$ is a polynomial of degree at most three.
	For the sake of exposition we first consider the nondegenerate case in which
	$a_{i,1}\neq 0$.
	
	\begin{assumption}[$a_{i,1}$ does not vanish]\label{asm:a_i1_not_zero}
		$a_{i,1} \neq 0$ for $i = 1,2,3$.
	\end{assumption}
	
	Under Assumption \ref{asm:a_i1_not_zero},
	the polynomial term takes the form
	\begin{align}\label{eq:polynomial_term}
		p_i(x) &= c_{i,2} x^2 + c_{i,1} x + c_{i,0}, \qquad
		p_i'(x) = 2 c_{i,2} x + c_{i,1},
	\end{align}
	with
	\begin{align}\label{eq:poly_coeffs}
		c_{i,2} &= -\frac{\mu}{a_{i,1}}, \qquad
		c_{i,1} = \frac{2(\mu a_{i,2}-a_{i,1})}{a_{i,1}^2}, \qquad
		c_{i,0} = \frac{\mu\gamma\,a_{i,1}^2+2a_{i,1}a_{i,2} + 2\mu a_{i,1}a_{i,3}-2\mu a_{i,2}^2}{a_{i,1}^3}.
	\end{align}
	
	\begin{remark}\label{rmk:a_i1_zero}
		If $a_{i,1}=0$, the solution method is unchanged.
		In that case one necessarily has $a_{i,2}\neq 0$, and the polynomial
		term becomes cubic:
		\begin{align*}
			p_i(x) &= c_{i,3} x^3 + c_{i,2} x^2 + c_{i,1} x,
		\end{align*}
		with
		\[
		c_{i,3} = -\frac{b_2}{3a_{i,2}}, \qquad
		c_{i,2} = -\frac{6 a_{i,3} c_{i,3} + 2b_1}{2 a_{i,2}}, \qquad
		c_{i,1} = -\frac{2 a_{i,3} c_{i,2} + b_0}{a_{i,2}}.
		\]
	\end{remark}
	
	Under Assumptions \ref{asm:neg-exp_jumps} and \ref{asm:parabolic_cost},
	we therefore propose a candidate pair $(V_*,\gamma_*)$ such that
	$V_*'=H_*$ and
	\begin{align}\label{eq:gamma_opt_example}
		\gamma_* \defeq 
		%-c_U(b - \delta\sigma + r/\mu) + c_U(1+\varepsilon)r/\mu + \uline{x}_*^2 =
		c_U(\delta\sigma - b + \varepsilon r/\mu) + \uline{x}_*^2,
	\end{align}
	with $H_*$ solving the free-boundary problem
	\begin{subequations}
		\begin{empheq}[left=\empheqlbrace]{alignat=2}
			0 &= a_{i,3}H_*''(x) + a_{i,2} H_*'(x) + a_{i,1}H_*(x) + b_2x^2 + 2b_1x + b_0,
			\quad\quad &x \in I_i^*,
			\label{eq:FBP_candidate_example_continuation} \\
			H_*(x) &= -c_U, \quad\quad &x \leq \uline{x}_*, \label{eq:FBP_candidate_example_lower_stopping} \\
			H_*(x) &= c_D, \quad\quad &x \geq \oline{x}_*, \label{eq:FBP_candidate_example_upper_stopping} \\
			H_*(x_*^\kappa) &= 0,\quad\quad \label{eq:FBP_example_candidate_xk} \\
			(\intOp H_*)(x_*^\lambda) &= 0,\quad\quad \label{eq:FBP_example_candidate_xla} \\
			H_1(\uline{x}_*) &= -c_U,\; H_1(x_*^\kappa) = H_2(x_*^\kappa),\; H_2(x_*^\lambda) = H_3(x_*^\lambda),\; H_3(\oline{x}_*) = c_D,
			\quad\quad &\text{(instantaneous stop)} \label{eq:FBP_example_candidate_instantaneous_stop} \\
			H_1'(\uline{x}_*) &= 0,\; H_1'(x_*^\kappa) = H_2'(x_*^\kappa),\; H_2'(x_*^\lambda) = H_3'(x_*^\lambda),\; H_3'(\oline{x}_*) = 0.
			\quad\quad &\text{(smooth fit)} \label{FBP_candidte_example_smooth-fit}
		\end{empheq}
	\end{subequations}
	Theorem \ref{thm:verifying_candidate} then implies that such
	$V_*$ and $\gamma_*$, together with the reflecting barriers and the ambiguity thresholds
	$\uline{x}_* < x_*^\kappa < x_*^\lambda < \oline{x}_*$,
	solve \eqref{eq:ergodic_singular_control_problem}.
	Note that we do not explicitly impose the higher smoothness in
	\eqref{eq:FBP_candidate_C4}, since it follows automatically from
	\eqref{eq:ODE_H_sol}.
	
	%-------------------------------------------------%
	\subsubsection{Stable representation of the candidate solution}
	%-------------------------------------------------%
	
	We now turn to the numerical implementation of the boundary-value problem
	\eqref{eq:FBP_candidate_example_continuation}--\eqref{FBP_candidte_example_smooth-fit}.
	A direct use of the representation \eqref{eq:ODE_H_sol} may become numerically unstable because of the rapidly growing or decaying exponential terms.
	To avoid overflow and loss of precision, we introduce an adaptive anchoring system.
	
	Specifically, we consider the scaled coefficients
	$\mathbf{u}_{i}^{\pm} \defeq \mathbf{c}_{i}^{\pm}e^{-\rho_i^\pm x_{i,\pm}^{\mathrm{a}}}$,
	where the anchors are chosen as
	\begin{align}\label{eq:anchors}
		x_{1,\pm}^\mathrm{a} 
		&=
		\begin{cases}
			\uline{x}    ,& \text{if } \rho_1^\pm \geq 0 \\
			x^\kappa  ,& \text{if } \rho_1^\pm < 0
		\end{cases} ,
		\quad
		x_{2,\pm}^\mathrm{a} 
		=
		\begin{cases}
			x^\kappa    ,& \text{if } \rho_2^\pm \geq 0 \\
			x^\lambda  ,& \text{if } \rho_2^\pm < 0
		\end{cases} ,
		\quad
		x_{3,\pm}^\mathrm{a} 
		=
		\begin{cases}
			x^\lambda    ,& \text{if } \rho_3^\pm \geq 0 \\
			\oline{x}  ,& \text{if } \rho_3^\pm < 0
		\end{cases}.
	\end{align}
	
	To make explicit the dependence of $H_i$ on $\gamma$ through $c_{i,0}$,
	we also define the polynomial
	$q_i(x) = p_i(x) - \frac{\mu}{a_{i,1}}\gamma$.
	Then \eqref{eq:ODE_H_sol} becomes
	\begin{align}
		H_i(x) &= 
		\mathbf{u}_{i}^{-} e^{-\rho_i^- (x - x_{i,-}^\mathrm{a})} 
		+ \mathbf{u}_{i}^{+} e^{-\rho_i^+ (x - x_{i,+}^\mathrm{a})}
		+ \frac{\mu}{a_{i,1}}\gamma + q_i(x), \label{eq:ODE_H_sol_gamma} \\
		H_i'(x) &= 
		-\rho_i^-\mathbf{u}_{i}^{-} e^{-\rho_i^- (x - x_{i,-}^\mathrm{a})} 
		-\rho_i^+\mathbf{u}_{i}^{+} e^{-\rho_i^+ (x - x_{i,+}^\mathrm{a})} + q_i'(x), \label{eq:ODE_Hprime_sol_gamma}
	\end{align}
	where $q_i$, which is independent of $\gamma$, has the form
	\begin{align*}
		q_i(x) &= c_{i,2} x^2 + c_{i,1} x + d_{i,0}, \qquad
		q_i'(x) = p_i'(x) = 2 c_{i,2} x + c_{i,1},
	\end{align*}
	with
	\begin{align*}
		d_{i,0} = c_{i,0} - \frac{\mu}{a_{i,1}}\gamma
		= \frac{2a_{i,1}a_{i,2} + 2\mu a_{i,1}a_{i,3}-2\mu a_{i,2}^2}{a_{i,1}^3}.
	\end{align*}
	
	%-------------------------------------------------%
	\subsubsection{Solving the boundary-value problem}
	\label{sec:solving_cauchy}
	%-------------------------------------------------%
	
	In this part we take the reflecting barriers and ambiguity thresholds as given,
	and solve the boundary-value problem
	\eqref{eq:FBP_candidate_example_continuation}--\eqref{eq:FBP_example_candidate_instantaneous_stop}.
	Thus, for fixed boundaries $(\uline{x},x^\kappa,x^\lambda,\oline{x})$,
	we determine $\gamma$ and the coefficients $\mathbf{u}_{i}^{\pm}$,
	leaving aside the smooth-fit conditions for the moment.
	
	%-----%
	\paragraph{Finding the value $\gamma$.}\label{par:solving_gamma}\ \\
	\vspace{-0.3cm}
	%-----%
	
	\noindent Without imposing the smooth-fit conditions \eqref{FBP_candidte_example_smooth-fit},
	the value of $\gamma$ is not yet the optimal one in \eqref{eq:gamma_opt_example}.
	Using only the continuity of $H$ from
	\eqref{eq:FBP_example_candidate_instantaneous_stop}, taking
	$x\downarrow \uline{x}$ in \eqref{eq:int_op_V_substitution}, and using
	Corollary \ref{cor:int_op_neg-exp-jumps} together with
	$H(\uline{x}) = (\intOp H)(\uline{x}) = -c_U$ from
	\eqref{eq:FBP_candidate_example_lower_stopping}, we obtain
	\begin{align}\label{eq:gamma_stable}
		\gamma &= \gamma(\uline{x}, \mathbf{u}_{1}^{\pm}) = (\InfGen^{\kappa,\lambda} V)(\uline{x}) + c(\uline{x}) 
		= \frac{1}{2}\sigma^2 H_1'(\uline{x}) + \gamma_* \\
		&= \frac{1}{2}\sigma^2 \lrp{-\rho_1^-\mathbf{u}_{1}^{-} e^{-\rho_1^- (\uline{x} - x_{1,-}^\mathrm{a})} - \rho_1^+\mathbf{u}_{1}^{+} e^{-\rho_1^+ (\uline{x} - x_{1,+}^\mathrm{a})} + q_1'(\uline{x})} + \gamma_* . \nonumber
	\end{align}
	
	%-----%
	\paragraph{Finding the coefficients $\mathbf{u}_{1}^{\pm}$.}
	\label{par:solving_c1_coeff}\ \\
	\vspace{-0.3cm}
	%-----%
	
	\noindent From \eqref{eq:FBP_example_candidate_instantaneous_stop} we obtain
	\begin{align*}
		-c_U &= H_1(\uline{x}) = \mathbf{u}_{1}^{-} e^{-\rho_1^- (\uline{x} -  x_{1,-}^\mathrm{a})} + \mathbf{u}_{1}^{+} e^{-\rho_1^+ (\uline{x} - x_{1,+}^\mathrm{a})} + \frac{\mu}{a_{1,1}}\gamma(\uline{x},\mathbf{u}_{1}^{\pm}) + q_1(\uline{x}) \\
		&= \mathbf{u}_{1}^{-} e^{-\rho_1^- (\uline{x} - x_{1,-}^\mathrm{a})} + \mathbf{u}_{1}^{+} e^{-\rho_1^+ (\uline{x} - x_{1,+}^\mathrm{a})} 
		- \frac{\mu}{a_{1,1}}\frac{\sigma^2}{2} \lrp{\rho_1^-\mathbf{u}_{1}^{-} e^{-\rho_1^- (\uline{x} - x_{1,-}^\mathrm{a})} + \rho_1^+\mathbf{u}_{1}^{+} e^{-\rho_1^+ (\uline{x} - x_{1,+}^\mathrm{a})} - q_1'(\uline{x})} + \frac{\mu}{a_{1,1}}\gamma_* + q_1(\uline{x}).
	\end{align*}
	That is,
	\[
	\mathbf{u}_{1}^{-} m_{1,-}^{(1)} 
	+ \mathbf{u}_{1}^{+} m_{1,+}^{(1)}
	= \beta_{1}^{(1)}
	\]
	for
	\begin{align*}
		m_{1,-}^{(1)} &= e^{-\rho_1^- (\uline{x} - x_{1,-}^\mathrm{a})}\lrp{1 - \frac{\mu}{a_{1,1}}\frac{\sigma^2}{2}\rho_1^-}, \quad 
		m_{1,+}^{(1)} = e^{-\rho_1^+ (\uline{x} - x_{1,+}^\mathrm{a})}\lrp{1 - \frac{\mu}{a_{1,1}}\frac{\sigma^2}{2}\rho_1^+}, \\
		\beta_{1}^{(1)} &= -c_U - q_1(\uline{x}) - \frac{\mu}{a_{1,1}}\lrp{\frac{\sigma^2}{2}q_1'(\uline{x}) + \gamma_*}.
	\end{align*}
	
	Likewise, from \eqref{eq:FBP_example_candidate_xk} we get
	\begin{align*}
		0 &= H_1(x^\kappa) = \mathbf{u}_{1}^{-} e^{-\rho_1^- (x^\kappa - x_{1,-}^\mathrm{a})} + \mathbf{u}_{1}^{+} e^{-\rho_1^+ (x^\kappa - x_{1,+}^\mathrm{a})} + \frac{\mu}{a_{1,1}}\gamma(\uline{x},\mathbf{u}_{1}^{\pm}) + q_1(x^\kappa) \\
		&= \mathbf{u}_{1}^{-} e^{-\rho_1^- (x^\kappa - x_{1,-}^\mathrm{a})} + \mathbf{u}_{1}^{+} e^{-\rho_1^+ (x^\kappa - x_{1,+}^\mathrm{a})} 
		- \frac{\mu}{a_{1,1}}\frac{1}{2}\sigma^2 \lrp{\rho_1^-\mathbf{u}_{1}^{-} e^{-\rho_1^- (\uline{x} - x_{1,-}^\mathrm{a})} + \rho_1^+\mathbf{u}_{1}^{+} e^{-\rho_1^+ (\uline{x} - x_{1,+}^\mathrm{a})} - q_1'(\uline{x})} + \frac{\mu}{a_{1,1}}\gamma_* + q_1(x^\kappa),
	\end{align*}
	which yields
	\[
	\mathbf{u}_{1}^{-} m_{2,-}^{(1)}  
	+ \mathbf{u}_{1}^{+} m_{2,+}^{(1)}
	= \beta_{2}^{(1)}
	\]
	for 
	\begin{align*}
		m_{2,-}^{(1)} &= e^{-\rho_1^- (x^\kappa - x_{1,-}^\mathrm{a})} - \frac{\mu}{a_{1,1}}\frac{\sigma^2}{2}\rho_1^-e^{-\rho_1^- (\uline{x} - x_{1,-}^\mathrm{a})}, \\
		m_{2,+}^{(1)} &= e^{-\rho_1^+ (x^\kappa - x_{1,+}^\mathrm{a})} - \frac{\mu}{a_{1,1}}\frac{\sigma^2}{2}\rho_1^+e^{-\rho_1^+ (\uline{x} - x_{1,+}^\mathrm{a})}, \\
		\beta_{2}^{(1)} &= - q_1(x^\kappa) - \frac{\mu}{a_{1,1}}\lrp{\frac{\sigma^2}{2}q_1'(\uline{x}) + \gamma_*}.
	\end{align*}
	
	Therefore, $\mathbf{u}_{1}^{\pm}$ (equivalently $\mathbf{c}_{1}^{\pm}$) solve
	\begin{align}\label{eq:linear_system_1_stable}
		\begin{pmatrix}
			m_{1,-}^{(1)} & m_{1,+}^{(1)} \\
			m_{2,-}^{(1)} & m_{2,+}^{(1)}
		\end{pmatrix}
		\begin{pmatrix}
			\mathbf{u}_{1}^{-} \\ \mathbf{u}_{1}^{+}
		\end{pmatrix} 
		=
		\begin{pmatrix}
			\beta_{1}^{(1)} \\ \beta_{2}^{(1)}
		\end{pmatrix}\quad
		\lrp{\text{resp.} 
			\begin{pmatrix}
				\alpha_{1,-}^{(1)} & \alpha_{1,+}^{(1)} \\
				\alpha_{2,-}^{(1)} & \alpha_{2,+}^{(1)}
			\end{pmatrix}
			\begin{pmatrix}
				\mathbf{c}_{1}^{-} \\ \mathbf{c}_{1}^{+}
			\end{pmatrix} 
			=
			\begin{pmatrix}
				\beta_{1}^{(1)} \\ \beta_{2}^{(1)}
		\end{pmatrix}},
	\end{align}
	where $\alpha_{i,\pm}^{(1)} = e^{-\rho_1^\pm x_{1,\pm}^\mathrm{a}}m_{i,\pm}^{(1)}$, $i=1,2$.
	
	%-----%
	\paragraph{Finding the coefficients $\mathbf{u}_{2}^{\pm}$.}
	\label{par:solving_c2_coeff}\ \\
	\vspace{-0.3cm}
	%-----%
	
	\noindent Using \eqref{eq:FBP_example_candidate_instantaneous_stop} one obtains
	\begin{align}\label{eq:C0_at_xk}
		0 = H_1(x^\kappa) = H_2(x^\kappa) &=
		\mathbf{u}_{2}^{-} e^{-\rho_2^- (x^\kappa - x_{2,-}^\mathrm{a})} + \mathbf{u}_{2}^{+} e^{-\rho_2^+ (x^\kappa - x_{2,+}^\mathrm{a})} + \frac{\mu}{a_{2,1}}\gamma(\uline{x},\mathbf{u}_{1}^{\pm}) + q_2(x^\kappa).
	\end{align}
	Hence,
	\[
	\mathbf{u}_{2}^{-} m_{1,-}^{(2)} + \mathbf{u}_{2}^{+} m_{1,+}^{(2)} = \beta_1^{(2)}
	\]
	for
	\begin{align*}
		m_{1,-}^{(2)} = e^{-\rho_2^- (x^\kappa - x_{2,-}^\mathrm{a})},\quad
		m_{1,+}^{(2)} = e^{-\rho_2^+ (x^\kappa - x_{2,+}^\mathrm{a})},\quad  
		\beta_1^{(2)} = - \frac{\mu}{a_{2,1}}\gamma(\uline{x},\mathbf{u}_{1}^{\pm}) - q_2(x^\kappa).
	\end{align*}
	
	Using next \eqref{eq:FBP_candidate_xla}, we obtain
	\begin{align}\label{eq:xLa_def_stable} 
		0 
		= (\intOp H)(x^\lambda)
		= \int_{-\infty}^0 H(x^\lambda+y) \mu e^{\mu y}\,\rmd y 
		= \sum_{i=0}^2 \mathcal{I}_i,
	\end{align}
	where
	\begin{align*}
		\mathcal{I}_0 
		&\defeq -c_U \int_{-\infty}^{\uline{x}-x^\lambda} \mu e^{\mu y}\,\rmd y, \\
		\mathcal{I}_1 
		&\defeq \int_{\uline{x}-x^\lambda}^{x^\kappa-x^\lambda} H_1(x^\lambda+y) \mu e^{\mu y}\,\rmd y, \\
		\mathcal{I}_2
		&\defeq \int_{x^\kappa-x^\lambda}^{0} H_2(x^\lambda+y) \mu e^{\mu y}\,\rmd y .
	\end{align*}
	
	We now reduce these terms to explicit expressions.
	
	\begin{itemize}[label=$\triangleright$]
		\item \emph{$\mathcal{I}_0$}: One immediately obtains
		\begin{align*}
			\mathcal{I}_0 
			= -c_U \int_{-\infty}^{\uline{x}-x^\lambda} \mu e^{\mu y}\,\rmd y 
			= -c_U e^{\mu(\uline{x}-x^\lambda)}.
		\end{align*}
		
		\item \emph{$\mathcal{I}_1$}: We write
		\begin{align*}
			\mathcal{I}_1 
			&= \int_{\uline{x}-x^\lambda}^{x^\kappa-x^\lambda} H_1(x^\lambda+y) \mu e^{\mu y}\,\rmd y 
			= \int_{\uline{x}}^{x^\kappa} H_1(y) \mu e^{\mu (y - x^\lambda)}\,\rmd y \\
			&= 
			\sum_{\eta\in\{-,+\}}\int_{\uline{x}}^{x^\kappa} \mu
			\mathbf{u}_{1}^{\eta} e^{-\rho_1^\eta (y - x_{1,\eta}^\mathrm{a})} e^{\mu (y - x^\lambda)}\,\rmd y
			+ \int_{\uline{x}}^{x^\kappa} \mu
			\lrp{\frac{\mu}{a_{1,1}}\gamma(\uline{x},\mathbf{u}_{1}^{\pm}) + q_1(y)} e^{\mu (y - x^\lambda)}\,\rmd y.
		\end{align*}
		The first term simplifies to
		\details{
			\begin{align*}
				\sum_{\eta\in\{-,+\}}\int_{\uline{x}}^{x^\kappa} \mu
				\mathbf{u}_{1}^{\eta} e^{-\rho_1^\eta (y - x_{1,\eta}^\mathrm{a})} e^{\mu (y - x^\lambda)}\,\rmd y
				&=
				\sum_{\eta\in\{-,+\}} 
				\mu\mathbf{u}_{1}^{\eta}
				\frac{
					e^{-\mu(x^\lambda-x^\kappa)}e^{-\rho_1^\eta (x^\kappa-x_{1,\eta}^\mathrm{a})} - 
					e^{-\mu(x^\lambda-\uline{x})}
					e^{-\rho_1^\eta(\uline{x}-x_{1,\eta}^\mathrm{a})}
				}{
					\mu -  \rho_1^\eta
				},
			\end{align*}
		}{
			\begin{align*}
				\sum_{\eta\in\{-,+\}}\int_{\uline{x}}^{x^\kappa} \mu
				\mathbf{u}_{1}^{\eta} e^{-\rho_1^\eta (y - x_{1,\eta}^\mathrm{a})} e^{\mu (y - x^\lambda)}\,\rmd y
				&=
				\sum_{\eta\in\{-,+\}} 
				\mu\mathbf{u}_{1}^{\eta}
				\frac{
					e^{-\mu(x^\lambda-x^\kappa)}e^{-\rho_1^\eta (x^\kappa-x_{1,\eta}^\mathrm{a})} - 
					e^{-\mu(x^\lambda-\uline{x})}
					e^{-\rho_1^\eta(\uline{x}-x_{1,\eta}^\mathrm{a})}
				}{
					\mu -  \rho_1^\eta
				}.
			\end{align*}
		}
		The second term, after repeated integration by parts, becomes
		\begin{align*}
			\int_{\uline{x}}^{x^\kappa} 
			\mu \lrp{\frac{\mu}{a_{1,1}}\gamma(\uline{x},\mathbf{u}_{1}^{\pm}) + q_1(y)} e^{\mu (y - x^\lambda)}\,\rmd y 
			&=
			\mu e^{-\mu x^\lambda}\lrp{e^{\mu x^\kappa}Q_1(x^\kappa) - e^{\mu \uline{x}}Q_1(\uline{x})},
		\end{align*}
		where
		\begin{align*}
			Q_1(y) &=
			c_{1,2}\lrp{\frac{y^2}{\mu}-\frac{2y}{\mu^2}+\frac{2}{\mu^3}}
			+ c_{1,1}\lrp{\frac{y}{\mu}-\frac{1}{\mu^2}}
			+ \lrp{d_{1,0} + \frac{\mu}{a_{1,1}}\gamma(\uline{x},\mathbf{u}_{1}^{\pm})}\,\frac{1}{\mu}.    
		\end{align*}
		
		\item \emph{$\mathcal{I}_2$}: Similarly,
		\begin{align*}
			\mathcal{I}_2 
			&=
			\sum_{\eta\in\{-,+\}}\int_{x^\kappa}^{x^\lambda} \mu \mathbf{u}_{2}^{\eta} e^{-\rho_2^\eta (y - x_{2,\eta}^\mathrm{a})} e^{\mu (y - x^\lambda)}\,\rmd y
			+ \int_{x^\kappa}^{x^\lambda} \mu
			\lrp{\frac{\mu}{a_{2,1}}\gamma(\uline{x},\mathbf{u}_{1}^{\pm}) + q_2(y)} e^{\mu (y - x^\lambda)}\,\rmd y.
		\end{align*}
		The exponential part becomes
		\details{
			\begin{align*}
				\sum_{\eta\in\{-,+\}}\int_{x^\kappa}^{x^\lambda} 
				\mu\mathbf{u}_{2}^{\eta} e^{-\rho_2^\eta (y - x_{2,\eta}^\mathrm{a})} e^{\mu (y - x^\lambda)}\,\rmd y
				&=
				\sum_{\eta\in\{-,+\}} 
				\mu\mathbf{u}_{2}^{\eta}
				\frac{
					e^{-\rho_2^\eta (x^\lambda-x_{2,\eta}^\mathrm{a})} - 
					e^{-\mu(x^\lambda-x^\kappa)}e^{-\rho_2^\eta (x^\kappa -x_{2,\eta}^\mathrm{a})}
				}{
					\mu - \rho_2^\eta
				},
			\end{align*}
		}{
			\begin{align*}
				\sum_{\eta\in\{-,+\}}\int_{x^\kappa}^{x^\lambda} 
				\mu\mathbf{u}_{2}^{\eta} e^{-\rho_2^\eta (y - x_{2,\eta}^\mathrm{a})} e^{\mu (y - x^\lambda)}\,\rmd y
				&=
				\sum_{\eta\in\{-,+\}} 
				\mu\mathbf{u}_{2}^{\eta}
				\frac{
					e^{-\rho_2^\eta (x^\lambda-x_{2,\eta}^\mathrm{a})} - 
					e^{-\mu(x^\lambda-x^\kappa)}e^{-\rho_2^\eta (x^\kappa -x_{2,\eta}^\mathrm{a})}
				}{
					\mu - \rho_2^\eta
				}.
			\end{align*}
		}
		The polynomial part is
		\begin{align*}
			\int_{x^\kappa}^{x^\lambda} \mu
			\lrp{\frac{\mu}{a_{2,1}}\gamma(\uline{x},\mathbf{u}_{1}^{\pm}) + q_2(y)} e^{\mu (y - x^\lambda)}\,\rmd y
			&=
			\mu e^{-\mu x^\lambda}\lrp{e^{\mu x^\lambda}Q_2(x^\lambda) - e^{\mu x^\kappa}Q_2(x^\kappa)},
		\end{align*}
		where
		\begin{align*}
			Q_2(y) &=
			c_{2,2}\lrp{\frac{y^2}{\mu}-\frac{2y}{\mu^2}+\frac{2}{\mu^3}}
			+ c_{2,1}\lrp{\frac{y}{\mu}-\frac{1}{\mu^2}}
			+ \lrp{d_{2,0} + \frac{\mu}{a_{2,1}}\gamma(\uline{x},\mathbf{u}_{1}^{\pm})}\,\frac{1}{\mu}.    
		\end{align*}
	\end{itemize}
	
	Combining the three components, \eqref{eq:xLa_def_stable} becomes
	\begin{align}
		0 = (\intOp H)(x^\lambda)
		&= -c_U e^{\mu(\uline{x} - x^\lambda)} \label{eq:xLa_def_explicit_stable} 
		\\
		&\hspace{0.4cm} + 
		\sum_{\eta\in\{-,+\}} 
		\mu\mathbf{u}_{1}^{\eta}
		\frac{
			e^{-\mu(x^\lambda-x^\kappa)}e^{-\rho_1^\eta (x^\kappa-x_{1,\eta}^\mathrm{a})} - 
			e^{-\mu(x^\lambda-\uline{x})}
			e^{-\rho_1^\eta(\uline{x}-x_{1,\eta}^\mathrm{a})}
		}{
			\mu -  \rho_1^\eta
		} \nonumber 
		\\
		&\hspace{0.4cm} 
		+ 
		\sum_{\eta\in\{-,+\}} 
		\mu\mathbf{u}_{2}^{\eta}
		\frac{
			e^{-\rho_2^\eta (x^\lambda-x_{2,\eta}^\mathrm{a})} - 
			e^{-\mu(x^\lambda-x^\kappa)}e^{-\rho_2^\eta (x^\kappa -x_{2,\eta}^\mathrm{a})}
		}{
			\mu - \rho_2^\eta
		} \nonumber 
		\\
		&\hspace{0.4cm} 
		+ \mu e^{-\mu x^\lambda}\lrp{
			(e^{\mu x^\kappa}Q_1(x^\kappa) - e^{\mu \uline{x}}Q_1(\uline{x}))
			+
			(e^{\mu x^\lambda}Q_2(x^\lambda) - e^{\mu x^\kappa}Q_2(x^\kappa))
		}. \nonumber
	\end{align}
	
	We are implicitly assuming $\mu \neq \rho_i^\eta$ in \eqref{eq:xLa_def_explicit_stable}.
	If this fails, the corresponding formula is obtained by taking the limit
	as $\mu\to \rho_i^\eta$.
	
	Therefore,
	\[
	\mathbf{u}_{2}^{-}m_{2,-}^{(2)}
	+
	\mathbf{u}_{2}^{+}m_{2,+}^{(2)} 
	= 
	\beta_2^{(2)}
	\]
	with
	\begin{align}\label{eq:coeff2_linear_system_2_stable}
		m_{2,-}^{(2)} 
		&= 
		\mu\frac{
			e^{-\rho_2^- (x^\lambda-x_{2,-}^\mathrm{a})} - 
			e^{-\mu(x^\lambda-x^\kappa)}e^{-\rho_2^- (x^\kappa -x_{2,-}^\mathrm{a})}
		}{
			\mu - \rho_2^-
		}, \quad
		m_{2,+}^{(2)} 
		= 
		\mu\frac{
			e^{-\rho_2^+ (x^\lambda-x_{2,+}^\mathrm{a})} - 
			e^{-\mu(x^\lambda-x^\kappa)}e^{-\rho_2^+ (x^\kappa -x_{2,+}^\mathrm{a})}
		}{
			\mu - \rho_2^+
		},
	\end{align}
	and
	\begin{align}\label{eq:linear_term2_linear_system_2_stable}
		\beta_2^{(2)} 
		&= 
		c_U e^{\mu(\uline{x} - x^\lambda)} 
		- 
		\sum_{\eta\in\{-,+\}} 
		\mu\mathbf{u}_{1}^{\eta}e^{-\mu(x^\lambda-x^\kappa)}
		\frac{
			e^{-\rho_1^\eta (x^\kappa-x_{1,\eta}^\mathrm{a})} - 
			e^{-\mu(x^\kappa-\uline{x})}
			e^{-\rho_1^\eta(\uline{x}-x_{1,\eta}^\mathrm{a})}
		}{
			\mu -  \rho_1^\eta
		} \\
		&\hspace{0.4cm} - 
		\mu e^{-\mu x^\lambda}\lrp{
			\lrp{e^{\mu x^\kappa}Q_1(x^\kappa) - e^{\mu \uline{x}}Q_1(\uline{x})}
			+
			\lrp{e^{\mu x^\lambda}Q_2(x^\lambda) - e^{\mu x^\kappa}Q_2(x^\kappa)}
		}. \nonumber
	\end{align}
	
	In matrix form,
	\begin{align}\label{eq:linear_system_2_stable}
		\begin{pmatrix}
			m_{1,-}^{(2)} & m_{1,+}^{(2)} \\
			m_{2,-}^{(2)} & m_{2,+}^{(2)}
		\end{pmatrix}
		\begin{pmatrix}
			\mathbf{u}_{2}^{-} \\ \mathbf{u}_{2}^{+}
		\end{pmatrix} 
		=
		\begin{pmatrix}
			\beta_{1}^{(2)} \\ \beta_{2}^{(2)}
		\end{pmatrix}\quad 
		\lrp{\text{ resp. } 
			\begin{pmatrix}
				\alpha_{1,-}^{(2)} & \alpha_{1,+}^{(2)} \\
				\alpha_{2,-}^{(2)} & \alpha_{2,+}^{(2)}
			\end{pmatrix}
			\begin{pmatrix}
				\mathbf{c}_{2}^{-} \\ \mathbf{c}_{2}^{+}
			\end{pmatrix} 
			=
			\begin{pmatrix}
				\beta_{1}^{(2)} \\ \beta_{2}^{(2)}
		\end{pmatrix}},
	\end{align}
	where $\alpha_{i,\pm}^{(2)} = e^{-\rho_2^\pm x_{2,\pm}^\mathrm{a}}m_{i,\pm}^{(2)}$,
	$i=1,2$.
	
	%-----%
	\paragraph{Finding the coefficients $\mathbf{u}_{3}^{\pm}$.}
	\label{par:solving_c3_coeff}\ \\
	\vspace{-0.3cm}
	%-----%
	
	\noindent Using \eqref{eq:FBP_candidate_upper_stopping} we obtain
	\begin{align}\label{eq:C0_at_xU}
		c_D &= H_3(\oline{x}) 
		= 
		\mathbf{u}_{3}^{-} e^{-\rho_3^- (\oline{x} - x_{3,-}^\mathrm{a})} 
		+ 
		\mathbf{u}_{3}^{+} e^{-\rho_3^+ (\oline{x} - x_{3,+}^\mathrm{a})} 
		+ \frac{\mu}{a_{3,1}}\gamma(\uline{x},\mathbf{u}_{1}^{\pm}) + q_3(\oline{x}).
	\end{align}
	That is,
	\[
	\mathbf{u}_{3}^{-} m_{1,-}^{(3)} 
	+ \mathbf{u}_{3}^{+} m_{1,+}^{(3)}
	= \beta_{1}^{(3)}
	\]
	for
	\begin{align*}
		m_{1,-}^{(3)} = e^{-\rho_3^- (\oline{x} - x_{3,-}^\mathrm{a})}, \quad 
		m_{1,+}^{(3)} = e^{-\rho_3^+ (\oline{x} - x_{3,+}^\mathrm{a})}, \quad
		\beta_{1}^{(3)} = c_D - \frac{\mu}{a_{3,1}}\gamma(\uline{x},\mathbf{u}_{1}^{\pm}) - q_3(\oline{x}).
	\end{align*}
	
	On the other hand, by \eqref{eq:FBP_example_candidate_instantaneous_stop},
	\begin{align*}
		&\mathbf{u}_{2}^{-} e^{-\rho_2^- (x^\lambda - x_{2,-}^\mathrm{a})} + \mathbf{u}_{2}^{+} e^{-\rho_2^+ (x^\lambda - x_{2,+}^\mathrm{a})} + \frac{\mu}{a_{2,1}}\gamma(\uline{x},\mathbf{u}_{1}^{\pm}) + q_2(x^\lambda) \\
		&= H_2(x^\lambda) 
		= H_3(x^\lambda) \\
		&= \mathbf{u}_{3}^{-} e^{-\rho_3^- (x^\lambda - x_{3,-}^\mathrm{a})} + \mathbf{u}_{3}^{+} e^{-\rho_3^+ (x^\lambda - x_{3,+}^\mathrm{a})} + \frac{\mu}{a_{3,1}}\gamma(\uline{x},\mathbf{u}_{1}^{\pm}) + q_3(x^\lambda),
	\end{align*}
	which gives
	\[
	\mathbf{u}_{3}^{-} m_{2,-}^{(3)} + \mathbf{u}_{3}^{+} m_{2,+}^{(3)} = \beta_{2}^{(3)}
	\]
	for
	\begin{align*}
		m_{2,-}^{(3)} &= e^{-\rho_3^- (x^\lambda - x_{3,-}^\mathrm{a})}, \quad
		m_{2,+}^{(3)} = e^{-\rho_3^+ (x^\lambda - x_{3,+}^\mathrm{a})}, \\
		\beta_2^{(3)} &= \mathbf{u}_{2}^{-} e^{-\rho_2^- (x^\lambda - x_{2,-}^\mathrm{a})} + \mathbf{u}_{2}^{+} e^{-\rho_2^+ (x^\lambda - x_{2,+}^\mathrm{a})} + \mu\gamma(\uline{x},\mathbf{u}_{1}^{\pm})\lrp{1/a_{2,1} - 1/a_{3,1}} + q_2(x^\lambda) - q_3(x^\lambda).
	\end{align*}
	
	Therefore, $\mathbf{u}_{3}^{\pm}$ solve
	\begin{align}\label{eq:linear_system_3_stable}
		\begin{pmatrix}
			m_{1,-}^{(3)} & m_{1,+}^{(3)} \\
			m_{2,-}^{(3)} & m_{2,+}^{(3)}
		\end{pmatrix}
		\begin{pmatrix}
			\mathbf{u}_{3}^{-} \\ \mathbf{u}_{3}^{+}
		\end{pmatrix} 
		=
		\begin{pmatrix}
			\beta_{1}^{(3)} \\ \beta_{2}^{(3)}
		\end{pmatrix}\quad
		\lrp{\text{resp.}
			\begin{pmatrix}
				\alpha_{1,-}^{(3)} & \alpha_{1,+}^{(3)} \\
				\alpha_{2,-}^{(3)} & \alpha_{2,+}^{(3)}
			\end{pmatrix}
			\begin{pmatrix}
				\mathbf{c}_{3}^{-} \\ \mathbf{c}_{3}^{+}
			\end{pmatrix} 
			=
			\begin{pmatrix}
				\beta_{1}^{(3)} \\ \beta_{2}^{(3)}
		\end{pmatrix}},
	\end{align}
	where $\alpha_{i,\pm}^{(3)} = e^{-\rho_3^\pm x_{3,\pm}^\mathrm{a}}m_{i,\pm}^{(3)}$,
	$i=1,2$.
	
	%-----%
	\paragraph{Uniqueness of the solution of the boundary-value problem.}\ \\
	\vspace{-0.3cm}
	%-----%
	
	\noindent The next lemma proves that the three linear systems above are nonsingular,
	and therefore uniquely determine the coefficients $\mathbf{u}_{i}^{\pm}$.
	
	\begin{lemma}[Unique solution of the linear systems]\label{lm:uniquenes_solution_linear_systems}
		The linear systems \eqref{eq:linear_system_1_stable},
		\eqref{eq:linear_system_2_stable}, and \eqref{eq:linear_system_3_stable}
		each have a unique solution.
	\end{lemma}
	
	%-------------------------------------------------%
	\subsubsection{Determination of the free boundary}
	\label{sec:smooth_fit}
	%-------------------------------------------------%
	
	It remains to determine the reflecting barriers and ambiguity thresholds
	that satisfy the smooth-fit conditions \eqref{FBP_candidte_example_smooth-fit}.
	To this end, define ${\bf x} = (\uline{x},x^\kappa,x^\lambda,\oline{x})$ 
	% \[
	% {\bf x} = (\uline{x},x^\kappa,x^\lambda,\oline{x})
	% \in
	% \left\{
	% {\bf x}\in\R^4 :
	% \uline{x} < x^\kappa < x^\lambda < \oline{x}
	% \right\},
	% \]
	and introduce the residual function
	\begin{align}\label{eq:C1-jump-residual}
		R(\mathbf{x})=\big(H_1'(\uline{x}),\;
		H_3'(\oline{x}; {\bf x}),\;
		H_1'(x^\kappa; {\bf x})-H_2'(x^\kappa; {\bf x}),\;
		H_2'(x^\lambda; {\bf x})-H_3'(x^\lambda; {\bf x})
		\big),
	\end{align}
	where
	\begin{align*}
		H_i'(x; {\bf x}) 
		=
		-\mathbf{u}_{i}^{-}({\bf x}) \rho_i^-e^{-\rho_i^- (x - x_{i,-}^\mathrm{a}({\bf x}))}
		- \mathbf{u}_{i}^{+}({\bf x}) \rho_i^+e^{-\rho_i^+ (x - x_{i,+}^\mathrm{a}({\bf x}))}
		+ q_i'(x),
	\end{align*}
	and where the anchors $x_{i,-}^\mathrm{a}({\bf x})$ are defined in \eqref{eq:anchors},
	while the coefficients $\mathbf{u}_{i}^{\pm}({\bf x})$ solve
	\eqref{eq:linear_system_1_stable}, \eqref{eq:linear_system_2_stable},
	and \eqref{eq:linear_system_3_stable}.
	
	In this way, the roots of the residual function $R$ determine the set of free-boundary
	points. Because these four smooth-fit equations are highly nonlinear,
	an explicit solution is generally not available.
	Instead, they can be solved numerically using root-finding methods.
	In Section~\ref{sec:numerics} we adopt a Broyden--Newton procedure to solve this
	system and thereby obtain the optimal thresholds.

	%%%%%%%%%%%
	
	%-------------------------------------------------%
	\section{Numerical results}\label{sec:numerics}
	%-------------------------------------------------%
	
	In this section we study the qualitative impact of model ambiguity and jump risk on the optimal long-run regulation policy. The analytical results of the previous sections reduce the problem to a nonlinear system determining the reflecting barriers, the ambiguity thresholds, and the ergodic value. Since this system does not admit a closed-form solution in general, numerical analysis is needed in order to understand how the optimal policy depends on the model parameters.
	
	We focus on the case of exponentially distributed negative jumps and quadratic running costs, for which the free-boundary problem admits the finite-dimensional representation derived in Section~\ref{sec:negative_exp_jumps}. Our objectives are twofold. First, we illustrate the structure of the optimal control and of the associated worst-case model distortions. Second, we perform comparative statics to assess how ambiguity, jump risk, volatility, and intervention costs affect the location of the reflecting barriers and ambiguity thresholds. We also quantify the value of robustness by comparing ergodic costs under ambiguity-aware and misspecified policies.
	
	To this end, we solve the nonlinear free-boundary system numerically by means of a two-stage root-search procedure equipped with stability safeguards. This allows us to identify the relevant regime and to obtain reliable solutions across a broad range of parameter values.
	
	Overall, the numerical results show that robustness has a systematic effect on the geometry of the optimal policy. Depending on the source of uncertainty, ambiguity may widen, shift, or skew the inaction region, thereby making the regulation rule more conservative. These effects are further shaped by jump risk and by asymmetries in intervention costs, which determine how the controller trades off the frequency and direction of future interventions.
	
	The code used to generate the numerical results is available at
	\url{https://github.com/aguazz/robust_ergodic-singular-control_jump-diffusion}.
	
	%-------------------------------------------------%
	\subsection{Numerical stability considerations}\label{sec:stability}
	%-------------------------------------------------%
	
	The numerical solution of the free-boundary system requires some care, because the linear systems defining the candidate functions $H_i$ may become ill-conditioned for certain parameter configurations. To ensure reliable computation, we incorporate several stability safeguards into the implementation.
	
	First, we monitor configurations in which some coefficients become nearly degenerate, including cases where $a_{i,1}$ is close to zero, characteristic roots nearly coincide, or jump and ODE roots become almost identical. In such cases, the solver either stops or issues a warning, depending on the severity of the instability. Second, we evaluate exponential terms using numerically stable transformations such as
	\[
	\mathrm{expm1}(x)=e^x-1,
	\qquad
	\mathrm{exprel}(x)=\frac{e^x-1}{x},
	\]
	thereby avoiding loss of precision when $x$ is close to zero. Third, the linear systems defining the coefficients of the piecewise ODE solutions are solved using pivoted $LU$ factorization with scaling, which improves numerical conditioning.
	
	These safeguards are particularly useful near parameter values where the model approaches a degenerate regime. They do not affect the qualitative conclusions of the numerical analysis, but they improve the robustness and reproducibility of the computations.
	
	%-------------------------------------------------%
	\subsection{The algorithm}
	%-------------------------------------------------%
	
	The numerical procedure has two nested components. For any candidate set of reflecting barriers and ambiguity thresholds, an \emph{inner step} solves the linear systems determining the piecewise ODE representation of the candidate solution and computes the residuals associated with the boundary and smooth-fit conditions. An \emph{outer step} then updates the thresholds by applying a root-search routine to these residuals.
	
	To preserve the natural ordering of the thresholds, the outer problem is solved in transformed variables based on logarithmic gap parametrization. This guarantees that the candidate thresholds remain ordered throughout the search and substantially improves numerical stability. In our implementation, the outer loop is solved by a quasi-Newton root-search method, while the inner loop computes the corresponding coefficients, candidate functions, and ergodic value.
	
	Algorithm~\ref{alg:two-stages_solver} summarizes the resulting two-stage procedure.
	
	\begin{algorithm}[ht]
		\caption{Two-stage solver}\label{alg:two-stages_solver}
		\begin{algorithmic}[1]
			\Require Parameters $p=(b,\delta,r,\varepsilon,\sigma,\mu,c_U,c_D)$; initial thresholds ${\bf x}_0=(\uline{x}_0,x^\kappa_0,x^\lambda_0,\oline{x}_0)$; tolerance $tol>0$.
			\Ensure Optimal thresholds ${\bf x}^*$, value $\gamma_*$, coefficients $\{\mathbf{u}_{i,*}^{\pm}\}_{i=1}^3$, and $\{H_{i,*}\}_{i=1}^3$.
			
			\Function{InnerSolver}{$\mathbf{x}=(\uline{x},x^\kappa,x^\lambda,\oline{x})$}
			\State \textbf{require} $(\uline{x}, x^\kappa, x^\lambda, \oline{x}) \in \R^4$ such that $\uline{x} < x^\kappa < \oline{x}$ and $x^\kappa < x^\lambda$.
			\If{$x^\lambda < \oline{x}$} \Comment{Regime 1}
			\State To obtain $\mathbf{u}_{1}^{\pm}$, solve \eqref{eq:linear_system_1_stable}; compute $\gamma$ using \eqref{eq:gamma_stable}; solve \eqref{eq:linear_system_2_stable} to obtain $\mathbf{u}_{2}^{\pm}$; solve \eqref{eq:linear_system_3_stable} to obtain $\mathbf{u}_{3}^{\pm}$.
			\State Compute $\{H_i\}_{i=1}^3$ and $\{H_i'\}_{i=1}^3$ using \eqref{eq:ODE_H_sol_gamma} and \eqref{eq:ODE_Hprime_sol_gamma}.
			\State Compute the residual vector 
			$
			R=\big(
			H_1'(\uline{x}),\;
			H_3'(\oline{x}),\;
			H_1'(x^\kappa)-H_2'(x^\kappa),\;
			H_2'(x^\lambda)-H_3'(x^\lambda)
			\big)
			$.
			\Else \Comment{Regime 2}
			\State Solve \eqref{eq:linear_system_1_stable} to obtain $\mathbf{u}_{1}^{\pm}$; compute $\gamma$ using \eqref{eq:gamma_stable}. 
			\State To obtain $\mathbf{u}_{2}^{\pm}$ in Regime 2, solve the linear system
			$$
			\begin{pmatrix}
				m_{1,-}^{(2)} & m_{1,+}^{(2)} \\
				m_{1,-}^{(3)} & m_{1,+}^{(3)}
			\end{pmatrix}
			\begin{pmatrix}
				\mathbf{u}_{2}^{-} \\ \mathbf{u}_{2}^{+}
			\end{pmatrix} 
			=
			\begin{pmatrix}
				\beta_{1}^{(2)} \\ \beta_{1}^{(3)}
			\end{pmatrix}\!.
			$$
			\State Compute $\{H_i\}_{i=1}^2$ and $\{H_i'\}_{i=1}^2$ using \eqref{eq:ODE_H_sol_gamma} and \eqref{eq:ODE_Hprime_sol_gamma}, and set $H_3 \equiv c_D$, $H_3' \equiv 0$, and $\bf u_3^- = \bf u_3^+ = 0$.
			\State Compute the residual vector 
			$
			R=\big(
			H_1'(\uline{x}),\;
			H_2'(\oline{x}),\;
			H_1'(x^\kappa)-H_2'(x^\kappa),\;
			\intOp H(x^\lambda)
			\big)
			$.
			\EndIf
			\State \Return $(\gamma,\{\mathbf{u}_{i}^{\pm}\}_{i=1}^3,\{H_i\}_{i=1}^3,\{H_i'\}_{i=1}^3, R)$.
			\EndFunction
			
			\Function{OuterSolver}{${\bf x}_0,\; tol$}
			\State Reparametrize gaps: $Z({\bf x}) = \big(z_1,z_2,z_3,z_4\big) = \big(\uline{x}, \log(x^\kappa-\uline{x}), \log(x^\lambda-x^\kappa), \log(\oline{x}-x^\kappa)\big)$.
			\State Define the inverse mapping $X({\bf z})$ to recover ordered thresholds.
			\State Find the root of $R \gets \textsc{InnerSolver}\big(X({\bf z})\big)$ by running a root-search method initialized at ${\bf z}_0 = Z({\bf x}_0)$.
			\State \Return $\big(\mathbf{x}^*,\gamma_*,\{\mathbf{u}_{i,*}^{\pm}\}_{i=1}^3,\{H_{i,*}\}_{i=1}^3\big)$.
			\EndFunction
		\end{algorithmic}
	\end{algorithm}
	
	%-------------------------------------------------%
	\subsection{Experiments}
	%-------------------------------------------------%
	
	We now use the numerical solution to study the qualitative impact of model ambiguity and jump risk on the optimal long-run regulation policy. The thresholds $\underline{x}$ and $\overline{x}$ determine the inaction region, while $x^\kappa$ and $x^\lambda$ identify the switching points at which the worst-case drift and worst-case jump intensity move between their extreme admissible values.
	
	Unless otherwise stated, the baseline parameters are
	\[
	b=0,\qquad \delta=1,\qquad r=1,\qquad \varepsilon=0.5,\qquad \sigma=1,\qquad \mu=1.
	\]
	To study the role of intervention-cost asymmetry, we consider the three cost configurations
	\[
	(c_U,c_D)\in\{(1,1),(2,1),(1,2)\}.
	\]
	When examining the direct effect of the intervention costs, we keep the same baseline parameters and also consider the drift values $b=-2$ and $b=2$, corresponding to downward- and upward-trending uncontrolled dynamics.
	
	We begin with a representative solution under the baseline specification. Figure~\ref{fig:solution} displays the numerically computed function $H$, the optimal reflecting barriers and ambiguity thresholds, and a sample path of the optimally controlled process together with the associated controls. The figure illustrates the reflecting-band structure of the optimal policy and shows how the ambiguity thresholds partition the state space into regions associated with different worst-case distortions.
	
	\begin{figure}[!ht]
		\captionsetup[sub]{labelformat=parens}
		\centering
		
		\begin{subfigure}[t]{0.47\textwidth}
			\centering
			\includegraphics[width=\linewidth]{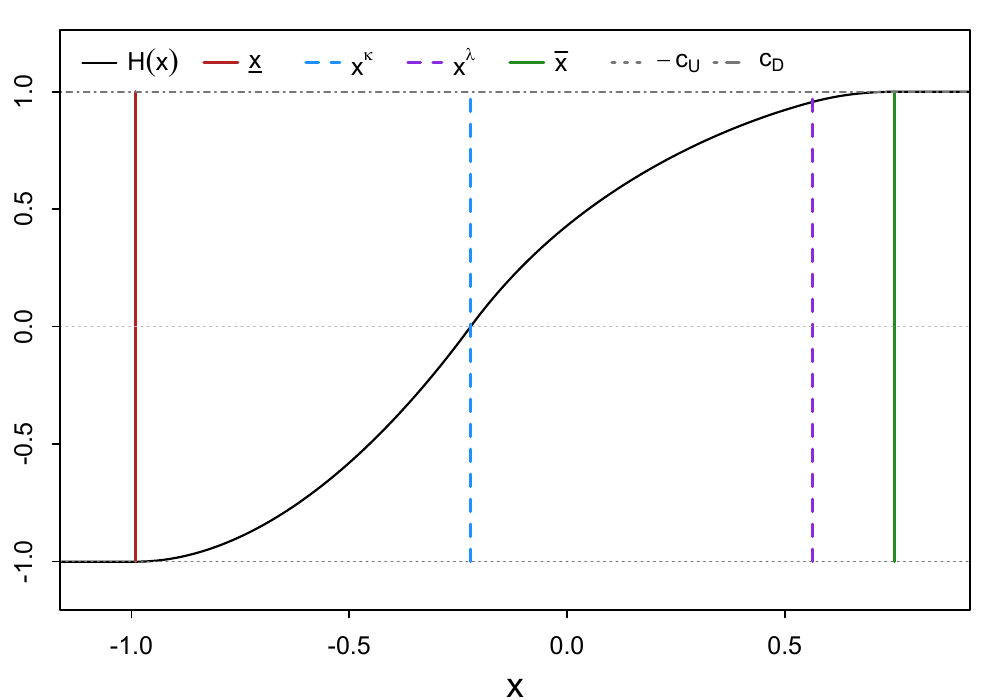}
			\caption{Plot of $H(x)$ as a function of $x$, with highlighted reflecting barriers and ambiguity thresholds.}
		\end{subfigure}\hspace{2pt}
		\begin{subfigure}[t]{0.47\textwidth}
			\centering
			\includegraphics[width=\linewidth]{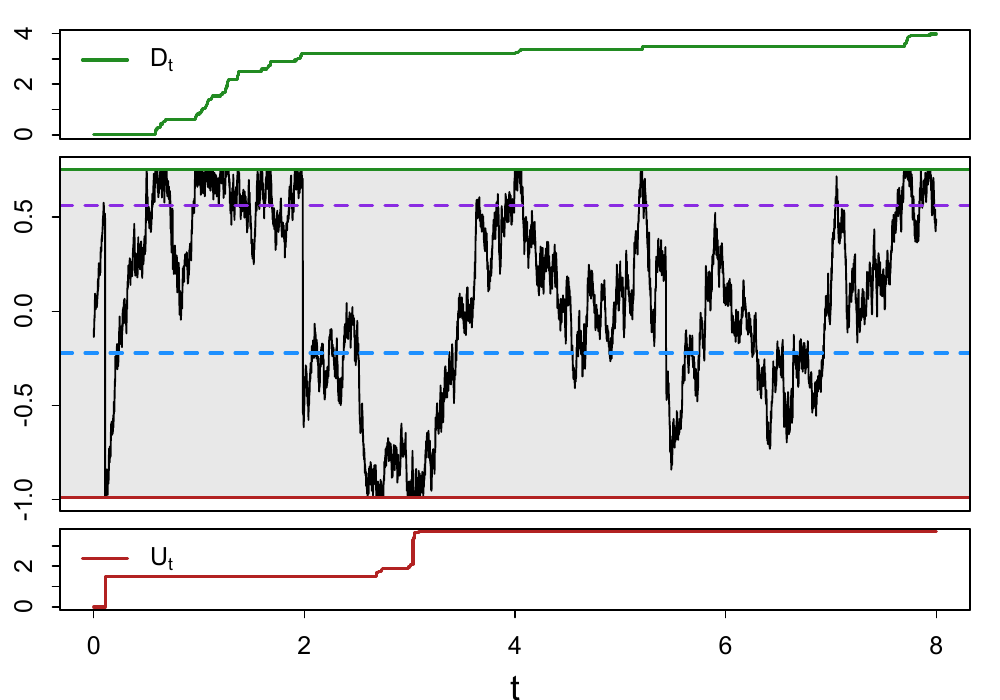}
			\caption{Sample path of the controlled process $\oline{X}$ and the controlling processes $D$ and $U$ as functions of time $t\in(0,8)$.}
		\end{subfigure}\hspace{2pt}
		
		\caption{\small Panel (a) displays the numerical solution for $H$ together with the ambiguity thresholds $x^\kappa$ and $x^\lambda$ and the reflecting barriers $\uline{x}$ and $\oline{x}$. Panel (b) shows a sample path of the optimally controlled process $\oline{X}_t$. The parameter values are $b=0$, $\delta=1$, $r=1$, $\varepsilon=0.5$, $\sigma=1$, $\mu=1$, $c_U=1$, and $c_D=1$.}
		\label{fig:solution}
	\end{figure}
	
	We next report comparative statics by varying one parameter at a time and recording the resulting reflecting barriers and ambiguity thresholds. Some curves exhibit small gaps. These correspond to isolated parameter values at which the solver activates a numerical stopping condition in order to avoid near-degenerate configurations. These cases do not affect the qualitative conclusions reported below.
	
	\paragraph{Effect of the drift parameter $\pmb{b}$.}
	Figure~\ref{fig:comparative_statistics_drift} shows that a larger positive drift pushes the uncontrolled system toward higher states, away from the running-cost minimizer near zero. 
	The controller reacts by adjusting the reflecting band in order to limit costly excursions. In our numerical experiments, a higher drift typically shifts the band downward, lowering the upper barrier and allowing more room on the lower side, since downward deviations are more likely to self-correct. 
	The ambiguity thresholds move accordingly, indicating that the worst-case model tends to amplify the upward pressure over a wider portion of the relevant state space. 
	The case of a negative drift exhibits the symmetric behavior. 
	When $b$ is close to zero, the system is easier to regulate and the inaction region remains relatively narrow. 
	Under asymmetric intervention costs, the reflecting band tilts toward the side where intervention is cheaper, reflecting the lower cost of correcting deviations in that direction.
	
	\begin{figure}[!ht]
		\captionsetup[sub]{labelformat=parens}
		\centering
		
		\begin{subfigure}[t]{0.32\textwidth}\centering
			\includegraphics[width=\linewidth]{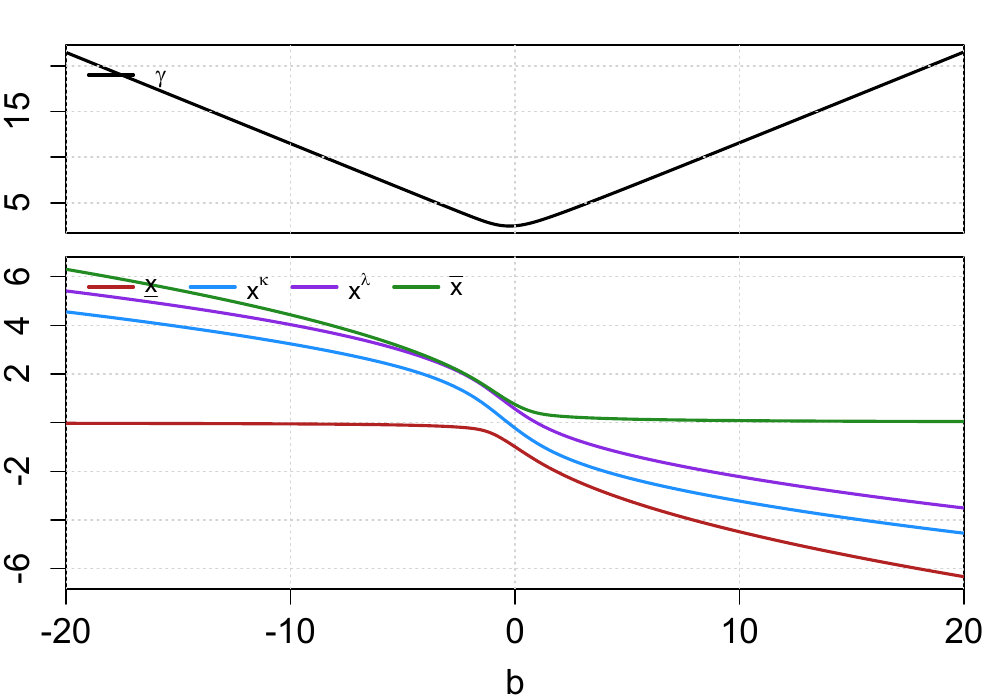}
			\caption{$c_U = 1$, $c_D = 1$}
		\end{subfigure}%
		\begin{subfigure}[t]{0.32\textwidth}\centering
			\includegraphics[width=\linewidth]{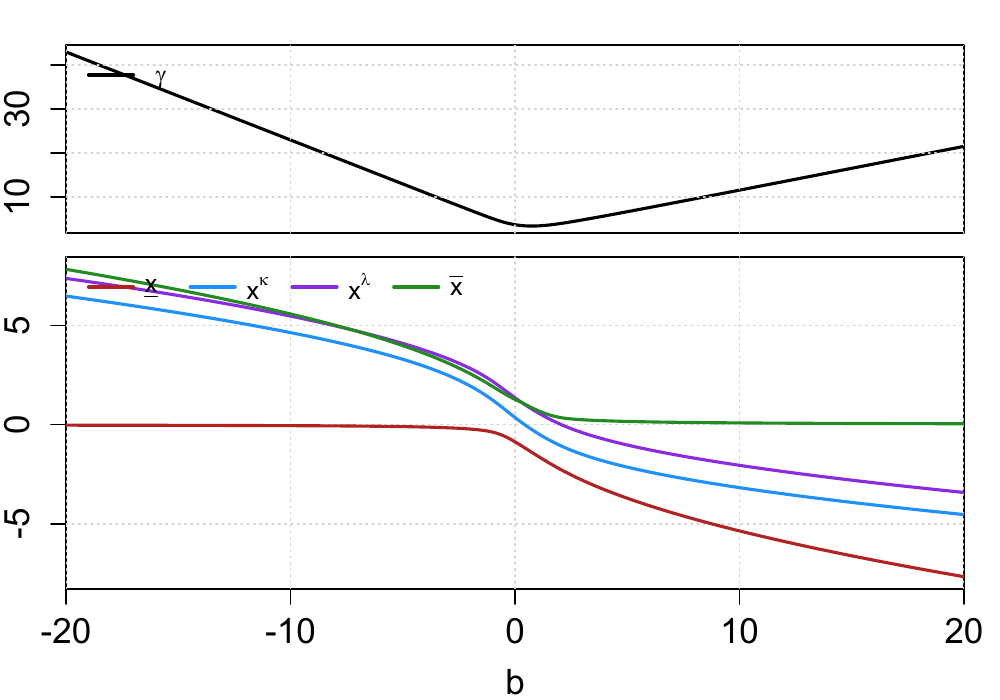}
			\caption{$c_U = 2$, $c_D = 1$}
		\end{subfigure}%
		\begin{subfigure}[t]{0.32\textwidth}\centering
			\includegraphics[width=\linewidth]{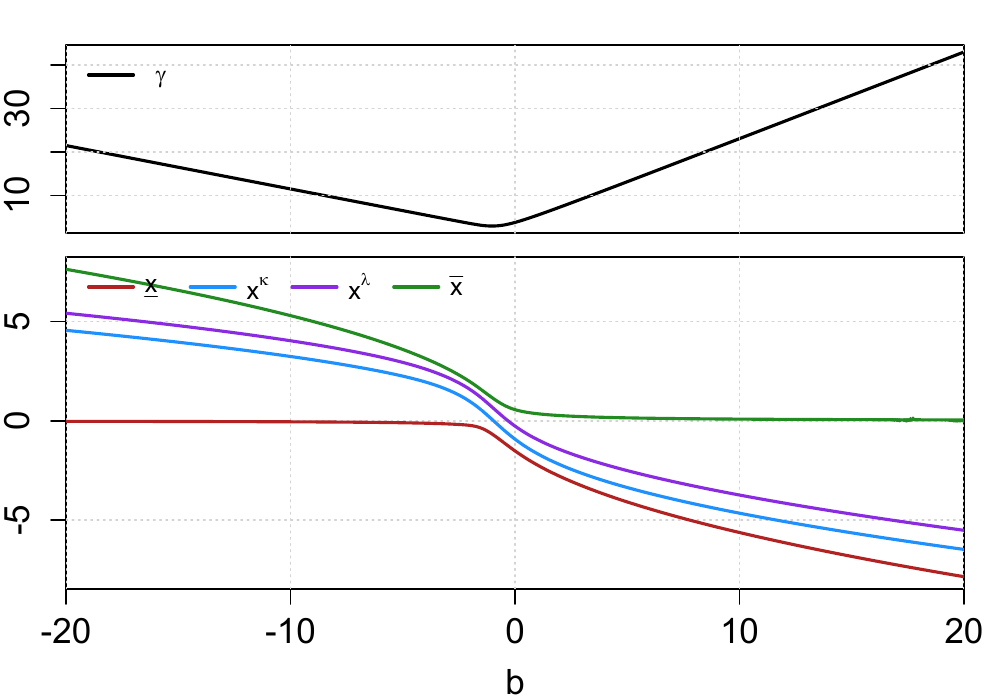}
			\caption{$c_U = 1$, $c_D = 2$}
		\end{subfigure}
		
		\caption{\small Effect of the drift parameter $b\in(-20,20)$ on the reflecting barriers and ambiguity thresholds.}
		\label{fig:comparative_statistics_drift}
	\end{figure}
	
	\paragraph{Effect of the drift-ambiguity parameter $\pmb{\delta}$.}
	Figure~\ref{fig:comparative_statistics_drift-ambiguity} shows that increasing $\delta$ enlarges the set of admissible drift distortions available to the adversary. 
	As a consequence, the controller adopts a more cautious policy and adjusts the reflecting band in order to reduce exposure to unfavorable drift realizations. 
	In the numerical experiments, this adjustment may involve a widening of the inaction region, a shift of the band, or a change in its asymmetry, depending on the configuration of the intervention costs. 
	Under symmetric intervention costs, the response is approximately balanced across the two sides of the band, whereas under asymmetric costs the adjustment is more pronounced on the side associated with the more expensive intervention. 
	The ambiguity threshold shifts accordingly, indicating that the worst-case drift distortion typically becomes relevant over a larger portion of the state space.
	
	\begin{figure}[!ht]
		\captionsetup[sub]{labelformat=parens}
		\centering
		
		\begin{subfigure}[t]{0.32\textwidth}\centering
			\includegraphics[width=\linewidth]{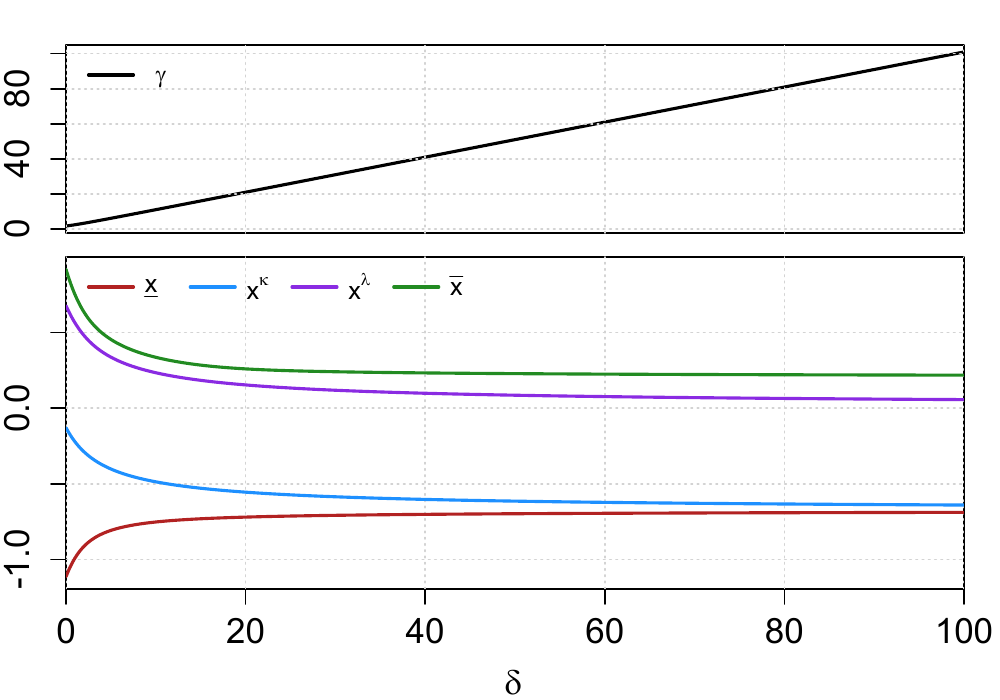}
			\caption{$c_U = 1$, $c_D = 1$}
		\end{subfigure}%
		\begin{subfigure}[t]{0.32\textwidth}\centering
			\includegraphics[width=\linewidth]{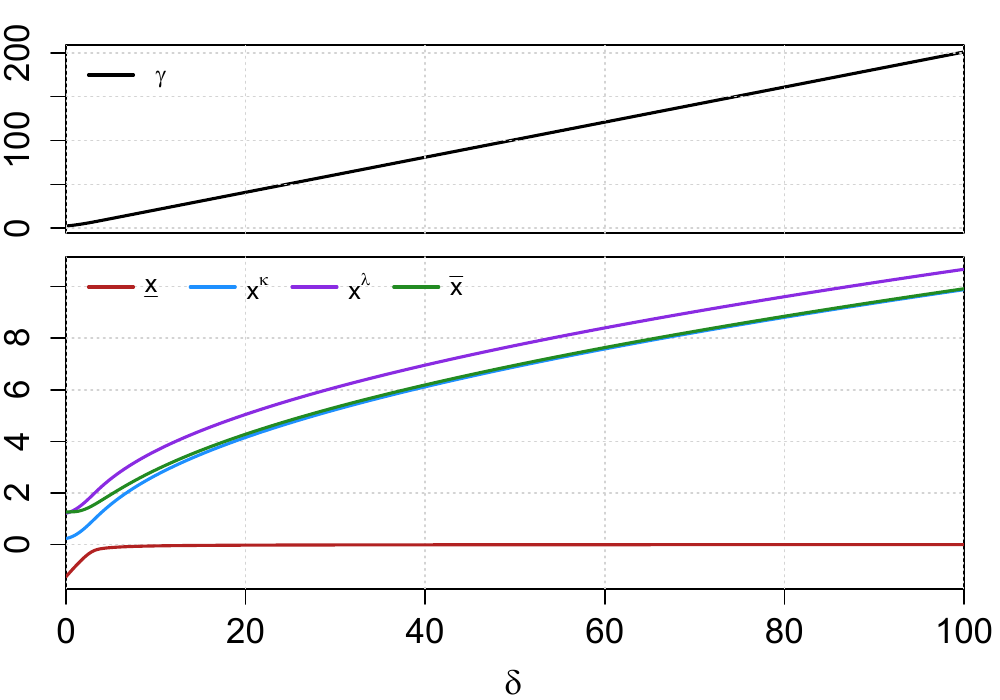}
			\caption{$c_U = 2$, $c_D = 1$}
		\end{subfigure}%
		\begin{subfigure}[t]{0.32\textwidth}\centering
			\includegraphics[width=\linewidth]{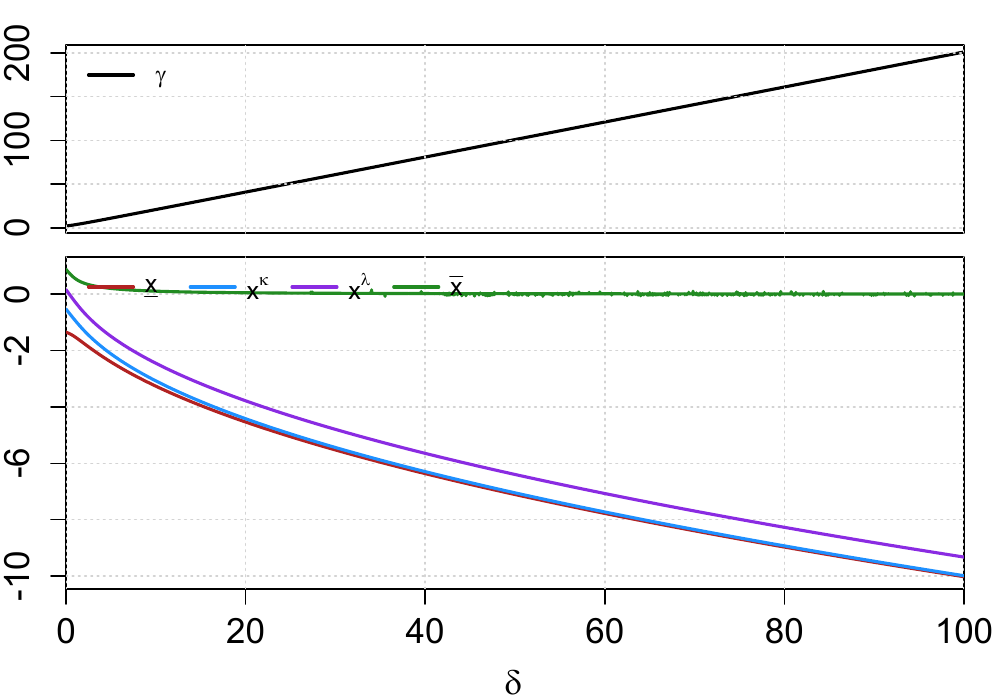}
			\caption{$c_U = 1$, $c_D = 2$}
		\end{subfigure}
		
		\caption{\small Effect of the drift-ambiguity parameter $\delta\in(0,100)$ on the reflecting barriers and ambiguity thresholds.}
		\label{fig:comparative_statistics_drift-ambiguity}
	\end{figure}
	
	\paragraph{Effect of the jump intensity $\pmb{r}$.}
	Figure~\ref{fig:comparative_statistics_intensity} highlights the effect of more frequent downward shocks. 
	When $r$ is small, jumps are rare and the state remains close to the running-cost minimizer, so a relatively tight reflecting band is optimal. 
	As $r$ increases, downward shocks occur more frequently, which changes the trade-off between immediate intervention and waiting for future shocks to move the state toward the desired region. 
	In the numerical experiments, the controller typically responds by widening the inaction region, allowing the system to absorb part of the adjustment through the jump dynamics rather than through costly control actions. 
	Under asymmetric intervention costs, the widening is tilted toward the side where adjustments are cheaper. 
	The ambiguity threshold shifts consistently with this change, indicating that the least-favorable regime tends to become relevant over a larger portion of the state space.
	
	\begin{figure}[!ht]
		\captionsetup[sub]{labelformat=parens}
		\centering
		
		\begin{subfigure}[t]{0.32\textwidth}\centering
			\includegraphics[width=\linewidth]{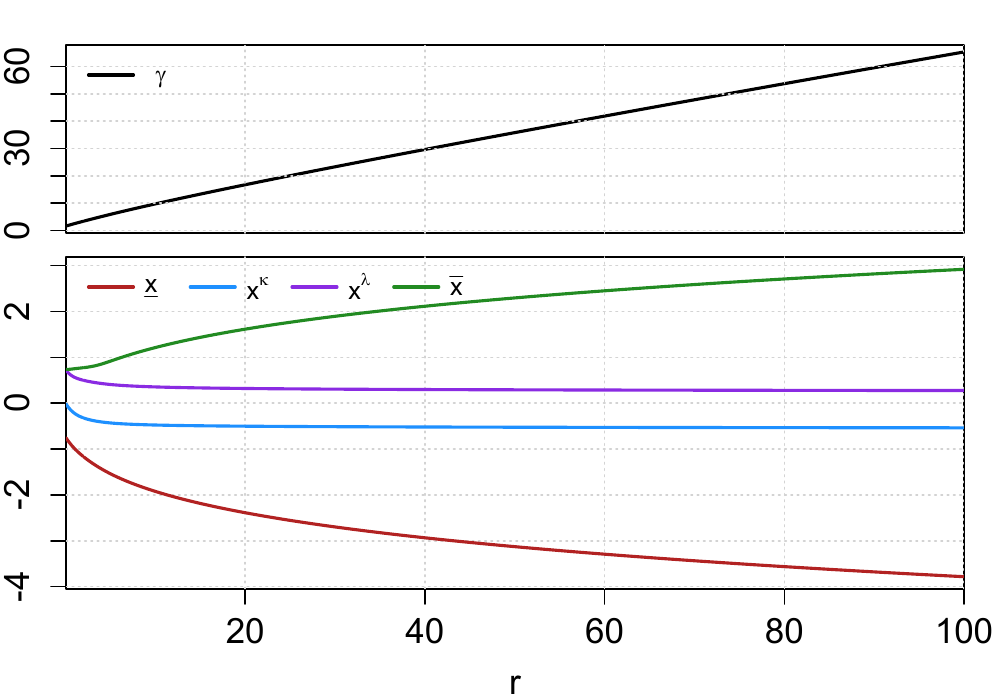}
			\caption{$c_U = 1$, $c_D = 1$}
		\end{subfigure}%
		\begin{subfigure}[t]{0.32\textwidth}\centering
			\includegraphics[width=\linewidth]{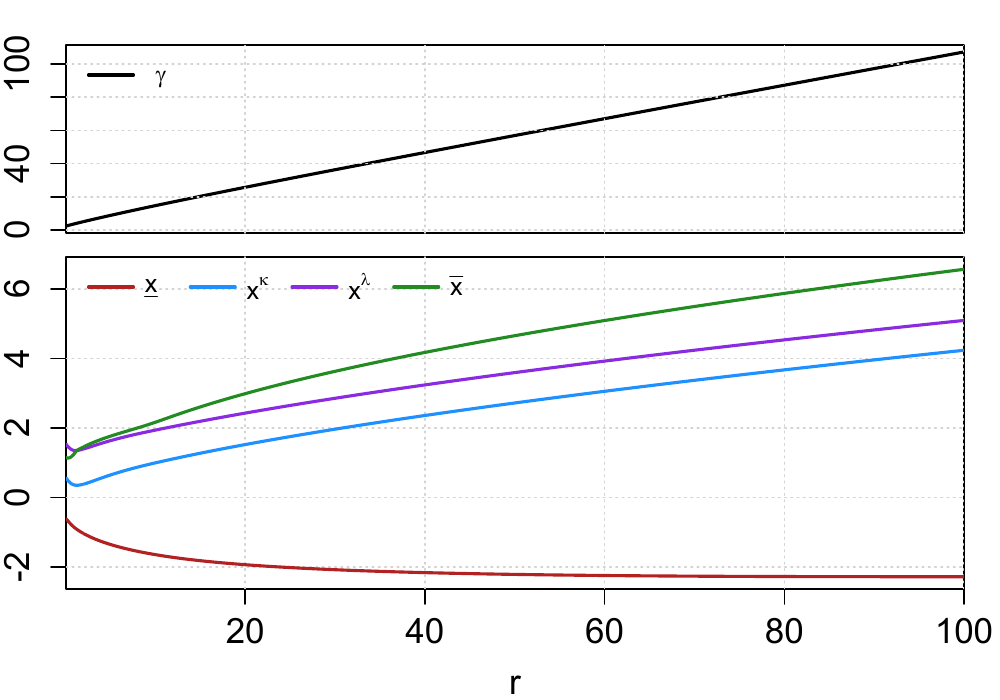}
			\caption{$c_U = 2$, $c_D = 1$}
		\end{subfigure}%
		\begin{subfigure}[t]{0.32\textwidth}\centering
			\includegraphics[width=\linewidth]{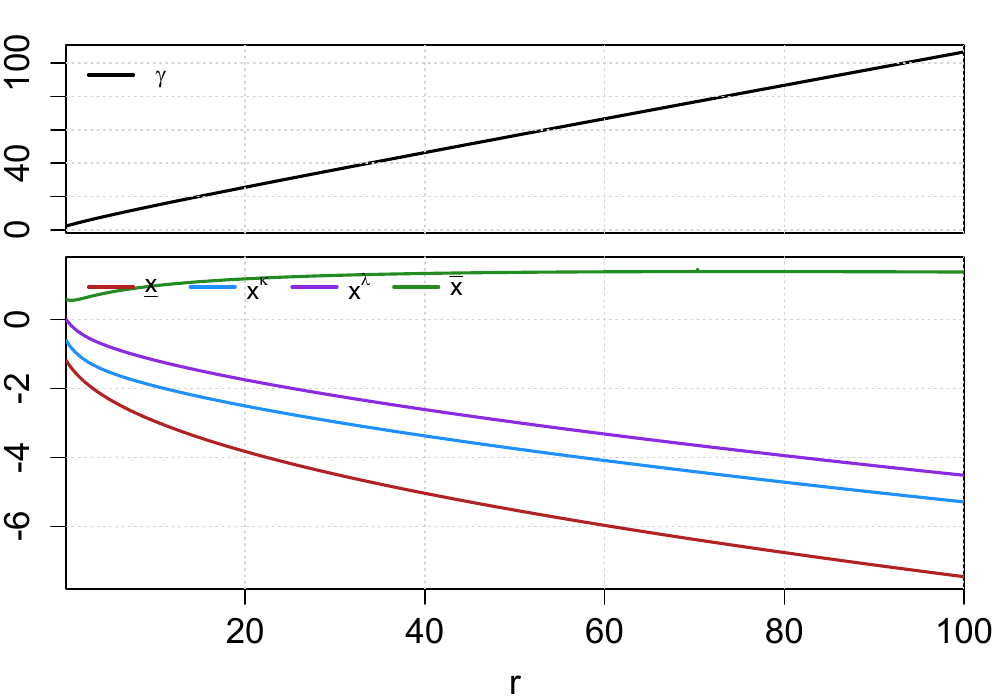}
			\caption{$c_U = 1$, $c_D = 2$}
		\end{subfigure}
		
		\caption{\small Effect of the jump intensity $r\in(0,100)$ on the reflecting barriers and ambiguity thresholds.}
		\label{fig:comparative_statistics_intensity}
	\end{figure}
	
	\paragraph{Effect of the intensity-ambiguity parameter $\pmb{\varepsilon}$.}
	Figure~\ref{fig:comparative_statistics_intensity-ambiguity} isolates the effect of ambiguity about the jump intensity. 
	As $\varepsilon$ increases, the admissible interval for the jump intensity widens, giving the adversary more scope to distort the frequency of downward shocks. 
	In response, the controller typically adopts a more conservative policy, enlarging or shifting the inaction region in order to reduce exposure to adverse jump realizations. 
	Under asymmetric intervention costs, the adjustment is tilted toward avoiding the more expensive intervention. 
	At the same time, the intensity threshold shifts accordingly, indicating that the high-intensity worst-case regime tends to apply over a larger portion of the relevant state space.
	
	\begin{figure}[!ht]
		\captionsetup[sub]{labelformat=parens}
		\centering
		
		\begin{subfigure}[t]{0.32\textwidth}\centering
			\includegraphics[width=\linewidth]{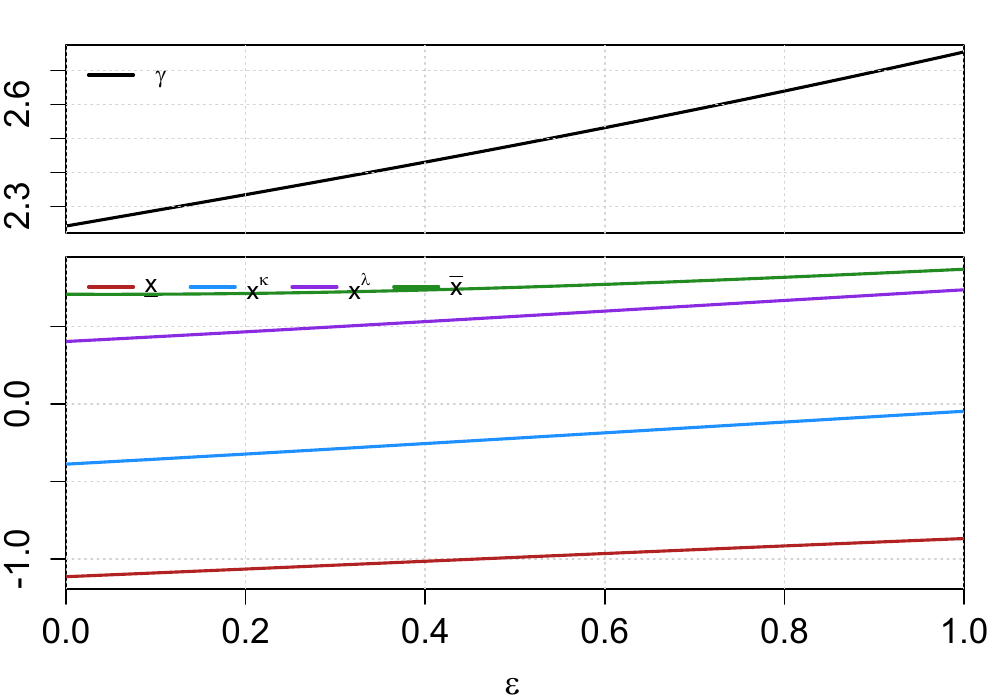}
			\caption{$c_U = 1$, $c_D = 1$}
		\end{subfigure}%
		\begin{subfigure}[t]{0.32\textwidth}\centering
			\includegraphics[width=\linewidth]{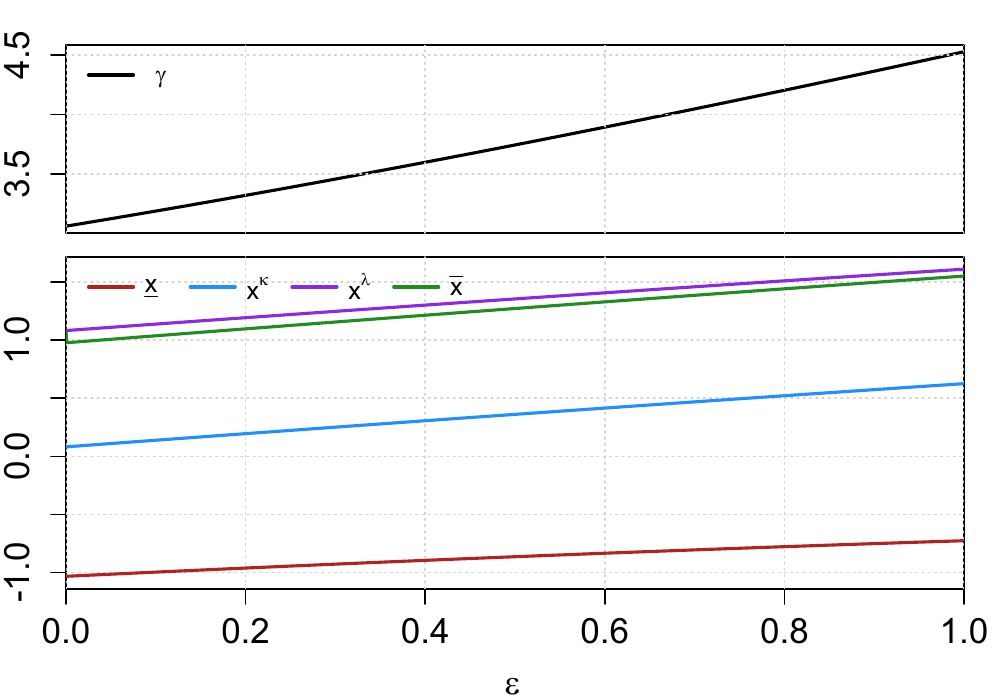}
			\caption{$c_U = 2$, $c_D = 1$}
		\end{subfigure}%
		\begin{subfigure}[t]{0.32\textwidth}\centering
			\includegraphics[width=\linewidth]{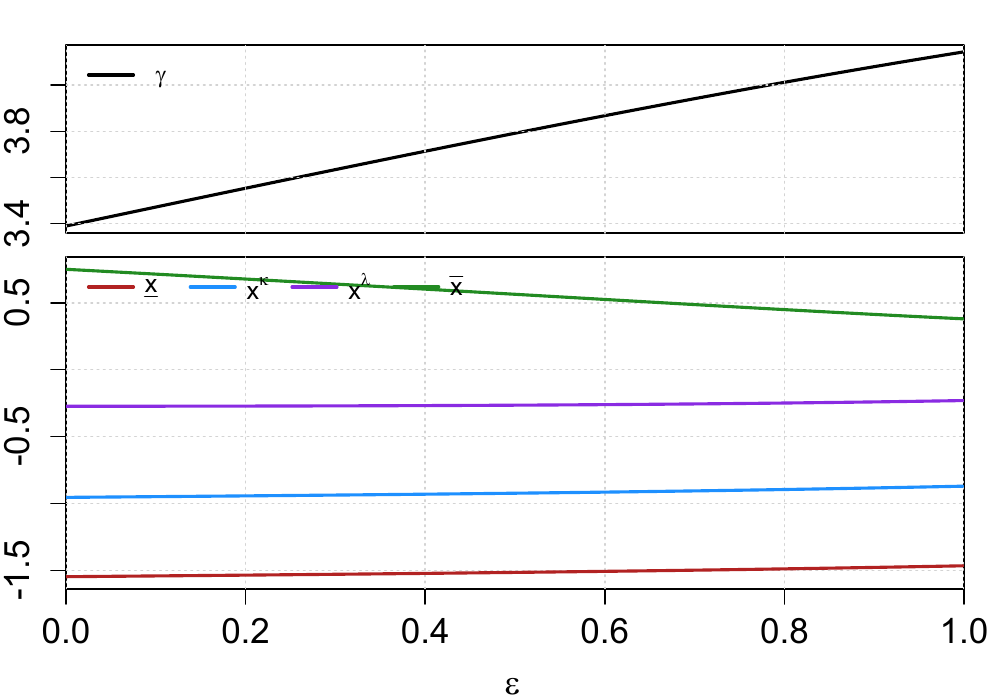}
			\caption{$c_U = 1$, $c_D = 2$}
		\end{subfigure}
		
		\caption{\small Effect of the intensity-ambiguity parameter $\varepsilon\in(0,1)$ on the reflecting barriers and ambiguity thresholds.}
		\label{fig:comparative_statistics_intensity-ambiguity}
	\end{figure}
	
	\paragraph{Effect of the volatility parameter $\pmb{\sigma}$.}
	Figure~\ref{fig:comparative_statistics_volatility} shows that, when $\sigma$ is small, diffusion fluctuations are modest and it is optimal to keep a narrow inaction region around the running-cost minimizer. 
	As $\sigma$ increases, Brownian variability generates more frequent excursions away from the low-cost region, and maintaining a narrow band would require more frequent interventions. 
	The optimal response is therefore typically to widen the reflecting band, approximately symmetrically when $c_U=c_D$, and to skew it toward the cheaper-action side when intervention costs are asymmetric. 
	The ambiguity thresholds shift accordingly, indicating that the worst-case model becomes relevant over a wider portion of the state space as volatility increases.
	
	\begin{figure}[!ht]
		\captionsetup[sub]{labelformat=parens}
		\centering
		
		\begin{subfigure}[t]{0.32\textwidth}\centering
			\includegraphics[width=\linewidth]{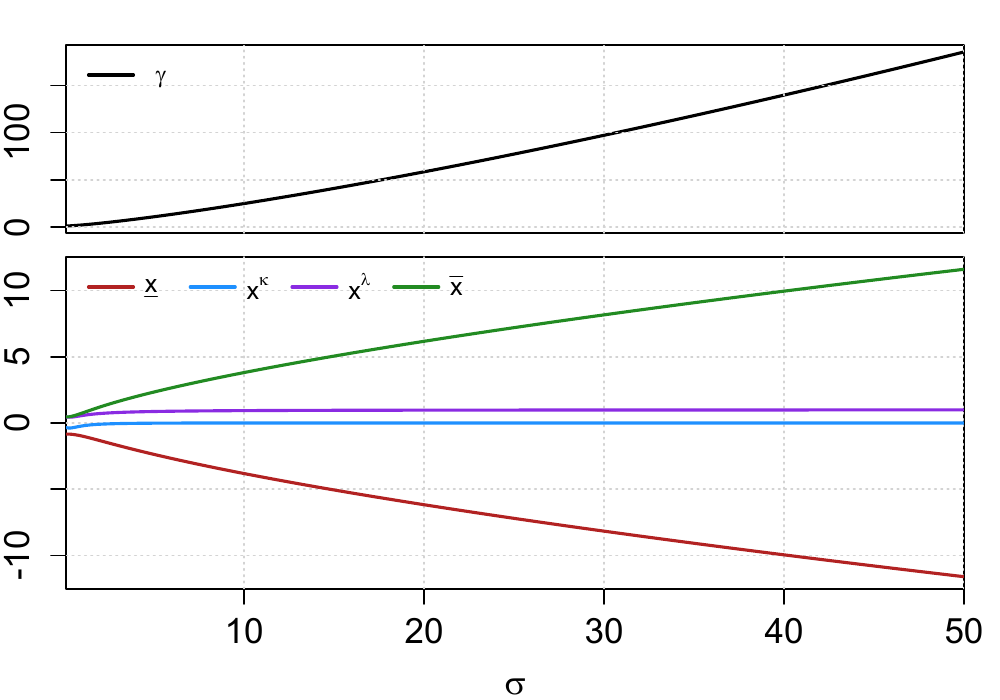}
			\caption{$c_U = 1$, $c_D = 1$}
		\end{subfigure}%
		\begin{subfigure}[t]{0.32\textwidth}\centering
			\includegraphics[width=\linewidth]{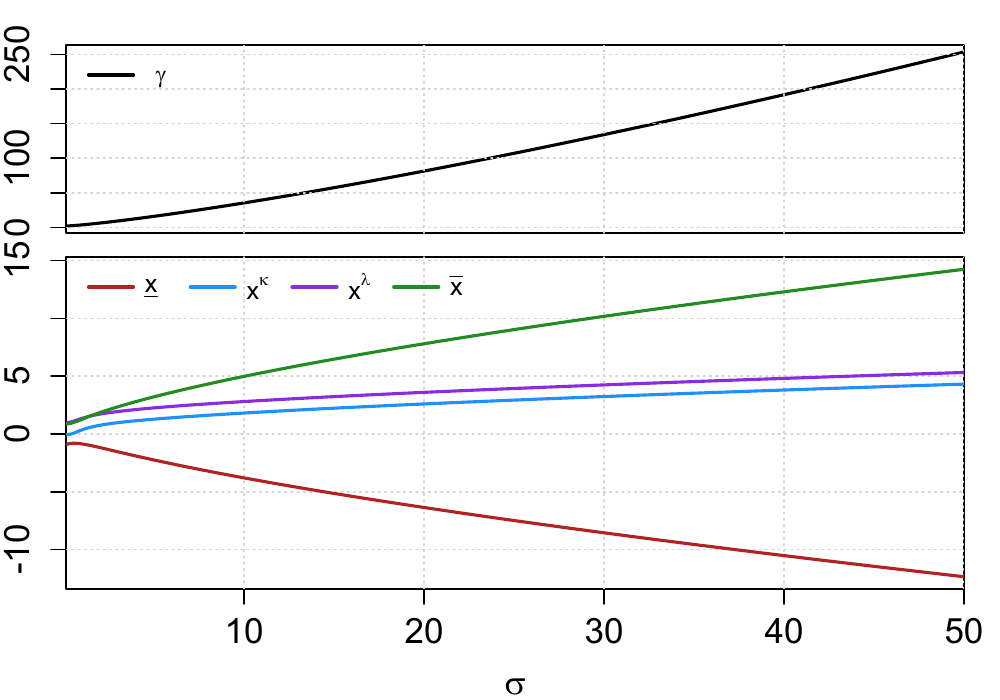}
			\caption{$c_U = 2$, $c_D = 1$}
		\end{subfigure}%
		\begin{subfigure}[t]{0.32\textwidth}\centering
			\includegraphics[width=\linewidth]{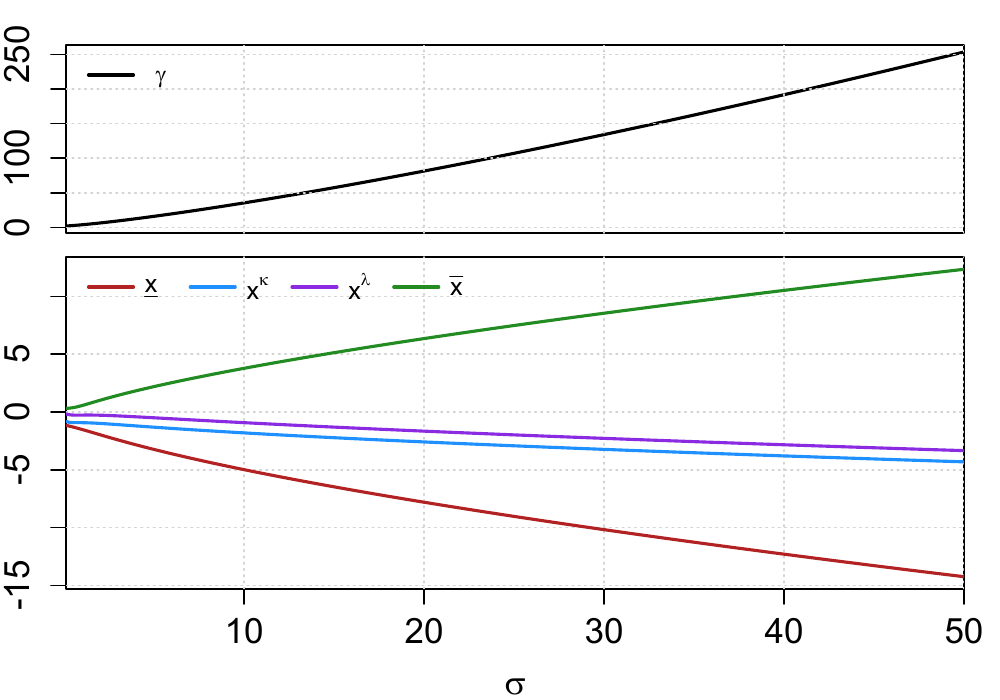}
			\caption{$c_U = 1$, $c_D = 2$}
		\end{subfigure}
		
		\caption{\small Effect of the volatility parameter $\sigma\in(0,50)$ on the reflecting barriers and ambiguity thresholds.}
		\label{fig:comparative_statistics_volatility}
	\end{figure}
	
	\paragraph{Effect of the mean jump size $\pmb{1/\mu=-\Esp[Y]}$.}
	Figure~\ref{fig:comparative_statistics_jump-mean} shows that, as the mean magnitude of downward jumps increases, individual shocks push the state further into low regions. 
	The controller reacts by allowing a wider inaction region, reducing the need for frequent and costly upward corrections after large shocks. 
	In the numerical experiments, the lower barrier typically moves downward, while the upper barrier may move upward, since high states are more likely to be corrected by future downward jumps. 
	Under asymmetric intervention costs, the widening is tilted toward the cheaper-action side. 
	The ambiguity threshold shifts accordingly, indicating that the worst-case jump intensity tends to play a larger role when shocks become more severe.
	
	\begin{figure}[!ht]
		\captionsetup[sub]{labelformat=parens}
		\centering
		
		\begin{subfigure}[t]{0.32\textwidth}\centering
			\includegraphics[width=\linewidth]{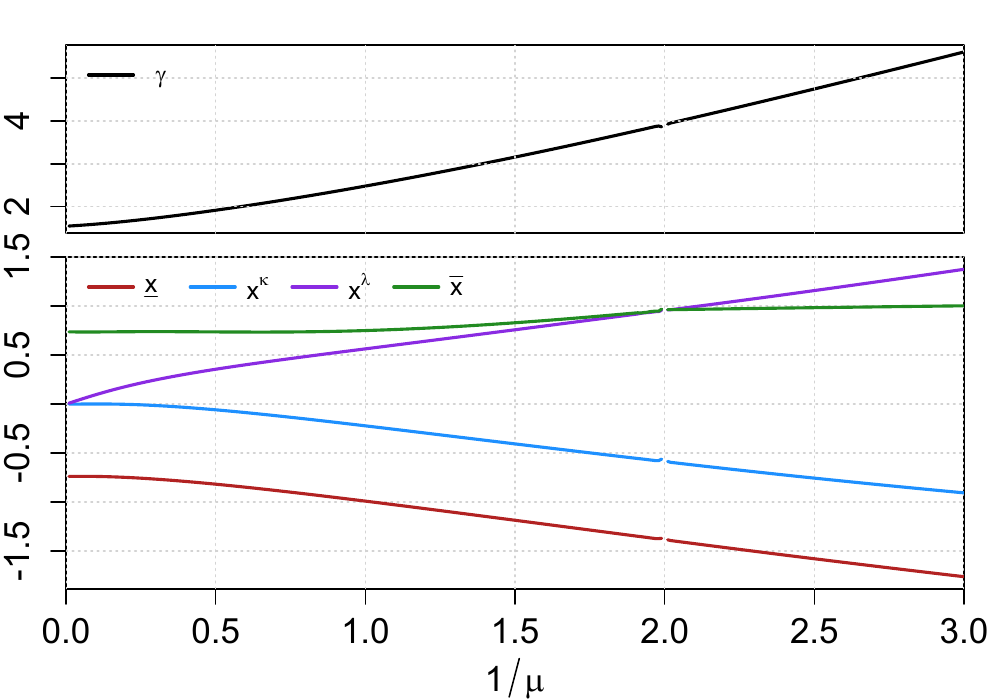}
			\caption{$c_U = 1$, $c_D = 1$}
		\end{subfigure}%
		\begin{subfigure}[t]{0.32\textwidth}\centering
			\includegraphics[width=\linewidth]{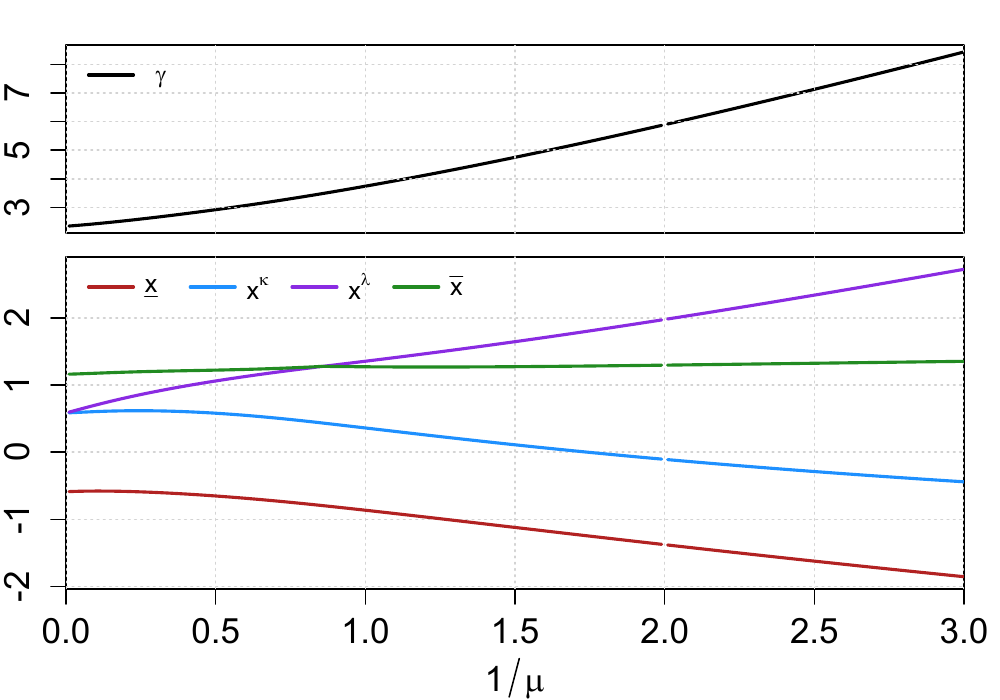}
			\caption{$c_U = 2$, $c_D = 1$}
		\end{subfigure}%
		\begin{subfigure}[t]{0.32\textwidth}\centering
			\includegraphics[width=\linewidth]{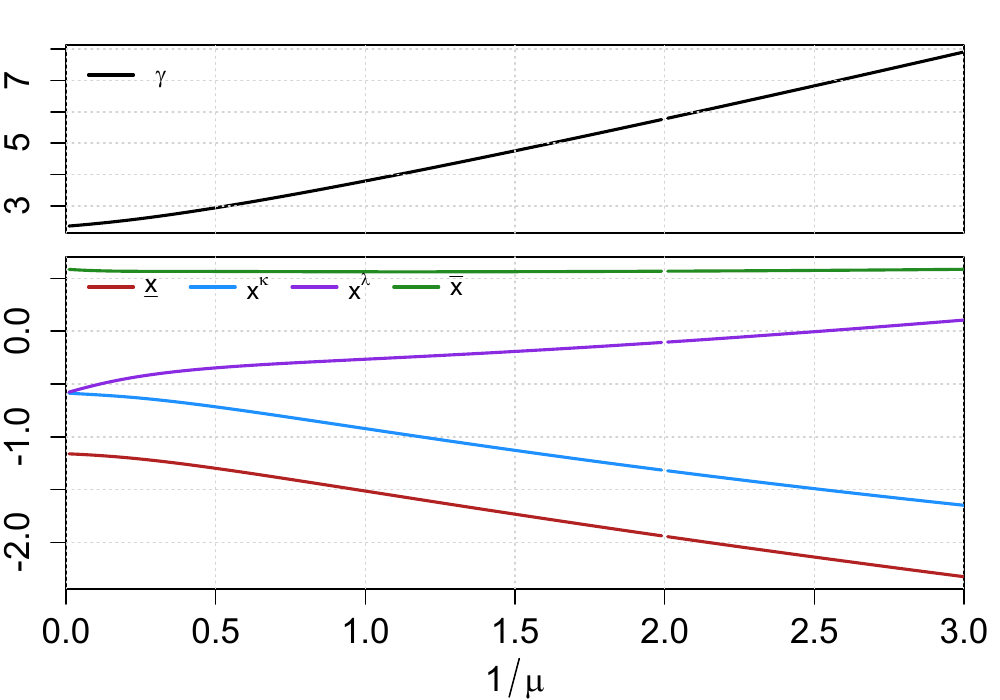}
			\caption{$c_U = 1$, $c_D = 2$}
		\end{subfigure}
		
		\caption{\small Effect of the mean jump size $1/\mu=-\Esp[Y]\in(0,3)$ on the reflecting barriers and ambiguity thresholds.}
		\label{fig:comparative_statistics_jump-mean}
	\end{figure}
	
	\paragraph{Effect of the intervention costs $\pmb{c_U}$ and $\pmb{c_D}$.}
	Figure~\ref{fig:comparative_statistics_pushing-costs} shows that the effect of the intervention costs depends on the direction in which the uncontrolled system tends to drift. 
	When $b=2$, the process naturally moves toward high states, so upward intervention at the lower barrier is relatively rare. Increasing $c_U$ then typically lowers the lower barrier, reducing the need for expensive upward corrections. 
	When $b=-2$, the system tends to remain in low states, so upward intervention becomes frequent and harder to avoid. In this case the controller keeps the lower barrier closer to the running-cost minimizer and relaxes the upper barrier instead, in order to reduce future costly adjustments. 
	The effect of varying $c_D$ is analogous, with the roles of the upper and lower interventions reversed. 
	In all cases, the ambiguity thresholds adjust consistently with these changes, tending to expand the region over which the least-favorable regime is active when the corresponding intervention becomes more expensive.
	
	\begin{figure}[!ht]
		\captionsetup[sub]{labelformat=parens}
		\centering
		
		\begin{subfigure}[t]{0.24\textwidth}\centering
			\includegraphics[width=\linewidth]{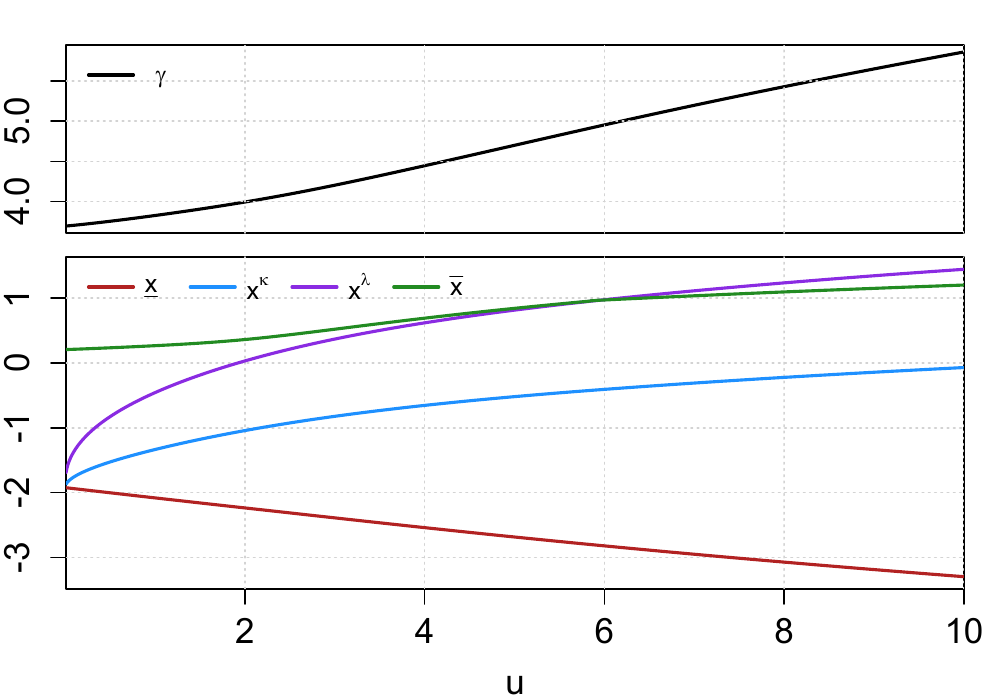}
			\caption{$b = 2$}
		\end{subfigure}%
		\begin{subfigure}[t]{0.24\textwidth}\centering
			\includegraphics[width=\linewidth]{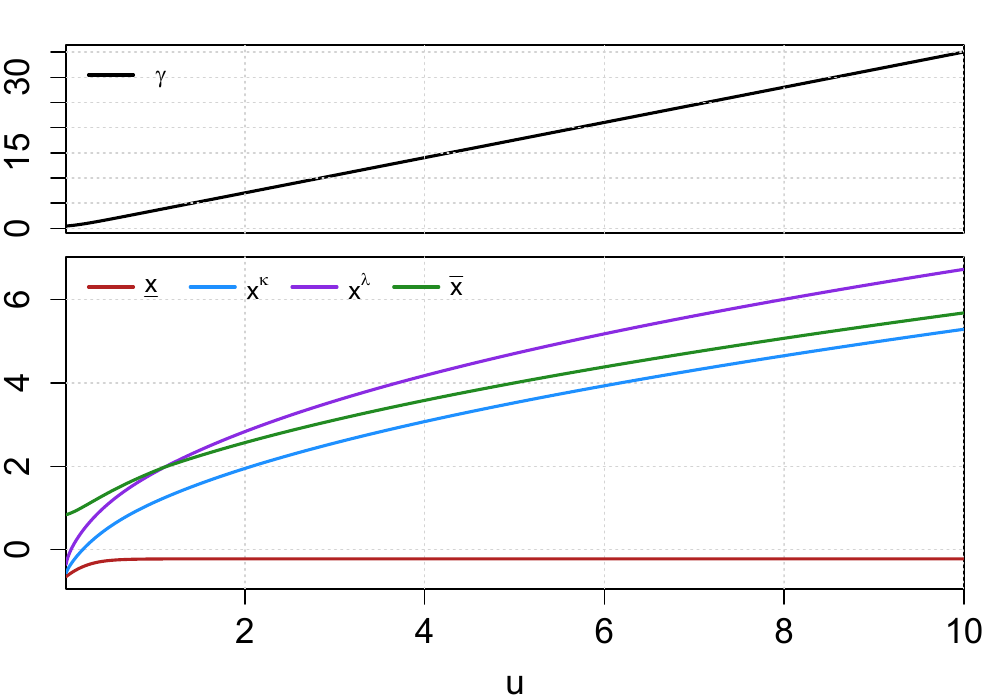}
			\caption{$b = -2$}
		\end{subfigure}%
		\begin{subfigure}[t]{0.24\textwidth}\centering
			\includegraphics[width=\linewidth]{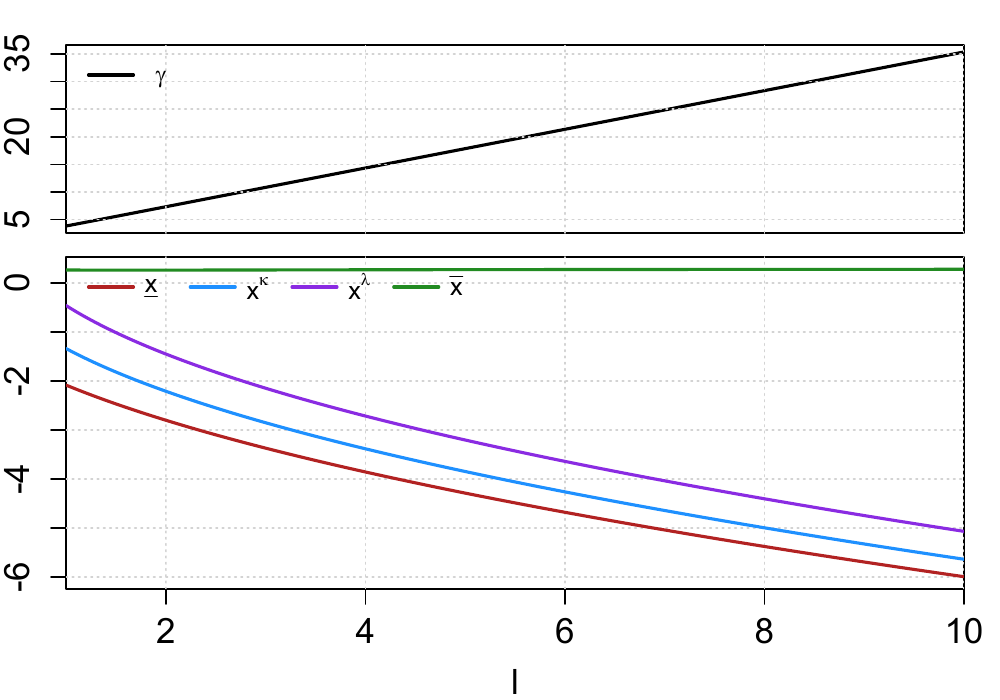}
			\caption{$b = 2$}
		\end{subfigure}
		\begin{subfigure}[t]{0.24\textwidth}\centering
			\includegraphics[width=\linewidth]{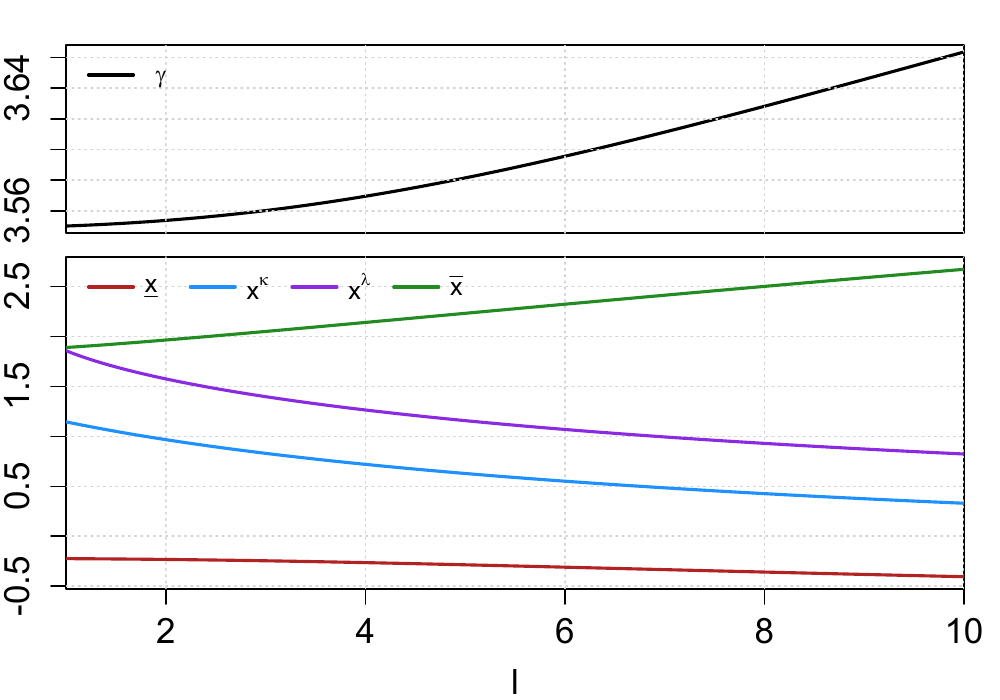}
			\caption{$b = -2$}
		\end{subfigure}
		
		\caption{\small Effect of the intervention costs: in (a) and (b) $c_U\in(0,10)$ and $c_D=1$, while in (c) and (d) $c_D\in(1,10)$  and $c_U=1$.}
		\label{fig:comparative_statistics_pushing-costs}
	\end{figure}
	
	Taken together, these experiments show that model ambiguity and jump risk systematically reshape the geometry of the optimal long-run regulation policy. Robustness operates as an endogenous caution mechanism: it modifies the width and asymmetry of the inaction region so as to reduce exposure to unfavorable dynamics and costly future interventions. The numerical evidence therefore complements the analytical characterization by showing how drift uncertainty, jump uncertainty, volatility, and cost asymmetries interact in determining economically meaningful long-run regulation rules.
	
	\paragraph{Value of robustness.}
	To quantify the cost of ignoring ambiguity, we compare the robust ergodic value with the worst-case performance of the policy that is optimal under the reference (non-robust) model. 
	More precisely, let $\gamma^{\mathrm{R}}$ denote the minimal long-run average cost in the robust problem defined at \eqref{eq:ergodic_singular_control_problem}, and let $\gamma^{\mathrm{NR},\, \mathrm{WC}}$ denote the long-run average cost obtained when the non-robust optimal policy is evaluated under the worst-case model. The Relative Misspecification Cost (RMC) is therefore given by
	\begin{align*}
		\mathrm{RMC}
		\defeq
		\frac{\gamma^{\mathrm{NR},\, \mathrm{WC}}-\gamma^{R}}{\gamma^{\mathrm{R}}} \times 100.    
	\end{align*}
	which measures the percentage performance loss generated by ignoring ambiguity. Notice that the RMC is non-negative, since
	\begin{align*}
		\gamma^{\mathrm{NR},\, \mathrm{WC}}
		&\defeq
		\sup_{(\kappa,\lambda)\in\Lambda}
		J_x(U^{\mathrm{NR}},D^{\mathrm{NR}},\kappa,\lambda) 
		\geq 
		\inf_{(U,D)\in\mathcal{A}}\sup_{(\kappa,\lambda)\in\Lambda}
		J_x(U,D,\kappa,\lambda) 
		= \gamma^\mathrm{R} \geq 0,
	\end{align*}
	where $(U^{\mathrm{NR}},D^{\mathrm{NR}})$ is the optimal policy for the reference non-robust model, namely,
	\[
	(U^{\mathrm{NR}},D^{\mathrm{NR}})
	=
	\arg\min_{(U,D)\in\mathcal{A}} J_x(U,D,0,r),
	\]
	where, with a slight abuse of notation, $0$ and $r$ denote the constant distortion functions $\kappa(\cdot)\equiv 0$ and $\lambda(\cdot)\equiv r$.

	Table~\ref{tab:rmc_table_b_comparison} and figures~\ref{fig:rmc_surface} and \ref{fig:rmc_2d-plot} report the corresponding RMC values for the set of parameters
	\[
	b \in \{-2,2\},\qquad r=1,\qquad \sigma=1,\qquad \mu=1,\qquad c_U = 1,\qquad c_D = 1.
	\]
	The results show a non-negative RMC that generally increases with respect to each ambiguity parameter. Hence, ignoring ambiguity is always costly, and large losses arise when overall ambiguity is high, although the leading source of ambiguity driving the RMC increase changes between the two drift regimes. 
	
	The two drift specifications differ substantially in quantitative terms. 
	When $b=2$, the RMC surface is noticeably steeper and reaches much larger values, showing that the value of robustness is substantially higher under a positive underlying drift. 
	When $b=-2$, the same monotone pattern remains, but the RMC is considerably smaller. 
	In both cases, the effect of one ambiguity source becomes more pronounced as the other increases, revealing a reinforcing interaction between drift and intensity ambiguity. 
	However, the sensitivity of the RMC with respect to the two ambiguity sources differs across the two drift regimes: the RMC appears more sensitive to intensity ambiguity when $b=2$, whereas for $b=-2$ it appears more sensitive to drift ambiguity.
	
	\begin{table}[htbp]
		\centering
		\caption{RMC as a function of $\delta$ and $\varepsilon$ for $b=2$ and $b=-2$.}
		\label{tab:rmc_table_b_comparison}
		\small
		\setlength{\tabcolsep}{3.5pt}
		\begin{tabular}{@{}c*{6}{r}*{6}{r}@{}}
			\toprule
			& \multicolumn{6}{c}{$b=2$} & \multicolumn{6}{c}{$b=-2$} \\
			\cmidrule(lr){2-7} \cmidrule(lr){8-13}
			$\delta \backslash \varepsilon$
			& 0.0 & 0.2 & 0.4 & 0.6 & 0.8 & 1.0
			& 0.0 & 0.2 & 0.4 & 0.6 & 0.8 & 1.0 \\
			\midrule
			0.0 & 0.00 & 0.07 & 0.28 & 0.66 & 1.23 & 2.00
			& 0.00 & 0.01 & 0.02 & 0.05 & 0.07 & 0.10 \\
			0.2 & 0.01 & 0.12 & 0.37 & 0.78 & 1.39 & 2.22
			& 0.06 & 0.09 & 0.12 & 0.16 & 0.18 & 0.21 \\
			0.4 & 0.03 & 0.17 & 0.45 & 0.89 & 1.52 & 2.37
			& 0.20 & 0.24 & 0.27 & 0.30 & 0.32 & 0.34 \\
			0.6 & 0.06 & 0.23 & 0.53 & 0.98 & 1.61 & 2.47
			& 0.38 & 0.41 & 0.43 & 0.45 & 0.47 & 0.48 \\
			0.8 & 0.10 & 0.29 & 0.59 & 1.05 & 1.67 & 2.52
			& 0.56 & 0.58 & 0.59 & 0.60 & 0.60 & 0.60 \\
			1.0 & 0.13 & 0.34 & 0.65 & 1.10 & 1.72 & --
			& 0.73 & 0.73 & 0.73 & 0.73 & 0.72 & -- \\
			\bottomrule
		\end{tabular}
	\end{table}
	
	\begin{figure}[!ht]
		\captionsetup[sub]{labelformat=parens}
		\centering
		
		\begin{subfigure}[t]{0.47\textwidth}
			\centering
			\includegraphics[width=\linewidth]{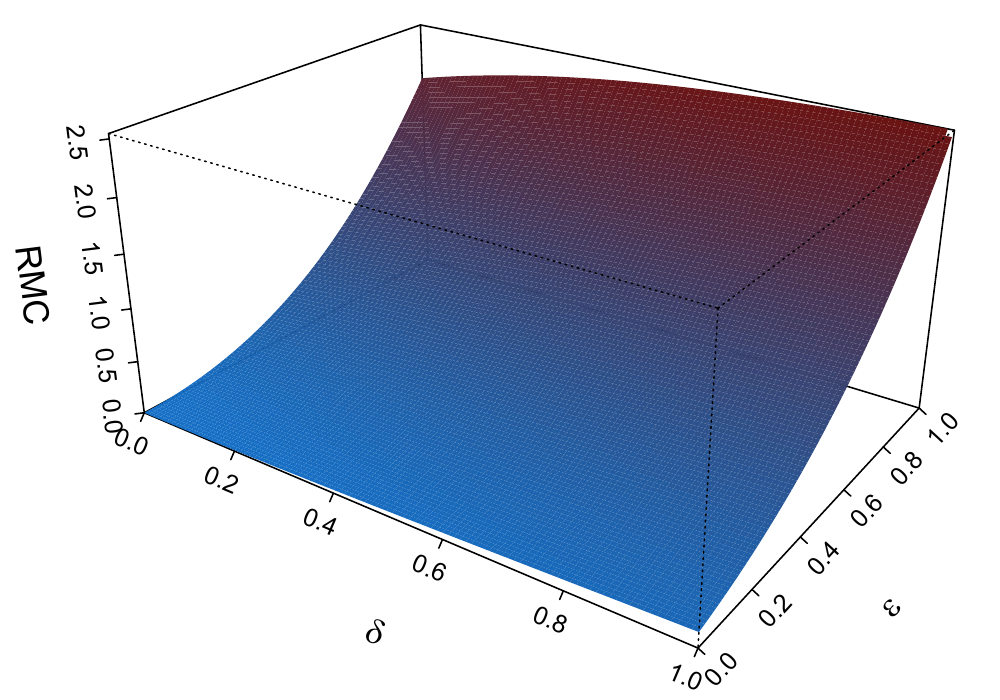}
			\caption{\small $b = 2$.}
		\end{subfigure}\hspace{2pt}
		\begin{subfigure}[t]{0.47\textwidth}
			\centering
			\includegraphics[width=\linewidth]{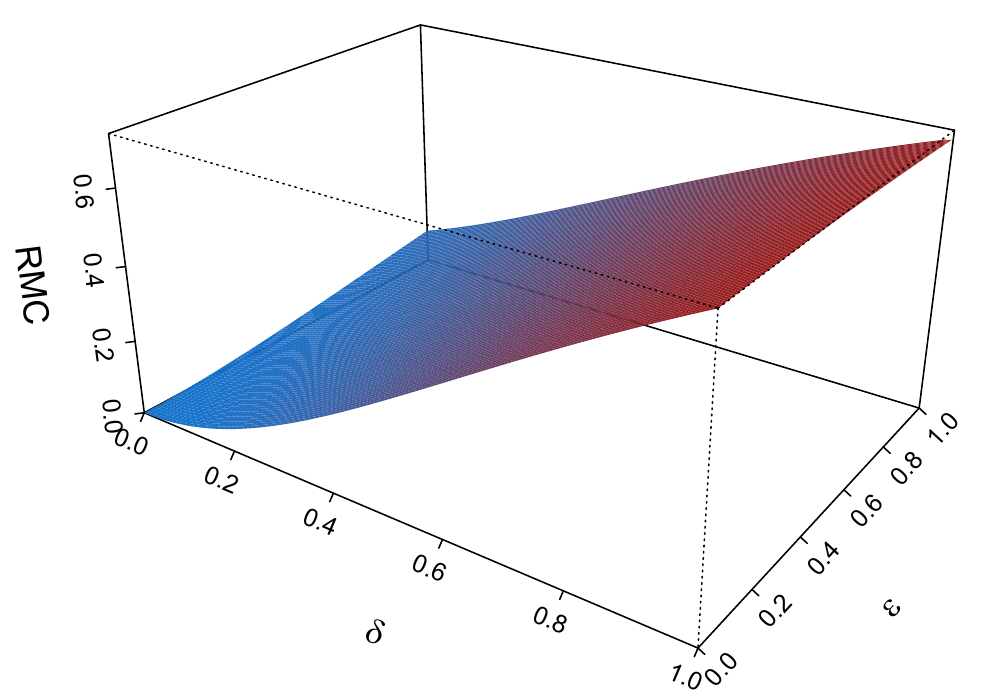}
			\caption{\small $b = -2$}
		\end{subfigure}\hspace{2pt}
		
		\caption{Relative misspecification cost (RMC) as a function of the drift-ambiguity parameter $\delta$ and the intensity-ambiguity parameter $\varepsilon$, for two values of the baseline drift.
		}
		\label{fig:rmc_surface}
	\end{figure}
	
	\begin{figure}[!ht]
		\captionsetup[sub]{labelformat=parens}
		\centering
		
		\begin{subfigure}[t]{0.32\textwidth}
			\centering
			\includegraphics[width=\linewidth]{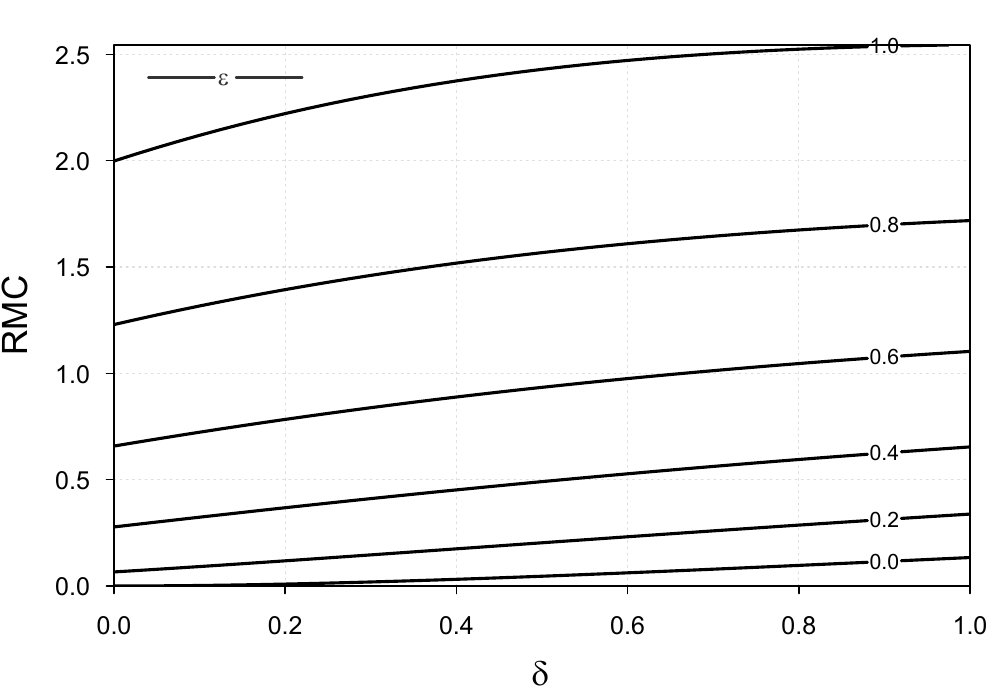}
			\caption{\small RMC as a function of $\delta$ for fixed values of $\varepsilon$, with $b = 2$.
			}
		\end{subfigure}\hspace{2pt}
		\begin{subfigure}[t]{0.32\textwidth}
			\centering
			\includegraphics[width=\linewidth]{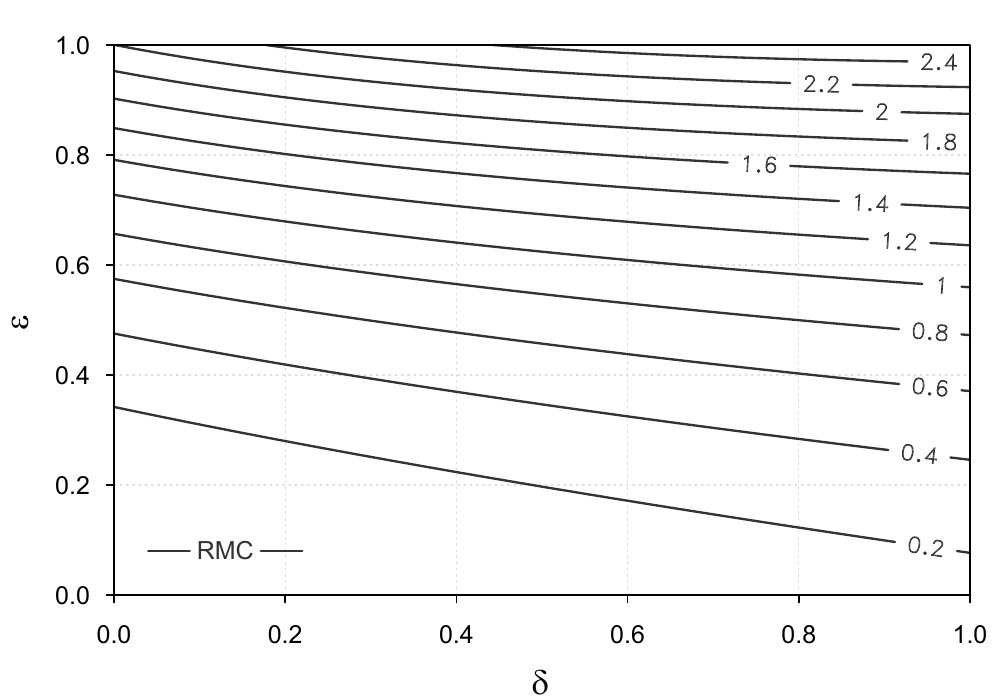}
			\caption{\small Contour plot of the RMC over $(\delta,\varepsilon)$, with $b = 2$.
			}
		\end{subfigure}\hspace{2pt}
		\begin{subfigure}[t]{0.32\textwidth}
			\centering
			\includegraphics[width=\linewidth]{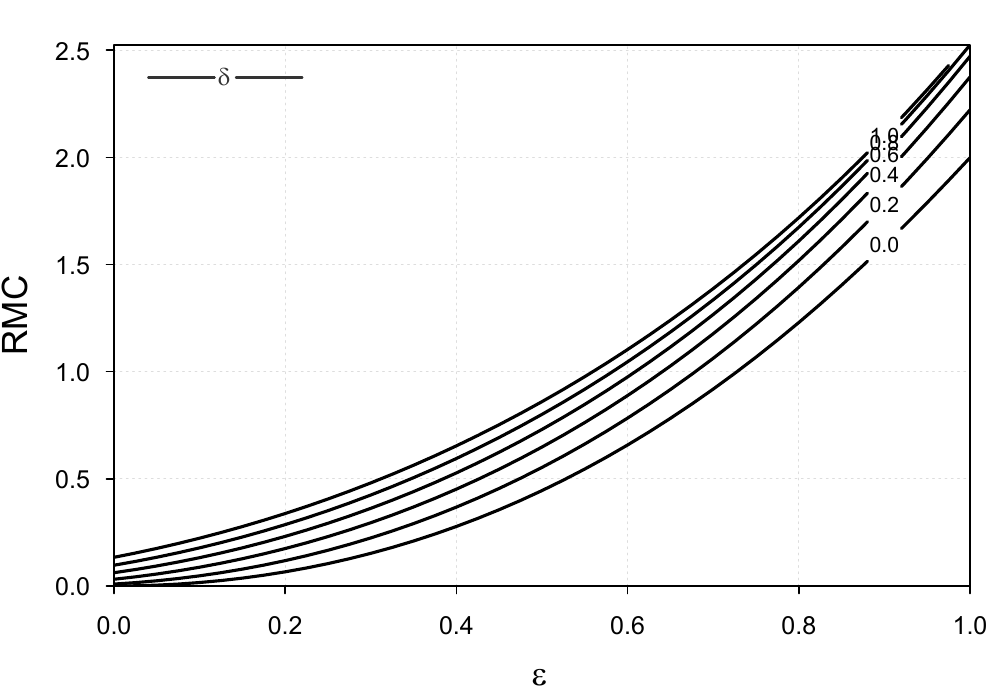}
			\caption{\small RMC as a function of the $\varepsilon$ for fixed values of $\delta$, with $b = 2$.
			}
		\end{subfigure}\hspace{2pt}
		
		\begin{subfigure}[t]{0.32\textwidth}
			\centering
			\includegraphics[width=\linewidth]{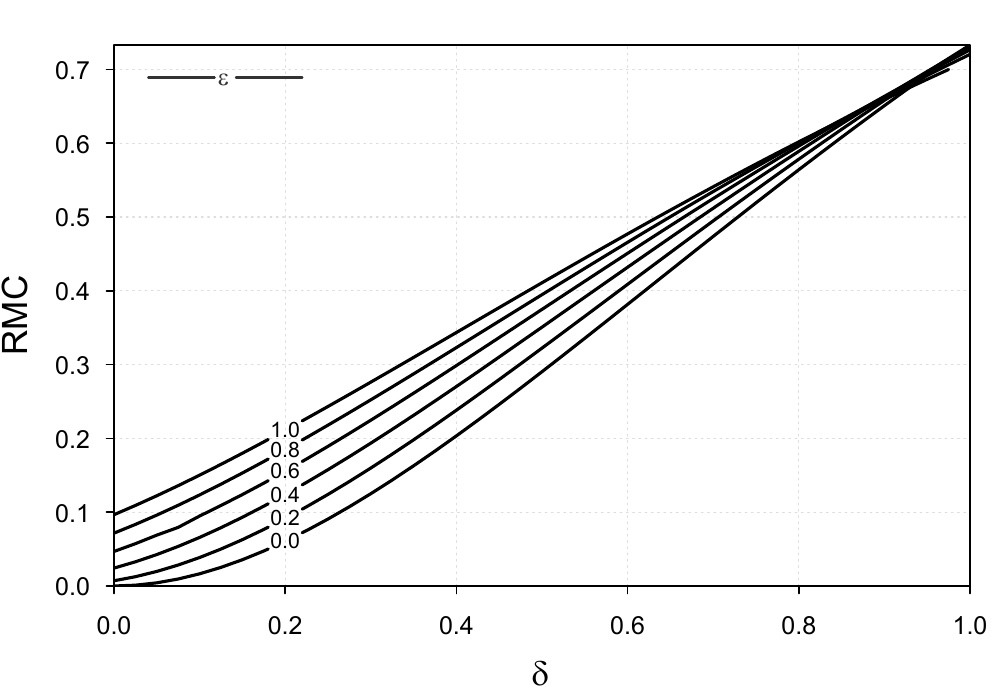}
			\caption{\small RMC as a function of $\delta$ for fixed values of $\varepsilon$, with $b = -2$.
			}
		\end{subfigure}\hspace{2pt}
		\begin{subfigure}[t]{0.32\textwidth}
			\centering
			\includegraphics[width=\linewidth]{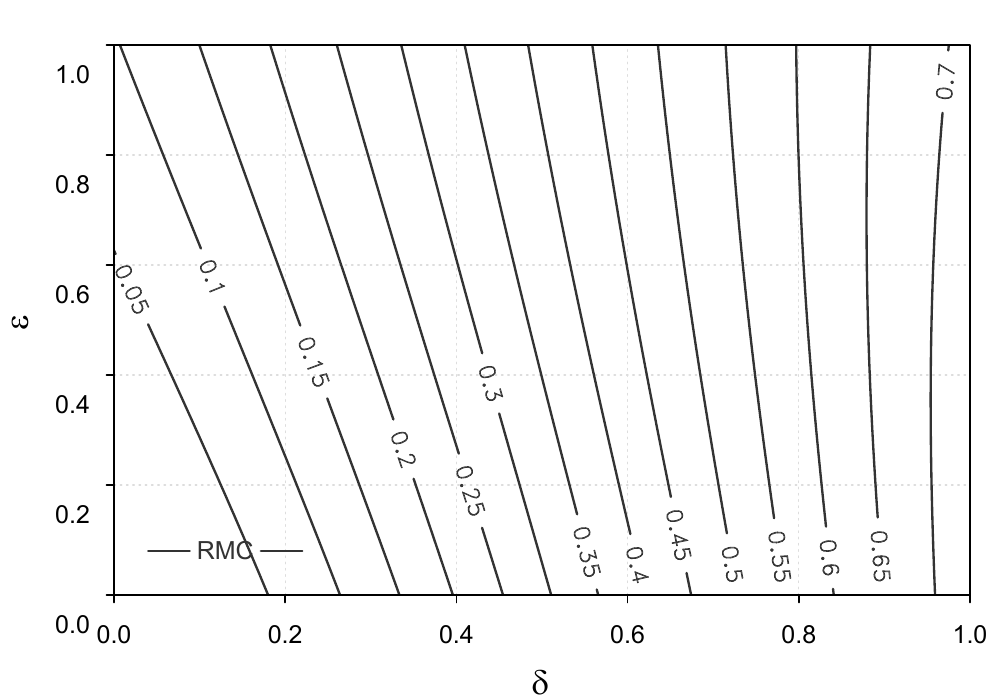}
			\caption{\small Contour plot of the RMC over $(\delta,\varepsilon)$, with $b = -2$.
			}
		\end{subfigure}\hspace{2pt}
		\begin{subfigure}[t]{0.32\textwidth}
			\centering
			\includegraphics[width=\linewidth]{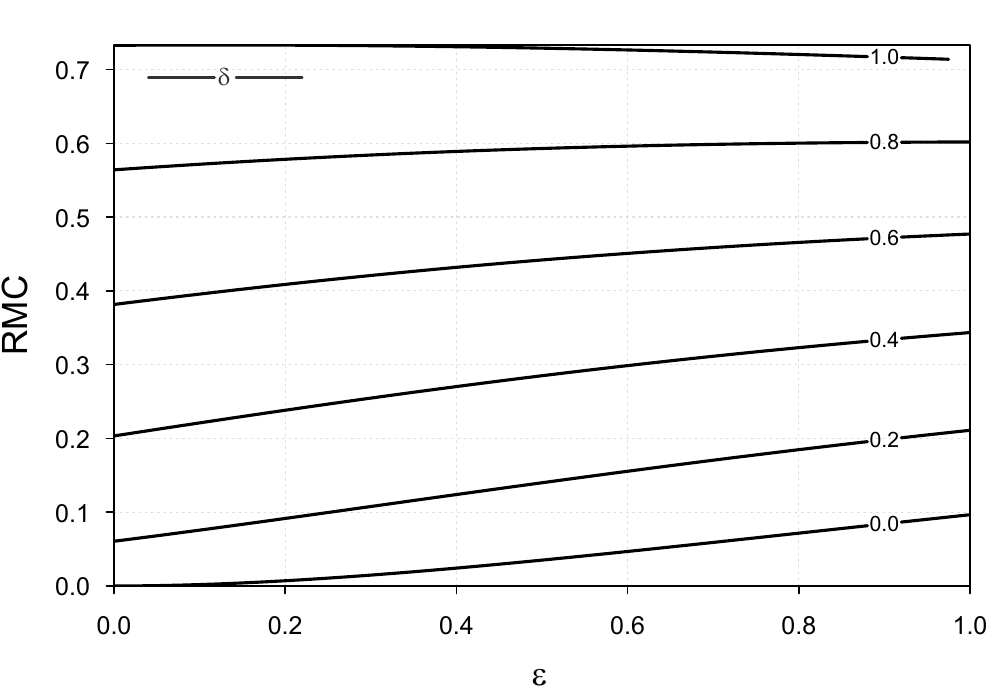}
			\caption{\small RMC as a function of the $\varepsilon$ for fixed values of $\delta$, with $b = -2$.
			}
		\end{subfigure}\hspace{2pt}
		
		\caption{Cross-sectional views of the relative misspecification cost (RMC). Panels~(a)--(c) correspond to $b=2$, while panels~(d)--(f) correspond to $b=-2$.
		}
		\label{fig:rmc_2d-plot}
	\end{figure}
	
	%-------------------------------------------------%
	\section{Concluding remarks}
	%-------------------------------------------------%
	
	We have studied a robust ergodic singular control problem for jump-diffusion systems with ambiguity in both the drift and the jump intensity. 
	Our analysis shows that the associated max--min control problem admits a tractable characterization in terms of a nonlinear integro-differential free-boundary problem, whose solution yields a reflecting-barrier policy and bang-bang worst-case distortions.
	
	From an operational perspective, the results highlight the importance of accounting for model uncertainty in long-run regulation problems. 
	The numerical analysis shows that ambiguity affects not only the width but also the asymmetry of the inaction region, leading to more conservative and state-dependent intervention policies.
	
	In addition, the comparison of ergodic costs provides a quantitative measure of the value of robustness. We show that ignoring ambiguity can lead to non-negligible performance losses, and that the dominant source of misspecification cost depends on the drift regime. In the reported experiments, intensity ambiguity plays the larger role when the baseline drift is positive, while drift ambiguity is more influential when the baseline drift is negative. Thus, the value of robustness reflects the interaction between persistent directional forces and uncertainty about the frequency of jumps.
	
	These findings suggest that neglecting uncertainty, especially in the presence of persistent directional misspecification, may result in systematically suboptimal decisions. 
	Future research could extend the analysis to multidimensional systems, alternative ambiguity specifications, or impulse-control formulations.
	
	\section*{Companion code} the code used to support the numerical experiments an be found at \url{https://github.com/aguazz/robust_ergodic-singular-control_jump-diffusion}.
	
	\section*{Funding}
	The second author acknowledges the financial support by the Italian SID project BIRD239937/23. The third author acknowledges the financial support by the Deutsche Forschungsgemeinschaft (DFG, German Research Foundation) under Germany's Excellence Strategy, Project-ID 317210226, SFB 1283.
	
	\section*{Acknowledgments}
	This work started during the visits at the Center for Mathematical Economics (IMW) at Bielefeld University of Bernardo D'Auria and Abel Guada Azze, who thank IMW and the SFB 1283 for the support and hospitality.

	%%%%%%%%%%%%%%%%%%%%%%%%%%%

	\appendix
	\section{Some Proofs}
	\label{app:someproofs}
	
	\subsection{Proof of Lemma \ref{lm:control_ergodic_bounds}}
	
	Define the midpoint and half-width $m \defeq (x_1+x_2)/2$ and $h \defeq (x_2-x_1)/2$,
	and consider the function $f(x)\defeq (x-m)^2$, for which $f'(x)=2(x-m)$ and $f''(x)=2$.
	
	From \eqref{eq:inf_gen}, we have for $x\in\R$,
	\begin{align*}
		(\InfGen^{\kappa,\lambda} f)(x)
		&= (b(x) + \sigma\kappa(x) - r\Esp[Y])2(x-m) + \sigma^2 + \lambda(x)(\Psi f)(x),
	\end{align*}
	where
	\begin{align*}
		(\Psi f)(x)
		&= \int_\R \lrp{(x+y-m)^2 - (x-m)^2}F(\rmd y)
		= \int_\R \lrp{2(x-m)y + y^2}F(\rmd y) \\
		&= 2(x-m)\Esp[Y] + \Esp[Y^2].
	\end{align*}
	Hence,
	\begin{align*}
		(\InfGen^{\kappa,\lambda} f)(x)
		&= 2(x-m)\lrp{b(x) + \sigma\kappa(x) - (r-\lambda(x))\Esp[Y]} + \sigma^2 + \lambda(x)\Esp[Y^2].
	\end{align*}
	
	Let $\oline{X}_t \defeq X_t^{x_1,x_2}$. Since $\oline{X}_t\in[x_1,x_2]$ $\Pr^{\kappa,\lambda}$-a.s. for all $t$, we have $|\oline{X}_t-m|\le h$.
	Using $|\kappa|\le \delta$, $|r-\lambda|\le \varepsilon r$, and $|\Esp[Y]|\le \Esp[|Y|]$, we obtain for all $x\in[x_1,x_2]$,
	\begin{align}\label{eq:InfGen_bound_band}
		|(\InfGen^{\kappa,\lambda} f)(\oline{X}_t)|
		&\leq (x_2-x_1)K(x_1,x_2). 
	\end{align}
	
	Use Itô's Lemma for semimartingales (see, e.g., Theorem II.32 from \cite{protter_stochastic_2005}) to $f(\oline{X}_T)$ and take $\Pr_x^{\kappa,\lambda}$-expectation with $x\in[x_1,x_2]$
	to vanish the diffusion and jump martingale terms (see also equation 8.2.6 from \cite{OksendalSulem2019}), obtaining the Dynkin-type formula:
	\begin{align}\label{eq:Dynkin_band}
		&\Esp[f(\oline{X}_T)]_x^{\kappa,\lambda} - f(x) \\
		&=
		\Esp[
		\int_0^T (\InfGen^{\kappa,\lambda}f)(\oline{X}_t)\,\rmd t
		+ \int_0^T f'(\oline{X}_t)\,\rmd U_t^{x_1,c} - \int_0^T f'(\oline{X}_t)\,\rmd D_t^{x_2,c}
		+ \sum_{0<t\le T}\Ind_{\{\Delta U_t^{x_1}+\Delta D_t^{x_2}>0\}}\lrp{f(\oline{X}_t)-f(\widehat{X}_t)}
		]_x^{\kappa,\lambda},
		\nonumber
	\end{align}
	where $U_t^{x_1,c}$ and $D_t^{x_2,c}$ denote the continuous parts of $U_t^{x_1}$ and $D_t^{x_2}$, respectively, and $\widehat{X}_t \defeq \oline{X}_{t-}+\Delta J_t$ is the pre-controlled post-jump state, with $\Delta J_t$ the Poisson jump at time $t$.
	
	Since $\rmd U_t^{x_1,c}$ and $\rmd D_t^{x_2,c}$ are activated only when $\oline{X}_t=x_1$ and $\oline{X}_t=x_2$, respectively, we have
	\begin{equation}\label{eq:cont_contribution_band}
		\left\{
		\begin{aligned}
			\int_0^T f'(\oline{X}_t)\,\rmd U_t^{x_1,c}
			&= \int_0^T 2(\oline{X}_t-m)\,\rmd U_t^{x_1,c}
			= 2(x_1-m)U_T^{x_1,c}
			= -(x_2-x_1)U_T^{x_1,c}, \\
			-\int_0^T f'(\oline{X}_t)\,\rmd D_t^{x_2,c}
			&= -\int_0^T 2(\oline{X}_t-m)\,\rmd D_t^{x_2,c}
			= -2(x_2-m)D_T^{x_2,c}
			= -(x_2-x_1)D_T^{x_2,c}.
		\end{aligned}
		\right.
	\end{equation}
	
	For the jump correction term, notice that if $\widehat{X}_t<x_1$ then $\Delta U_t^{x_1}=x_1-\widehat{X}_t>0$ and $\oline{X}_t=x_1$. Hence, using $x_1-m=-h$,
	\begin{align*}
		f(\oline{X}_t)-f(\widehat{X}_t) 
		&= f(x_1)-f(x_1-\Delta U_t^{x_1})
		= h^2-(h+\Delta U_t^{x_1})^2 \\
		&= -\Delta U_t^{x_1}\lrp{\Delta U_t^{x_1}+2h}\le -(2h)\Delta U_t^{x_1} = -(x_2-x_1)\Delta U_t^{x_1}.
	\end{align*}

	Similarly, if $\widehat{X}_t>x_2$ then $\Delta D_t^{x_2}=\widehat{X}_t-x_2>0$ and $\oline{X}_t=x_2$. Consequently, using $x_2-m=h$,
	\begin{align*}
		f(\oline{X}_t)-f(\widehat{X}_t) 
		&= f(x_2)-f(x_2+\Delta D_t^{x_2})
		= h^2-(h+\Delta D_t^{x_2})^2 \\
		&= -\Delta D_t^{x_2}\lrp{\Delta D_t^{x_2}+2h}\le -(2h)\Delta D_t^{x_2} 
		= -(x_2-x_1)\Delta D_t^{x_2}.
	\end{align*}
	Therefore,
	\begin{align}\label{eq:jump_contribution_band}
		\sum_{0<t\le T}\Ind_{\{\Delta U_t^{x_1}+\Delta D_t^{x_2}>0\}}\lrp{f(\oline{X}_t)-f(\widehat{X}_t)}
		\le -(x_2-x_1)\sum_{0<t\le T}\lrp{\Delta U_t^{x_1}+\Delta D_t^{x_2}}
		= -(x_2-x_1)\lrp{U_T^{x_1,d}+D_T^{x_2,d}},
	\end{align}
	where $U^{x_1,d}$ and $D^{x_2,d}$ denote the purely discontinuous parts.
	
	Plugging \eqref{eq:cont_contribution_band} and \eqref{eq:jump_contribution_band} into \eqref{eq:Dynkin_band}, and merging continuous and jump parts, we get
	\begin{align}\label{eq:ergodic_control_bound_auxiliary}
		(x_2-x_1)\Esp[U_T^{x_1}+D_T^{x_2}]_x^{\kappa,\lambda}
		&\leq f(x) - \Esp[f(\oline{X}_T)]_x^{\kappa,\lambda}
		+ \Esp[\int_0^T (\InfGen^{\kappa,\lambda}f)(\oline{X}_t)\,\rmd t]_x^{\kappa,\lambda}.
	\end{align}
	Since $\oline{X}_t\in[x_1,x_2]$, we have $0\le f(\oline{X}_T)\le h^2$, hence $0\leq \Esp[f(\oline{X}_T)]\leq h^2$.
	Hence, \eqref{eq:ergodic_control_bound} is obtained after 
	dividing both sides in \eqref{eq:ergodic_control_bound_auxiliary} by $(x_2-x_1)T$, using \eqref{eq:InfGen_bound_band}, and letting $T\to\infty$.
	For $x\notin[x_1,x_2]$, the reflectors apply an initial projection to $\hat{x} \in [x_1,x_2]$, that is, $\hat{x} = x + \Delta U_0^{x_1} - \Delta D_0^{x_2}$.  Hence, 
	\begin{align*}
		\Esp[U_T^{x_1}+D_T^{x_2}]_x^{\kappa,\lambda} = (\Delta U_0^{x_1} + \Delta D_0^{x_2}) + \Esp[U_T^{x_1}+D_T^{x_2}]_{\hat{x}}^{\kappa,\lambda},
	\end{align*}
	making $\limsup_{T\to\infty}T^{-1}\Esp[U_T^{x_1}+D_T^{x_2}]_x^{\kappa,\lambda} = \limsup_{T\to\infty}T^{-1}\Esp[U_T^{x_1}+D_T^{x_2}]_{\hat{x}}^{\kappa,\lambda}$.
	
	\subsection{Proof of Proposition \ref{pr:ergodic-value_upper_bound}}

	The proof follows immediately after using the suboptimality of the Skorokhod reflectors $(U^{x_1}, D^{x_2})$ in \eqref{eq:ergodic_singular_control_problem}, the ergodic control bound \eqref{eq:ergodic_control_bound} in Lemma \ref{lm:control_ergodic_bounds}, and the fact that the controlled process $Z_t = X_t^{U^{x_1},D^{x_2}}$ is constrained to $[x_1,x_2]$. That is, for all $(\kappa,\lambda)$ and for all $x \in \R$, $\Esp[\frac{1}{t}\int_0^t c(Z_s)\,\rmd s]_x^{\kappa,\lambda} \leq \max_{x\in[x_1,x_2]} c(x)$ and 
	\begin{align}
		\frac{1}{t}\Esp[c_U U_t^{x_1} + c_D D_t^{x_2}]_x^{\kappa,\lambda} \leq (c_U + c_D)K(x_1,x_2),
	\end{align}
	which results in
	\begin{align*}
		\gamma \leq \sup_{(\kappa,\lambda)\in\Lambda} J_x(U^{x_1},D^{x_2},\kappa,\lambda) \leq \max_{x\in[x_1,x_2]} c(x) + (c_U + c_D)K(x_1,x_2).
	\end{align*}
	
	\subsection{Proof of Theorem \ref{thm:verification_theorem}}
	
	Let $\oline{\gamma}\defeq \inf_{(U,D)\in\cA} \sup_{(\kappa,\lambda)\in\Lambda} J_x(U,D,\kappa,\lambda)$
	and $\uline{\gamma}\defeq \sup_{(\kappa,\lambda)\in\Lambda}\inf_{(U,D)\in\cA} J_x(U,D,\kappa,\lambda)$.
	We show first that $\gamma \leq \uline{\gamma}$ and then that $\gamma \geq \oline{\gamma}$, implying that $\uline{\gamma}=\gamma=\oline{\gamma}$.
	\begin{itemize}[label=$\triangleright$]
		\item \emph{$\gamma \leq \uline{\gamma}$}: 
		Take $\kappa^*(x)$ and $\lambda^*(x)$ defined in \eqref{eq:bang-bang_opt} and let $\Pr_x^* = \Pr_x^{\kappa^*, \lambda^*}$, $\Esp_x^* = \Esp_x^{\kappa^*, \lambda^*}$ and $\InfGen^* = \InfGen^{\kappa^*,\lambda^*}$. For $(U,D)\in\cA$, denote $\oline{X}_t = X_t^{U,D}$. 
		Since $V\in C^2(\R)$ (recall \eqref{eq:FBP_C2}), we can use Dynkin formula \eqref{eq:Dynkin_band} for $V$ instead of $f$ to get that
		\begin{align}\label{eq:Dynkin_optimal_ambiguity} 
			&\Esp[V(\oline{X}_T)]_x^* - V(x) \\ 
			&= 
			\Esp[\int_0^T (\InfGen^*V)(\oline{X}_t) \,\rmd t + \int_0^T V'(\oline{X}_t)\,\rmd U_t^c - \int_0^T V'(\oline{X}_t)\,\rmd D_t^c
			+ \sum_{0<t\leq T}\Ind_{\{\Delta U_t + \Delta D_t > 0\}}(V(\oline{X}_t) - V(\oline{X}_{t-} + \Delta J_t))]_x^*, \nonumber \\
			&= 
			\Esp[\int_0^T \lrp{c(\oline{X}_t) + (\InfGen^*V)(\oline{X}_t) - \gamma}\,\rmd t]_x^* 
			+ \gamma T - \Esp[\int_0^T c(\oline{X}_t)\,\rmd t + \int_0^T -V'(\oline{X}_t)\,\rmd U_t^c + \int_0^T V'(\oline{X}_t)\,\rmd D_t^c]_x^* \nonumber \\
			&
			\hspace{0.4cm} + 
			\Esp[\sum_{0<t\leq T}\lrp{\Ind_{\{\Delta U_t > 0\}}\int_{\oline{X}_{t-} + \Delta J_t}^{\oline{X}_{t-} + \Delta J_t + \Delta U_t}V'(z)\,\rmd z 
				- 
				\Ind_{\{\Delta D_t > 0\}}\int_{\oline{X}_{t-} + \Delta J_t - \Delta D_t}^{\oline{X}_{t-} + \Delta J_t}V'(z)\,\rmd z 
			}]_x^*, \nonumber
		\end{align}
		where $\rmd U_t^c$ and $\rmd D_t^c$ are the continuous part of the Stieltjes measures associated $U_t$ and $D_t$, respectively, and with $\Delta J_t \defeq \Ind_{\{N_t - N_{t-} = 1\}}Y_{N_t}$ the Poisson jump.
		It follows from \eqref{eq:FBP_continuation}-\eqref{eq:FBP_upper_stopping} that
		\begin{align*}
			\Esp[V(\oline{X}_T)]_x^* - V(x) &\geq \gamma T - \Esp[\int_0^T c(\oline{X}_t)\,\rmd t + \int_0^T c_U \,\rmd U_t + \int_0^T c_D\,\rmd D_t]_x^*,   
		\end{align*}
		or, equivalently,
		\begin{align*}%\label{eq:gamma_inequality_leq}
			\gamma \leq \frac{1}{T}\lrp{\Esp[V(\oline{X}_T)]_x^* - V(x) + \Esp[\int_0^T c(\oline{X}_t)\,\rmd t + \int_0^T c_U \,\rmd U_t + \int_0^T c_D\,\rmd D_t]_x^*}.
		\end{align*}
		Taking $T \rightarrow \infty$, recalling the growth condition in \eqref{eq:admissible_controls}, and noticing that $|V(x)| \leq M(1 + |x|)$ for some $M > 0$ as $|V'(x)| \leq M$, one gets that
		\begin{align*}
			\gamma \leq \limsup_{T\rightarrow\infty}\frac{1}{T}\Esp[\int_0^T c(\oline{X}_t)\,\rmd t + c_U U_T +  c_D D_T]_x^* = J_x(U, D, \kappa^*, \lambda^*).
		\end{align*}
		Since $(U,D)$ was arbitrarily chosen from $\cA$, we get that
		\begin{align*}
			\gamma
			&\leq \inf_{(U,D)\in\cA}J_x(U, D, \kappa^*, \lambda^*) 
			\leq \sup_{(\kappa,\lambda)\in\Lambda}\inf_{(U,D)\in\cA}J_x(U, D, \kappa, \lambda) = \uline{\gamma}.
		\end{align*}
		
		\item \emph{$\gamma \geq \oline{\gamma}$}: 
		Consider the optimal controls $(U^*,D^*)$ defined in \eqref{eq:optimal_controls}, that is, the $\Pr^{\kappa,\lambda}$-Skorokhod reflectors that keeps the processes within the interval $[\uline{x}, \oline{x}]$ for all $(\kappa, \lambda) \in\Lambda$. Write $\oline{X}_t^* = X_t^{U^*,D^*}$. Reasoning as in \eqref{eq:Dynkin_optimal_ambiguity}, we get that
		\begin{align*}%\label{eq:Dynkin_optimal_control}
			&\Esp[V(\oline{X}_T^*)]_x^{\kappa,\lambda} - V(x) \\ 
			&= 
			\Esp[\int_0^T \lrp{c(\oline{X}_t^*) + (\InfGen^{\kappa,\lambda}V)(\oline{X}_t^*) - \gamma}\,\rmd t]_x^{\kappa,\lambda} 
			+ \gamma T - \Esp[\int_0^T c(\oline{X}_t^*)\,\rmd t + \int_0^T -V'(\oline{X}_t^*)\,\rmd U_t^{*,c} + \int_0^T V'(\oline{X}_t^*)\,\rmd D_t^{*,c}]_x^{\kappa,\lambda} \nonumber \\
			&
			\hspace{0.4cm} - 
			\Esp[\sum_{0<t\leq T}\lrp{\Ind_{\{\Delta U_t^* > 0\}}\int_{\oline{X}_{t-}^* + \Delta J_t}^{\oline{X}_{t-}^* + \Delta J_t + \Delta U_t^*}-V'(z)\,\rmd z 
				+ 
				\Ind_{\{\Delta D_t^* > 0\}}\int_{\oline{X}_{t-}^* + \Delta J_t - \Delta D_t^*}^{\oline{X}_{t-}^* + \Delta J_t}V'(z)\,\rmd z 
			}]_x^{\kappa,\lambda}. \nonumber
		\end{align*}
		Since $\Pr[\oline{X}_t^*\in [\uline{x}, \oline{x}]]_x^{\kappa,\lambda} = 1$, it follows from \eqref{eq:optimal_ambiguity_inner_max} and \eqref{eq:FBP_continuation} that $c(\oline{X}_t^*) + (\InfGen^{\kappa,\lambda}V)(\oline{X}_t^*) - \gamma \leq 0$, which, by also using \eqref{eq:FBP_stopping}-\eqref{eq:FBP_upper_stopping}, results in
		\begin{align*}%\label{eq:gamma_inequality_geq}
			\gamma \geq \frac{1}{T}\lrp{\Esp[V(\oline{X}_T^*)]_x^{\kappa,\lambda} - V(x) + \Esp[\int_0^T c(\oline{X}_t^*)\,\rmd t + \int_0^T c_U \,\rmd U_t^* + \int_0^T c_D\,\rmd D_t^*]_x^{\kappa,\lambda}}.
		\end{align*}
		Taking $T \rightarrow \infty$, noticing that $V(\oline{X}_T^*)$ is bounded as $V$ is continuous and $\oline{X}_T^*$ is bounded $\Pr_x^{\kappa,\lambda}$-a.s., and using the arbitrary selection of $(\kappa,\lambda)\in\Lambda$, we obtain that 
		\begin{align*}
			\gamma 
			&\geq \sup_{(\kappa,\lambda)\in\Lambda}J_x(U^*,D^*,\kappa,\lambda) 
			\geq \inf_{(U,D)\in\cA}\sup_{(\kappa,\lambda)\in\Lambda}J_x(U,D,\kappa,\lambda) = \oline{\gamma}.
		\end{align*}
	\end{itemize}
	
	\subsection{Proof of Proposition \ref{pr:ambg-drift_bang-bang}}

	Since $V'(\uline{x}) = -c_U < 0$ and $V'(\oline{x}) = c_D > 0$ (recall \eqref{eq:FBP_lower_stopping} and \eqref{eq:FBP_upper_stopping}), continuity and increasingness (due to strict convexity) on $(\uline{x}, \oline{x})$  of $V'$ (recall \eqref{eq:FBP_C2}) guarantees the existence of a unique point $x^\kappa\in (\uline{x}, \oline{x})$ such that $V'(x^\kappa) = 0$.
	
	\subsection{Proof of Proposition \ref{pr:ambg-intensity_bang-bang}}
	
	Due to \eqref{eq:FBP_C2} we get that $\Psi V \in C^1(\R)$ (actually $\Psi V \in C^2(\R)$, but $C^1$-regularity is enough for our matter in this proposition). Moreover, since $V'$ is bounded due to \eqref{eq:FBP_lower_stopping}-\eqref{eq:FBP_C2} and $\Esp[|Y|]$ is bounded, we can differentiate inside the integral sign to get
	$
	(\Psi V)'(x) = \int_\R \lrp{V'(x + y) - V'(x)}\,\rmd F(y)
	$.
	Then, the increasing property of $V'$ in $(\uline{x}, \oline{x})$ (coming from its strict convexity) and the negative support of $F$ guarantees that  $(\Psi V)'(x) < 0$ for all $x \in (\uline{x}, \infty)$. Actually, adding \eqref{eq:FBP_lower_stopping} and \eqref{eq:FBP_upper_stopping}, we get that $(\Psi V)'(x) \leq 0$ everywhere.
	
	Using \eqref{eq:FBP_lower_stopping} one obtains that $V(x) = V(\uline{x}) - c_U(x - \uline{x})$ for all $x < \uline{x}$, meaning that $(\Psi V)(x) = \Esp[V(x + Y) - V(x)] = -c_U\Esp[Y] > 0$ for all $x \leq \uline{x}$. 
	On the other hand, due to \eqref{eq:FBP_upper_stopping}, we get that, for all $y \leq 0$,
	\begin{align*}
		\lim_{x\to\infty} V(x + y) - V(x) = c_Dy.
	\end{align*}
	Hence, since $|V(x + y) - V(x)| \leq (c_U + c_D)|y|$ (which follows by the boundedness of $V'$ derived from \eqref{eq:FBP_lower_stopping}-\eqref{eq:FBP_C2}), the dominated convergence theorem ensures that 
	\begin{align*}
		\lim_{x\to\infty}(\Psi V)(x) 
		&= \lim_{x\to\infty} \Esp[V(x + Y) - V(x)] 
		= c_D\Esp[Y] < 0.
	\end{align*}
	Hence, due to the continuity and the decreasing property of $\Psi V$, there exists a unique $x^\lambda > \uline{x}$ such that $\Psi V$ changes sign around $x^\lambda$. The same continuity and monotonicity argument guarantees that $x^\lambda < \oline{x}$ if and only if condition \eqref{eq:xla_in_inaction_suf_cond} is met. Finally, since $V'(x) < 0$ for all $x < x^\kappa$, then $(\Psi V)(x) > 0$ on the same interval, meaning that $x^\lambda > x^\kappa$.  
	
	\subsection{Proof of Proposition \ref{pr:suf_condition_regime1}}

	Recalling \eqref{eq:inf_gen}, the free-boundary problem \eqref{eq:FBP_continuation}-\eqref{eq:FBP_C2}, and the fact that $\kappa^*(\oline{x}) = \delta$, we get the following expression for $(\Psi V)(\oline{x})$ after taking $x\uparrow \oline{x}$ in $(\InfGen V)(x) + c(x) - \gamma$:
	\begin{align*}
		(\Psi V)(\oline{x}) = \lrp{\gamma - c(\oline{x}) - c_D(b(\oline{x}) + \sigma\delta - r\Esp[Y])}/\lambda^*(\oline{x}).
	\end{align*}
	Since $\lambda^*(\oline{x}) \geq 0$, the sign of $(\Psi V)(\oline{x})$ coincides with that of the right-hand side numerator. Therefore, 
	\begin{align*}
		\gamma < \uline{c} + c_D(\uline{b} + \sigma\delta - r\Esp[Y]) 
		\Longrightarrow (\Psi V)(\oline{x}) < 0.
	\end{align*}
	The proof concludes after using the ergodic-value upper bound \eqref{eq:ergodic-value_upper_bound}.
	
	\subsection{Proof of Proposition \ref{pr:ergodic-value_upper_bound_parabolic-cost}}
	Let $C(x_1,x_2)=\max_{x\in[x_1,x_2]} x^2$, and rewrite 
	\begin{align*}
		\Gamma(x_1,x_2) = C(x_1,x_2) + (c_U + c_D)\lrp{K_1 + \frac{K_2}{x_2-x_1}},
	\end{align*} 
	for $K_1 \defeq |b| + \sigma\delta + \varepsilon r\Esp[|Y|]$ and $K_2 \defeq \sigma^2 + r(1+\varepsilon)\Esp[Y^2]$. Define the length $L\defeq x_2-x_1$. Notice that the second term defining $\Gamma(x_1,x_2)$ depends only on $L$.  Since $C(x_1,x_2)\ge C(-L/2,L/2)=L^2/4$,
	%On the other hand, the smallest possible value of $C(x_1,x_2)$ is achieved by centering the interval $[x_1,x_2]$, that is, by setting $[x_1,x_2] = [-L/2,L/2]$:
	%\begin{align*}
	%C(x_1,x_2) \defeq \max_{x\in[x_1,x_2]} x^2 = \max\{x_1^2,x_2^2\} \geq L^2/4.
	%\end{align*}
	we get that $\Gamma(-L/2,L/2) = \alpha K_1 + \phi(L)$,
	for $\alpha \defeq (c_U + c_D)$ and
	\begin{align*}
		\phi(L) \defeq L^2/4 + \alpha K_2/L.
	\end{align*}
	Setting $\phi'(L) = 0$ we get the minimizer $L^* = (2\alpha K_2)^{1/3}$, which yields the global minimum for $\Gamma(-L/2,L/2)$
	\begin{align*}
		\min_{x_1<x_2} \Gamma(x_1,x_2) = \min_{L>0}\Gamma(-L/2,L/2) = \alpha K_1 + \min_{L>0} \phi(L) = \alpha K_1 + \phi(L^*) 
		&= \alpha K_1 + 3\lrp{\alpha K_2/4}^{2/3}.
	\end{align*}
	The proof concludes after Proposition \ref{pr:ergodic-value_upper_bound}.
	
	\subsection{Proof of Lemma \ref{lm:Regime-1_parabolic_cost}}

	After Proposition \ref{pr:suf_condition_regime1}, and using that $c(x) = x^2 \geq 0$ as well as the constancy of the drift $b$, we know that Regime 1 holds if $\min_{x_1<x_2} \Gamma(x_1,x_2) < c_D(\delta\sigma + b - r\Esp[Y])$, which Proposition \ref{pr:ergodic-value_upper_bound_parabolic-cost} tells is exactly inequality \eqref{eq:suf_cond_regime1_parabolic-cost} for $\Esp[Y] = -1/\mu$ and $\Esp[Y^2] = 2/\mu^2$.
	
	To prove that the simpler condition \eqref{eq:suf_to_suf_regime1_condition} suffices for \eqref{eq:suf_cond_regime1_parabolic-cost}, set $y \defeq (r/\mu)^{1/3}$. Since $|b| = b$ under $b \geq 0$ from \eqref{eq:suf_to_suf_regime1_condition}, \eqref{eq:suf_cond_regime1_parabolic-cost} is equivalent to 
	\begin{align*}
		\hat{K}_1 y^3 \geq c_U(b + \delta\sigma) + 3\lrp{\frac{c_U + c_D}{4}(\sigma^2 + \frac{2(1+\varepsilon)}{\mu}y^3)}^{2/3},
	\end{align*}
	which, by using the concavity inequality $(u + v)^{2/3} \leq u^{2/3} + v^{2/3}$ for $u = \sigma^2$ and $v = 2(1+\varepsilon)y^3/\mu$, is implied by the inequality
	\begin{align}\label{eq:suf_regime1_auxiliary}
		\hat{K}_1y^3 &\geq \hat{K}_2 + \hat{K}_3y^2,
	\end{align}
	which holds true by taking $y > \max\{2\hat{K}_3/\hat{K}_1, (2\hat{K}_2/\hat{K}_1)^{1/3}\}$, since $\hat{K}_1 > 0$ under $\varepsilon < c_D/(c_D+c_U)$ from \eqref{eq:suf_to_suf_regime1_condition}.    
	
	\subsection{Proof of Lemma \ref{lm:uniquenes_solution_linear_systems}}
	
	To shorten notation, we prove uniqueness for the systems written in terms of
	the original coefficients $\mathbf{c}_{i}^{\pm}$.
	Let $\Theta_i=\alpha_{1,-}^{(i)}\alpha_{2,+}^{(i)}-\alpha_{1,+}^{(i)}\alpha_{2,-}^{(i)}$.
	
	\begin{itemize}[label=$\triangleright$]
		\item \emph{$\Theta_1$}:
		After some basic manipulations, and using
		$\frac{\mu\sigma^2}{2a_{1,1}} = \frac{\mu}{\rho_1^+\rho_1^-}$,
		the determinant $\Theta_1$ can be written as
		\details{
			\begin{align}
				\Theta_1
				&=
				e^{-(\rho_1^- + \rho_1^+)\uline{x}}\lrp{h(\rho_1^+) - h(\rho_1^-)},
			\end{align}
		}{
			\begin{align}
				\Theta_1
				&=
				e^{-(\rho_1^- + \rho_1^+)\uline{x}}\lrp{h(\rho_1^+) - h(\rho_1^-)}, \label{eq:discriminant_1_non-zero}
			\end{align}
		}    
		where $d_1 = x^\kappa-\uline{x} > 0$ and
		$h(t) = (1 - \frac{\mu}{t})(e^{-t d_1} - 1)$.
		Since $h$ is strictly decreasing, and $\rho_1^+ > \rho_1^-$,
		it follows from \eqref{eq:discriminant_1_non-zero} that $\Theta_1<0$.
		Hence \eqref{eq:linear_system_1_stable} has a unique solution.
		
		\item \emph{$\Theta_2$}:
		Similarly, $\Theta_2$ takes the form
		\details{
			\begin{align*}
				\Theta_2 
				= 
				\mu e^{-\rho_2^- x^\kappa}e^{-\rho_2^+ x^\lambda}
				\lrp{h_2(\rho_2^+) - e^{(\rho_2^+ - \rho_2^-)(x^\lambda-x^\kappa)}h_2(\rho_2^-)},
			\end{align*}
		}{
			\begin{align*}
				\Theta_2 
				= 
				\mu e^{-\rho_2^- x^\kappa}e^{-\rho_2^+ x^\lambda}
				\lrp{h_2(\rho_2^+) - e^{(\rho_2^+ - \rho_2^-)(x^\lambda-x^\kappa)}h_2(\rho_2^-)}.
			\end{align*}
		}
		where $h_2(t) = (1 - e^{-(\mu - t)d_2})/(\mu - t)$ and
		$d_2 = x^\lambda - x^\kappa > 0$.
		Using the representation
		$h_2(t) = \int_0^{d_2} e^{-(\mu - t)s}\,ds$,
		one obtains
		\details{
			\begin{align*}
				h_2(\rho_2^+) - e^{(\rho_2^+ - \rho_2^-)(x^\lambda-x^\kappa)}h_2(\rho_2^-)
				&= 
				\int_0^{d_2}
				e^{-(\mu - \rho_2^+)s}\lrp{1 - e^{(\rho_2^+ - \rho_2^-)(d_2 - s)}}\,\rmd s 
				< 0.
			\end{align*}
		}{
			\begin{align*}
				h_2(\rho_2^+) - e^{(\rho_2^+ - \rho_2^-)(x^\lambda-x^\kappa)}h_2(\rho_2^-)
				&=
				\int_0^{d_2}
				e^{-(\mu - \rho_2^+)s}\lrp{1 - e^{(\rho_2^+ - \rho_2^-)(d_2 - s)}}\,\rmd s 
				< 0.
			\end{align*}
		}
		Therefore $\Theta_2<0$, so \eqref{eq:linear_system_2_stable}
		also has a unique solution.
		
		\item \emph{$\Theta_3$}:
		Finally,
		\begin{align*}
			\Theta_3 &= 
			\alpha_{1,-}^{(3)}\alpha_{2,+}^{(3)}-\alpha_{1,+}^{(3)}\alpha_{2,-}^{(3)} 
			= e^{-\rho_3^-\oline{x}-\rho_3^+x^\lambda}\left(1 - e^{-(\rho_3^+-\rho_3^-)(\oline{x}-x^\lambda)}\right) > 0.
		\end{align*}
		Hence \eqref{eq:linear_system_3_stable} has a unique solution as well.
	\end{itemize}
	
	\bibliographystyle{apalike}
	\bibliography{Azze-etal_2026.bib}

\end{document}